\documentclass[areasetadvanced, sfdefaults=false]{scrartcl}

\usepackage{amsmath}
\usepackage{amssymb}
\usepackage{amsthm, MnSymbol}
\usepackage{mathtools}
\usepackage{xcolor}
\usepackage[mathscr]{euscript}
\usepackage{tikz-cd}
\usepackage{enumitem}
\usepackage{todonotes}
\usepackage{hyperref}
\usepackage{afterpage}
\usepackage{aliascnt}
\usepackage{quiver}
\usepackage{inputenc}[utf8]
\usetikzlibrary{trees}
\usetikzlibrary{calc}
\usetikzlibrary{shapes.geometric}

\theoremstyle{theorem}
\newtheorem{theorem}{Theorem}[section]

\newtheorem{proposition}[theorem]{Proposition}

\newtheorem{corollary}[theorem]{Corollary}

\newtheorem{lemma}[theorem]{Lemma}

\theoremstyle{definition}
\newtheorem{definition}[theorem]{Definition}

\newtheorem{construction}[theorem]{Construction}

\newtheorem{example}[theorem]{Example}

\newtheorem{remark}[theorem]{Remark}

\newtheorem{notation}[theorem]{Notation}

\newtheorem{conjecture}[theorem]{Conjecture}

\newcommand{\map}{\mathrm{map}}
\newcommand{\Map}{\mathrm{Map}}

\newcommand{\Lan}{\mathrm{Lan}}
\newcommand{\Ran}{\mathrm{Ran}}

\newcommand{\N}{\mathbb{N}}
\newcommand{\Z}{\mathbb{Z}}
\newcommand{\Q}{\mathbb{Q}}
\newcommand{\R}{\mathbb{R}}
\newcommand{\C}{\mathbb{C}}

\newcommand{\id}{\mathrm{id}}

\newcommand{\Hom}{\mathrm{Hom}}
\newcommand{\End}{\mathrm{End}}
\newcommand{\Fun}{\mathrm{Fun}}
\newcommand{\simp}[1]{\mathrm{simp}(#1)} 
\newcommand{\op}{^{op}}

\newcommand{\perf}[1]{\mathrm{Perf}(#1)}

\newcommand{\colim}{\mathrm{Colim }}
\renewcommand{\lim}{\mathrm{Lim }}

\renewcommand{\dim}{\mathrm{dim }}
\newcommand{\ind}[2]{\mathrm{ind}_{#1}^{#2}}
\newcommand{\res}[2]{\mathrm{res}_{#1}^{#2}}
\newcommand{\rep}[1]{\mathrm{Rep}(#1)}
\newcommand{\pr}[1]{\mathrm{pr}_{#1}}
\newcommand{\sk}[1]{\mathrm{sk}_{#1}}
\newcommand{\K}[1]{\mathrm{K}(#1)}
\newcommand{\MV}{\mathcal{MV}}
\newcommand{\An}{\mathrm{An}}

\title{$K$-theory of rank one reductive $p$-adic groups and Bernstein blocks}
\author{Maximilian Tönies}
\date{}
\begin{document}
\maketitle

\begin{abstract}
   \noindent \textbf{Abstract}: We prove a colimit formula for the $K$-theory spectra of reductive $p$-adic groups of rank one with regular coefficients in terms of the $K$-theory of certain compact open subgroups.
   Furthermore, in the complex case, we show, using the construction of types provided by Roche \cite{RO98},  that this result can be improved to obtain a formula for the $K$-theory spectrum of every principal series Bernstein block if the group is split. 
\end{abstract}

\tableofcontents

\section{Introduction}
\subsection{Smooth representations}
Let $\mathbb{G}$ be a reductive algebraic group over a non-archimedean local field $F$ and let $G:=\mathbb{G}(F)$ be the $F$-points of $\mathbb{G}$.
We will always mean the $F$-points when we mention algebraic groups.
The topology of $F$ induces a topology on $G$ which turns $G$ into a totally disconnected, Hausdorff topological group.
To these topological groups, we can associate a certain kind of representations that take the topology of $G$ into account.
For a ring $R$, a $G$-representation on an $R$-module $V$ is said to be \cal{smooth} if all isotropy groups are open. 
This is equivalent to the $G$-action on $V$ being continuous with respect to the discrete topology on $V$.
In this paper, we will always assume that $\Q \subseteq R$ and $R$ is regular in the sense that $R$ is Noetherian and regular coherent.
We denote by $\rep{G}$ the Abelian category of smooth $R$-representations together with equivariant maps and by $D(G)$ the associated derived stable $\infty$-category.
Throughout this paper we will freely use the language of $\infty$-categories as developed in \cite{HTT} and \cite{HA}. 
In particular, we will always identify ordinary categories with their nerve and treat them as $\infty$-categories.
For any open subgroup $H \subseteq G$ there is an adjoint pair of functors 
\[\begin{tikzcd}
	{D(H)} && {D(G)}
	\arrow[""{name=0, anchor=center, inner sep=0}, "{\ind{H}{G}}", shift left=2, curve={height=-6pt}, from=1-1, to=1-3]
	\arrow[""{name=1, anchor=center, inner sep=0}, "{\res{G}{H}}", shift left=2, curve={height=-6pt}, from=1-3, to=1-1]
	\arrow["\dashv"{anchor=center, rotate=-90}, draw=none, from=0, to=1]
\end{tikzcd}\]
given by compact induction and restriction along the inclusion $H \subseteq G$.
More generally this adjunction exists for any open homomorphism $H \to G$.
For a reductive algebraic group $G$ we define the (non-connective) $K$-theory spectrum as 
\begin{equation*}
   K(G):= K(\perf{G}) 
\end{equation*}
where $\perf{G}:=D(G)^\omega$ is the full subcategory of compact objects.
Whenever we write $K$-theory we will always mean non-connective $K$-theory as a localizing invariant ${Cat}_{\mathrm{ex}}^{\mathrm{perf}} \to \mathrm{Sp}$ in the sense of \cite{BGT13}.
In general, this is very difficult to compute as the representation theory of $G$ can be quite complicated.
However, if $G$ is compact, and we additionally assume that $R$ is semi-simple, then every $G$-representation over $R$ decomposes into a direct sum of irreducible representations. 
In combination with Schur's lemma, this induces a decomposition of the derived category $D(G)$ into a product of $D(A_\pi)$ for division $R$-algebras $A_{\pi}$ indexed over the set of (isomorphism classes of) irreducible representations of $G$. 
If $R$ is an algebraically closed field, then $A_\pi = R$ for all $\pi \in \mathrm{Irr}(G)$.
For general $G$ one might ask whether it is possible to reduce the computation of $K(G)$ to the compact case. 
The claim that this is true is the content of the Farrell--Jones conjecture for reductive algebraic groups which we will briefly sketch in the following.

Let $\mathrm{Orb}^\infty(G)$ be the category of transitive $G$-sets with open isotropy subgroups and equivariant maps.
Note that an equivariant map $G/H \to G/K$ can be (non-uniquely) represented by conjugation with a group element $g \in G$ such that $H^g =gHg^{-1}\subseteq K$.
To a transitive $G$-set $G/H$ we assign the $K$-theory spectrum $K(H)$.
This can be made functorial in $\mathrm{Orb}^{\infty} (G)$ such that for any equivariant map $G/H \to G/K$ represented by an element $g \in G$ the associated map is induced by compact induction along $(-)^g: H \to K$.
Hence, we obtain a functor 
\[\begin{tikzcd}
	{\mathrm{Orb}^{\infty(G)}} && {\mathrm{Sp}.}
	\arrow["K", from=1-1, to=1-3]
\end{tikzcd}\]
We can left Kan extend this along the Yoneda embedding $\mathrm{Orb}^{\infty}(G) \to \mathcal{P}(\mathrm{Orb}^{\infty}(G))$
\[\begin{tikzcd}
	{\mathrm{Orb}^{\infty}(G)} && {\mathrm{Sp}} \\
	\\
	{\mathcal{P}(\mathrm{Orb}^{\infty}(G))}
	\arrow["K", from=1-1, to=1-3]
	\arrow[from=1-1, to=3-1]
	\arrow[""', from=3-1, to=1-3]
\end{tikzcd}\]
to obtain a colimit preserving functor out of $\mathcal{P}(\mathrm{Orb}^{\infty}(G))= \Fun((\mathrm{Orb}^{\infty}(G))^{op}, \An)$ which we will again denote by $K$. 
Here $\mathrm{An}$ denotes the category of homotopy types of spaces, modeled for example via Kan complexes.
Note that $\mathcal{P}(\mathrm{Orb}^{\infty}(G))$ is equivalent to the category of smooth $G$-anima,
that is $G$-spaces with open isotropy groups.
In $\mathcal{P}(\mathrm{Orb}^{\infty}(G))$ we have the terminal object $*$, which is precisely the Yoneda embedding of $G/G$, hence any other $G$-space comes with a unique equivariant map to $*$. 
This induces for any $X \in \mathcal{P}(\mathrm{Orb}^{\infty}(G))$ a canonical map 
\[\begin{tikzcd}
	{K(X)} && {K(G)}
	\arrow[from=1-1, to=1-3]
\end{tikzcd}\]
where the left-hand side is, by the colimit formula for left Kan extensions, equivalent to 
$$
K(X) \simeq \colim_{G/H \to X}K(H).
$$
Let $\mathrm{Orb^{\infty}_{\mathcal{C}op}} (G)\subseteq \mathrm{Orb}^{\infty}(G)$ be the full subcategory of transitive $G$-sets with compact open isotropy.
There exists $E_{\mathcal{C}op} \in \mathcal{P}(\mathrm{Orb}^{\infty} (G))$ such that $E_{\mathcal{C}op}(G/H)$ is contractible if $G/H \in \mathrm{Orb}^{\infty}_{\mathcal{C}op}(G)$ and empty otherwise. 
To any reductive algebraic group, we can associate its extended Bruhat--Tits building.
It can be equipped with a metric and comes with an action of $G$ by isometries.
Furthermore, it can be given the structure of a polysimplical complex that is preserved by the action \cite[Section 2.3]{TI79}.
This building is a model for $E_{\mathcal{C}op}$ \cite[Theorem 6.4.13]{LU05}.
The canonical map $E_{\mathcal{C}op} \to *$ induces 
\begin{equation}    
\label{FJconj}
\begin{tikzcd}
	{K(E_{\mathcal{C}op}) \simeq \colim_{G/H \in \mathrm{Orb}^{\infty}_{\mathcal{C}op}(G)}K(H)} && {K(G).}
	\arrow[from=1-1, to=1-3]
\end{tikzcd}
\end{equation}
In particular, the colimit on the left-hand side only contains the $K$-theory spectra of compact groups.
For a more detailed discussion see for instance \cite{DL98} and \cite[Section 4.2.4]{LR05}.
Now we are in the position to formulate the Farrell--Jones conjecture.

\begin{conjecture}[Farrell--Jones]
Let $G$ be a totally disconnected group.
Then the map \eqref{FJconj} is an equivalence of spectra.
\end{conjecture}
\begin{remark}
    Farrell and Jones originally formulated the conjecture in \cite{FJ93} for discrete groups $G$ and $\Z$-coefficients, and the assembly map was constructed with respect to all virtual cyclic subgroups. 
    This conjecture has been proven for a broad class of groups, including hyperbolic groups \cite{BLR08}, $\mathrm{CAT}(0)$-groups \cite{BL12} and mapping class groups \cite{BB19}.
    The variant for reductive $p$-adic groups with respect to compact open subgroups that we treat in this paper was formulated by Lück and Bartels in \cite[Section 1.H]{bartels2023algebraic} and proven for uniformly regular rings with $\Q \subseteq R$ in \cite{bartels2023algebraic}.
\end{remark}
In this paper, we will only treat the case where $G$ is a reductive algebraic group and $E_{\mathcal{C}op}$ is one dimensional.
This class contains split groups like $\mathrm{SL}_2(F)$ and $\mathrm{PGL}_2(F)$ but also non-split groups like $SU_n(D,h)$ for some $n$, where $D$ is a division algebra over $F$ with involution and $h$ a certain non-degenerate form on $D^n$.
For a list of examples where $G$ is absolutely simple of rank one see \cite{carbone2001classification}. 
In this situation, the assertion that the Farrell--Jones assembly map is an equivalence reduces to the following claim.
\begin{theorem}[see \ref{Kseq}]
    \label{theoremA}
    Let $R$ be a regular ring with $\Q \subseteq R$ and 
    let $G$ be a reductive, algebraic group that acts smoothly on a tree $X$ such that the action is transitive and without inversion on edges and has compact stabilizers.
    Then, for any edge $e \in X$ the following holds. 
\begin{enumerate}
        \item If $X_0$ has one $G$-orbit, then 
\[\begin{tikzcd}
	{K(G_e)} && {K(G_{\partial_0e})} && {K(G)}
	\arrow["{\ind{}{}}", from=1-3, to=1-5]
	\arrow["{\ind{}{}(\partial_0) - \ind{}{}(\partial_1)^g}", from=1-1, to=1-3]
\end{tikzcd}\]
            with maps induced by induction and $g \in G$ such that $g\partial_1e = \partial_0e$, is a cofibre sequence. 
        \item If $X_0$ has two $G$-orbits, then 
\[\begin{tikzcd}
	{K(G_e)} && {K(G_{\partial_0e})} \\
	\\
	{K(G_{\partial_1e})} && {K(G)}
	\arrow["{\ind{}{}}", from=1-3, to=3-3]
	\arrow["{\ind{}{}}", from=1-1, to=1-3]
	\arrow["{\ind{}{}}"', from=1-1, to=3-1]
	\arrow["{\ind{}{}}"', from=3-1, to=3-3]
\end{tikzcd}\]
   with maps induced by induction, is a pushout square.
    \end{enumerate}   
 In particular $K(G)$ is connective.
\end{theorem}
\begin{remark}
    The subdivision is necessary if and only if the action of $G$ inverts edges, that is if the pointwise stabilizer of an edge does not agree with the elements that leave the edge invariant as a set. 
    In that case, we have to baricentrically subdivide $X$ once and the subdivision inherits a $G$-action which does not invert any edges.
\end{remark}
The methods we use to prove this theorem are closely related to the ideas of Waldhausen in \cite{WH78a} and \cite{WH78b} and to work by Winges in \cite{winges2014filtering}.
This reproves the rank one part, and slightly generalizes the admissible coefficients, of a recent result by Bartels and Lück \cite{bartels2023algebraic} who show that the Farrell--Jones conjecture holds more generally for reductive algebraic groups without any assumption on the rank.

\begin{theorem}[Bartels--Lück \cite{bartels2023algebraic}]
The Farrell--Jones conjecture holds for all reductive algebraic groups and uniformly regular rings $R$ with $\Q \subseteq R$. 
\end{theorem}

\subsection{Bernstein blocks and types}
In \cite{BE84} and \cite{bernsteindraft} Bernstein proves a product decomposition theorem for the category $\rep{G}$ of smooth $G$-representations.
This induces a decomposition of the derived category $D(G)$ and with that a decomposition of $K(G)$ into a direct sum of the $K$-theory of the components.
As the second part of this paper deals with the $K$-theory of these components we will briefly recall this decomposition as well as some of its properties.
For a more detailed explanation see \cite{bernsteindraft}.

Let $L$ be a Levi subgroup of $G$. 
Then we have a pair of adjoint functors 
\[\begin{tikzcd}
	{\rep{G}} && {\rep{L}}
	\arrow[""{name=0, anchor=center, inner sep=0}, "{r_G^L}", shift left=2, from=1-1, to=1-3, curve={height=-6pt}]
	\arrow[""{name=1, anchor=center, inner sep=0}, "{i_L^G}", shift left=2, from=1-3, to=1-1, curve={height=-6pt}]
	\arrow["\dashv"{anchor=center, rotate=-90}, draw=none, from=0, to=1]
\end{tikzcd}\]
given by the so-called normalized parabolic induction $i^G_L$ and the normalized Jacquet restriction $r_G^L$.
\begin{definition}
    A smooth $G$-representation $\pi$ is called quasi-cuspidal if $r_G^L\pi =0$ for every Levi subgroup $L \neq G$.
    It is called cuspidal if it is in addition finitely generated.
\end{definition}
Let $\mathfrak{X}(L)$ be the group of unramified characters of $L$, that is, those characters that factor over $L/L^0$ where $L^0$ is the subgroup generated by the union of all maximal compact open subgroup of $L$.
Let $C(G)$ be the set of all pairs $(L, \rho)$ where $L$ is a Levi subgroup of $G$ and $\rho$ is an irreducible cuspidal $L$-representation.
There is an equivalence relation on $C(G)$ by declaring $(L_1, \rho_1) \sim (L_2, \rho_2)$ if there exists $g \in G$ and $\chi \in \mathfrak{X}(L_2)$ such that 
\begin{equation}
    \begin{split}
    L_1^g &= L_2 \\
    \rho_1^g &= \chi \rho_2,
    \end{split}
\end{equation}
and we denote by $\Omega(G)$ the set of equivalence classes $[L, \rho]$ under this relation.
To any such equivalence class, we can associate the full subcategory $\rep{G}_{[L,\rho]}$ that is spanned by all smooth $G$-representations $\pi$ such that every irreducible subquotient of $\pi$ appears as a subquotient of $i_{L'}^G(\rho')$ for some $(L',\rho') \in [L, \rho]$. 
\begin{theorem}[Bernstein]
The category $\rep{G}$ splits as a product 
$$
\rep{G} = \prod_{[L,\rho] \in \Omega(G)} \rep{G}_{[L,\rho]}
$$
and each factor is indecomposable. The subcategories $\rep{G}_{[L, \rho]}$ are called Bernstein blocks. 
\end{theorem}
This decomposition immediately implies that  
$$
K(G) = \bigoplus_{[L, \rho] \in \Omega(G)}K(G, \rho)
$$
where $K(G, \rho)$ is the $K$-theory spectrum of the derived category of $\rep{G}_{[L, \rho]}$.
Since the Farrell--Jones conjecture for $G$ is true there is a decomposition of $K(E_{\mathcal{C}op})$ corresponding to the factors $K(G, \rho)$.
Now one might also hope that there is a description of $K(G, \rho)$ in terms of a colimit with respect to the compact open subgroups, similar to the global case.
Further analysis of the Bernstein blocks $\rep{G}_{[L,\rho]}$ shows that $i_L^G \rho$ is always a compact generator and hence $\rep{G}_{[L, \rho]} \simeq \mathrm{Mod}(\End(i_L^G\rho))$.
While this is a very pleasant structural result about Bernstein blocks, there is no direct relation between parabolic induction and compact induction.
Hence, it seems unclear how the $K$-theory of Bernstein blocks can be computed as the colimit of compact inductions. 
However, it turns out that there are sometimes compact generators that can be obtained by compact induction.
This leads to the theory of types by Bushnell and Kutzko \cite{BK98}. 
\begin{definition}
    A pair $(K, \pi)$ where $\pi$ is an irreducible representation of a compact open subgroup $K\subset G$ is called a type for the Bernstein block $\rep{G}_{[L, \rho]}$ if $\ind{K}{G} \pi \in \rep{G}_{[L,\rho]}$ and $\ind{K}{G} \pi$ is a compact generator of $\rep{G}_{[L,\rho]}$.
\end{definition}
We will be primarily interested in the case where $L$ is the maximal torus $T$ as this is the only non-trivial Levi subgroup in the rank $1$ case if $G$ is split and connected.
Here Roche showed in \cite{RO98} that there always exists a type $(J, \rho)$ for $\rep{G}_{[T,\rho]}$ under a restriction on the residual characteristic of the field $F$. 
For these types, Roche also computes the associated Hecke algebra $\mathcal{H}(G, \rho)$ which is given by the endomorphism ring $\End(\ind{J}{G}\rho )$.
In general, the existence of types is not clear, however, it is conjectured that every Bernstein block admits a type. 
This is known by the work of Bushnell and Kutzko \cite{BK93, BK94, BK98, BK99} for $\mathrm{GL}_n$ and in combination with work by Roche and Goldberg \cite{GR02} for $\mathrm{SL}_n$.
For recent results on this conjecture see Fintzen \cite{FI21}.

In the case of principal series blocks with associated type $(J, \rho)$ by \cite{RO98} we can associate to every compact open subgroup $K$ a full subcategory $D(K, \rho) \subseteq D(K)$ (see definition \ref{defrel}) and we define $K(K, \rho):=K(D(K, \rho)^\omega)$. 
This system of subcategories is closed under compact induction. 
Therefore, the Farrell--Jones assembly map can be restricted  to obtain a relative map 
\[\begin{tikzcd}
	{\colim_{H\in \mathrm{Orb}^\infty_{\mathcal{C}op}(G)}K(H, \rho)} && {K(G, \rho),}
	\arrow[from=1-1, to=1-3]
\end{tikzcd}\]
and we conjecture that this map is an equivalence of spectra.

Similar to the absolute case this simplifies to the assertion that the following diagram is a cofibre sequence resp.\ a pushout square if $G$ has rank one.
\begin{theorem}[see Theorem \ref{KBernstein}]
    \label{theoremB}
            Let $G$ be a connected split reductive algebraic group of rank one with associated (subdivided) Bruhat--Tits tree $X$. Let $\rho_\chi$ be a principal series type associated to a smooth character $\chi: T^0 \to \C^\times$ and $e$ an edge of $X$.
        \begin{enumerate}
            \item If $X_0$ has one $G$-orbit, then 
            \[\begin{tikzcd}
	            {K(G_e, \rho_\chi)} && {K(G_{\partial_0e}, \rho_\chi)} && {K(G, \rho_\chi)}
	            \arrow["{\ind{}{}}", from=1-3, to=1-5]
	            \arrow["{\ind{}{}(\partial_0) - \ind{}{}(\partial_1)^g}", from=1-1, to=1-3]
            \end{tikzcd}\]
            with maps induced by induction and $g \in G$ such that $g\partial_1e = \partial_0e$, is a cofibre sequence.
            \item If $X_0$ has two $G$-orbits, then 
            \[\begin{tikzcd}
            	{K(G_e, \rho_\chi)} && {K(G_{\partial_0e}, \rho_\chi)} \\
            	\\
	            {K(G_{\partial_1e}, \rho_\chi)} && {K(G, \rho_\chi)}
            	\arrow["{\ind{}{}}",from=1-1, to=1-3]
            	\arrow["{\ind{}{}}",from=1-1, to=3-1]
            	\arrow["{\ind{}{}}",from=3-1, to=3-3]
            	\arrow["{\ind{}{}}",from=1-3, to=3-3]
            \end{tikzcd}\]
            with maps induced by induction, is a pushout square.
        \end{enumerate}
        In particular, $K(G, \rho_\chi)$ is connective.
\end{theorem}
As a consequence, we obtain the following corollary about resolutions of finitely generated representations.
This generalizes a result by Schneider and Stuhler in \cite{SS97} for admissible representations to finitely generated representations and Bernstein blocks in the case of rank one groups.
\begin{corollary}
    Let $G$ be a reductive, algebraic group of rank $1$.
    Then every finitely generated smooth complex $G$-representation $V$ admits a resolution by a finite complex that is degree-wise a finite direct sum of compactly induced representations.
    Furthermore, if $G$ is split and connected, and $V \in \rep{G, \rho_\chi}$ lies in the Bernstein block for some type $\rho_\chi$, then the resolution can be chosen such that all representations are finite sums of direct summands of $\ind{J_\chi}{G_x}\rho_{\chi^\omega}$ for $\omega$ in the Weyl group $W$ and $x\in X$. 
\end{corollary}
\begin{proof}
First let $P$ be a finitely generated, projective $G$-representation. Then by the surjectivity of the assembly map on $\pi_0$ by Theorem \ref{theoremA} there are $G$-representations $Q$ and $Q'$ such that $[P]=[Q]-[Q']$ in $K_0(G)$ and $Q$,$Q'$ are direct sums of compactly induced representations.    
Hence, there exists some finite sum of compactly induced representations $F$ such that 
$P \oplus Q' \oplus F \cong Q \oplus F$, i.e. every finitely generated projective $G$-representation can be stabilized by compactly induced representations to a compactly induced representation.
Furthermore, if $P \in \rep{G, \rho_\chi}$, then $Q$ and $Q'$ can be chosen such that they are a sum of direct summands of $\ind{J_\chi}{G} \rho_{\chi^\omega}$ for $\omega \in W$ by \ref{theoremB}.
Additionally, $F$ can be chosen to be a finite direct sum of $\ind{J}{G}\rho_\chi$ since the latter is a compact generator of $\rep{G, \rho_\chi}$.

By \cite[Chapter IV.4]{bernsteindraft} the category of smooth complex representations is Noetherian and has finite cohomological dimension. This implies that any finitely generated smooth $G$-representation admits a finite resolution
\[\begin{tikzcd}
	\ldots & {P_2} & {P_1} & {P_0} & V
	\arrow[from=1-1, to=1-2]
	\arrow[from=1-2, to=1-3]
	\arrow[from=1-3, to=1-4]
	\arrow[from=1-4, to=1-5]
\end{tikzcd}\]
by finitely generated projective $G$-representations $P_i$. 
By the previous discussion, we can add chain complexes of the form 
\[\begin{tikzcd}
	\ldots & 0 & F & F & 0 & \ldots
	\arrow[from=1-1, to=1-2]
	\arrow[from=1-2, to=1-3]
	\arrow["\id", from=1-3, to=1-4]
	\arrow[from=1-4, to=1-5]
	\arrow[from=1-5, to=1-6]
\end{tikzcd}\]
for $F$ a finite direct sum of compactly induced representations in the appropriate degrees to obtain the desired resolution.
Again, if $V \in \rep{G, \rho_\chi}$ we can choose $P_i$ to be finitely generated projective in $\rep{G, \rho_\chi}$ and proceed analogously.
\end{proof}
The methods used to prove Theorem \ref{theoremB} are very similar to the absolute case in Theorem \ref{theoremA}. 
The main difficulty lies in the proof that a certain functor associated to the type is a left Bousfield localization. This followed from purely geometric arguments in the absolute case. 
The previous theorem answers a question asked by Bartels and Lück in \cite[Section 1.N.]{bartels2023algebraic} in the case of split rank one groups about the existence of a relative assembly map for Bernstein blocks and the corresponding relative Farrell--Jones conjecture. 
\subsection{Structure of the paper}

This paper is divided into three main sections. 
In the first section, we start by introducing the definitions we will need later. 
The most important definition is that of the category of Mayer--Vietoris resolutions $\MV(X)$ that we can assign to a semi-simplicial set $X$ with an action of a group $G$. 
These are certain kinds of functors that encode a family of representations parametrized over the simplices of $X$ together with equivariant maps for every face inclusion.
In particular the category of Mayer--Vietoris resolutions on the terminal semi-simplicial set with the trivial $G$-action is equivalent to the derived category of (smooth) $G$-representations.
The name is based on the terminology used by Waldhausen in \cite{WH78a} and \cite{WH78b} and where many ideas in this paper originate. 
Note also that similar objects have already been considered by Schneider and Stuhler in \cite{SS97} and Winges in \cite{winges2014filtering}.
After introducing this definition we prove some properties of the category of all Mayer--Vietoris resolutions, in particular, we show that it is (co-)complete and compactly generated, and we provide an explicit set of compact generators. 
Furthermore, we show that the formation of Mayer--Vietoris resolutions is a functor on smooth semi-simplicial $G$-sets.

The goal of the second section is to compute the $K$-theory of $G$. 
We start by constructing a fibre sequence of $K$-theory spectra.
To do this we use the unique $G$-map $X \to *$ which induces a functor of Mayer--Vietoris categories. 
It turns out that this is a left Bousfield localization, hence taking the kernel on the left, which we will denote by $\MV_0(X)$, we obtain an exact sequence of categories in $Pr^L_{st, \omega}$ 
\[\begin{tikzcd}[ampersand replacement=\&]
	{\MV_0(X)} \&\& {\MV(X)} \&\& {D(G)}
	\arrow[from=1-3, to=1-5]
	\arrow[from=1-1, to=1-3]
\end{tikzcd}\] 
and hence a bifibre sequence on $K$-theory.
We then proceed by computing the $K$-theory of the middle and left-hand terms together with the induced map between them to obtain the $K$-theory spectrum of $G$.

In the next section, we establish a relative version of the above fibre sequence for every principal series type.
Before we begin with the construction, we recall some definitions and results about the existence of types due to Roche \cite{RO98} that assign to every character $\chi: T^0 \to \C^\times$ a type.
After that, we start by defining a full subcategory of $\MV(X)$ by only allowing those resolutions that pointwise take values in some specified subcategories depending on the type. 
This again defines a functor 
\[\begin{tikzcd}
	{\MV(X, \rho_\chi)} && {D(G, \rho_\chi)}
	\arrow[from=1-1, to=1-3]
\end{tikzcd}\]
and by the same means as before we can compute the kernel. 
However, in this situation, it turns out to require a lot more work to show that this functor is a left Bousfield localization or equivalently that the counit of a certain adjunction is an equivalence.
The idea to prove this is to define a filtration that in some sense measures how much the counit fails to be an equivalence.  
We then show that all quotients for this filtration vanish, which will occupy the main part of this section, and conclude that the counit is an equivalence. 
The key ingredient to achieve this is an explicit calculation of the intertwiners of the principal series types that was done by Roche in \cite{RO98}.

At the end of this paper, we apply the results we have obtained to the explicit example of $\mathrm{SL}_2(F)$ for non-archimedean local fields $F$.
We first provide a pushout formula for the entire $K$-theory spectrum of $\mathrm{SL}_2(F)$.
After that, we proceed with the calculation for the cuspidal Bernstein blocks followed by the principal series blocks. 
In particular, we highlight the case of the Iwahori block. 
Since the only non-trivial Levi subgroup of $\mathrm{SL}_2(F)$ is (up to conjugation) the maximal torus, it follows that these calculations cover all Bernstein blocks of $\mathrm{SL}_2(F)$.

\subsection*{Acknowledgements}
This paper is my PhD-project at the University of Münster.
First and foremost, I am very grateful to my advisor, Arthur Bartels, for his support, advice, and the many insightful conversations we had in the process of this project.
I would also like to thank Christoph Winges for his help during the early stages of this paper which led me in the right direction.
Last but not least, I also want to thank Claudius Heyer for answering numerous questions I had.
Funded by the Deutsche Forschungsgemeinschaft (DFG, German Research Foundation) under Germany's Excellence Strategy EXC 2044 –390685587, Mathematics Münster: Dynamics–Geometry–Structure.

\section{Mayer--Vietoris resolutions}

\subsection{Absolute resolutions}
The goal of this section is to set up the machinery that will be used in the later sections.
We will introduce the notion of Mayer--Vietoris resolutions as a localization of a certain functor category, followed by establishing the first general categorical properties.
In particular, we will show (co-)completeness and compute the compact objects.  The latter will be later used to compute the $K$-theory spectrum.

Throughout this paper, we will always fix a regular ring $R$ with $\Q \subseteq R$ whereby regular we mean Noetherian and regular coherent.
However, for the purpose of readability, we will always omit $R$ from our notation.

We begin with the following observation about the derived category $D(G)$ of $G$-representations over $R$ and $G$-actions on $D(R)$.
\begin{lemma}
    Let $G$ be a discrete group. Then the category $D(R)^{BG}=\Fun(BG, D(R))$ is equivalent to the derived category $D(G)$ of $G$-representations over $R$.
\end{lemma}

\begin{proof}
    The strategy of the proof is to show that both categories are equivalent to a category of modules over the same ring. 
    Write $f$ for the functor $* \to BG$.
    The forgetful functor $f^*: D(R)^{BG} \to D(R)$ is part of an adjunction 
\[\begin{tikzcd}
	{D(R)^{BG}} &&& {D(R)}
	\arrow[""{name=0, anchor=center, inner sep=0}, "{\res{f}{}}", from=1-1, to=1-4]
	\arrow[""{name=1, anchor=center, inner sep=0}, "{\Lan_f}"', shift right, curve={height=18pt}, from=1-4, to=1-1]
	\arrow[""{name=2, anchor=center, inner sep=0}, "{\Ran_f}", shift left, curve={height=-18pt}, from=1-4, to=1-1]
\end{tikzcd}\]
with $\Lan_f \dashv \res{f}{}  \dashv \Ran_f$.
This implies that $\Lan_f$ preserves compact objects, in particular $\Lan_fR$, with $R$ concentrated in degree $0$, is compact. 
Furthermore, for any $D \in D(R)^{BG}$ we have 

$$
\map_{D(R)^{BG}}(\Lan_fR, D) \simeq \map_{D(R)}(R, \res{f}{}D) \simeq D,
$$ 
in other words $\Lan_fR$ corepresents the forgetful functor.
In particular, it also corepresents the forgetful functor $\mathrm{Mod}(R)^{BG} \to \mathrm{Mod}(R)$ when restricted to the heart of the standard $t$-structure.
Since equivalences in functor categories can be checked pointwise the latter functor is conservative, hence $\Lan_f R$ is a compact generator of $D(R)^{BG}$ and by Schwede--Shipley \cite[7.1.2.1]{HA} 
$$D(R)^{BG} \simeq  D(\End(\Lan_f R)^{\op}).$$ 
By use of the colimit formula for left Kan extensions, we compute that the underlying object in $D(R)$ of $\Lan_f R$ is the colimit of the constant diagram $G \to D(R)$ with value $R$. 
But this is precisely $G \otimes R G$.
Since $G$ is discrete this has homology concentrated in degree $0$, hence $\Lan_f R$ is in the heart of the standard $t$-structure $( D(R)^{BG})^\heartsuit \simeq \mathrm{Mod}(R)^{BG}$.

On the other hand we know that $R[G]$ as a complex concentrated in degree $0$ is a compact generator of $D(\mathrm{Mod}(R)^{BG})$ and hence again by Schwede--Shipley $D(\mathrm{Mod}(R)^{BG})\simeq D(\End(R[G])^{\op})$.
Furthermore, $R[G]$ lies in the heart $D(\mathrm{Mod}(R)^{BG})^\heartsuit \simeq \mathrm{Mod}(R)^{BG}$ and also corepresents the forgetful functor $\mathrm{Mod}(R)^{BG} \to \mathrm{Mod}(R)$.
Hence, the Yoneda lemma implies that $\Lan_f R$ and $R[G]$ are equivalent in $\mathrm{Mod}(R)^{BG}$, in particular, both have the same endomorphism ring, and their module categories are equivalent.
\end{proof}

Now let $G$ act on a semi-simplicial set $X$. Then we define a ($G$-equivariant) category of simplices $\int_G \simp{X}$. 
\begin{definition}
    Let $G$ be a group acting on a semi-simplicial set $X$ and let $\Delta^{\op}_{\mathrm{inj}}\times BG \to \mathrm{Sets}$ be the corresponding functor.
    We define $\int_G \simp{X}$ to be the total space of the associated left fibration $\int_G \simp{X} \to \Delta^{\op}_{\mathrm{inj}} \times BG$.
\end{definition}
\begin{remark}
   Because $\Delta^{\op}_{\mathrm{inj}} \times BG \to Sets$ takes values in sets, the total space of the associated left fibration is a $1$-category and can be described by the classical Grothendieck construction. 
   More explicitly, an object of $\int_G \simp{X}$ is given by a simplex  $\Delta^n \to X$. For simplices $x:\Delta^n \to X$ and $y: \Delta^m \to X$ in $X$, a morphism from $x$ to $y$ is given by a map $\partial :\Delta^n \to \Delta^m$ together with an element $g \in G$ such that 
\[\begin{tikzcd}
	{\Delta^n} && {\Delta^m} \\
	\\
	& X
	\arrow["gx"', from=1-1, to=3-2]
	\arrow["\partial", from=1-1, to=1-3]
	\arrow["y", from=1-3, to=3-2]
\end{tikzcd}\]
     commutes. In other words, a morphism from $x$ to $y$ is a pair $(\partial,g)$ with $g \in G$ and a morphism $\partial$ in $\Delta^{\op}_{\mathrm{inj}}$ such that $\partial(gx)=y$.
     Note that maps go from simplices of higher dimensions to simplices of lower dimensions. 
     In particular, the endomorphisms for a simplex $x \in X$ are precisely the elements on $G$ that stabilize $x$.
\end{remark}
    From now on we assume $G=\mathbb{G}(F)$ for some reductive algebraic group $\mathbb{G}$ over a non-archimedean local field $F$, in particular $G$ carries a natural topology.
    Furthermore, we will always assume that the action of $G$ on $X$ is smooth, i.e. all stabilizers $G_x$ for $x \in X$ are open in $G$.
    Let $D(G)$ be the derived category of smooth $G$-representations over $R$.
    The underlying discrete group of $G$ is denoted by $G^\delta$.
    Furthermore, there is an adjunction 

\[\begin{tikzcd}
	{D(G)} && {D(G^\delta)}
	\arrow[""{name=0, anchor=center, inner sep=0}, shift left=3, from=1-1, to=1-3, curve={height=-6pt}]
	\arrow[""{name=1, anchor=center, inner sep=0}, "{(-)^\infty}", shift left=3, from=1-3, to=1-1, curve={height=-6pt}]
	\arrow["\dashv"{anchor=center, rotate=-90}, draw=none, from=0, to=1]
\end{tikzcd}\]
    where the left adjoint is induced by the inclusion of the full subcategory of smooth $G$-representations into all $G$-representations and $(-)^\infty$ is induced by taking the largest smooth subrepresentation.

    Next, we consider the category of functors $M: \int_G \simp{X} \to D(\C)$. For any simplex $x \in X$, the value of $M$ at $x$ carries an action of the stabilizer $G_x$ i.e. $M(x)$ is a $G_x$-representations.
    Furthermore, if $y$ is a face of $x$ then $G_x \subset G_y$ and there is a morphism $x \to y$ in $\int_G \simp{X}$ that induced a ($G_x$-equivariant map) from $M(x)$ to $M(y)$. 
    Note that since $G$ is totally disconnected, the homotopy type of $G$ is discrete, in particular, the classifying space $BG$ can not distinguish $G$ from its underlying discrete group, and our construction so far does not take into account the topology of $G$.
    This is what we are going to fix with the next definition where we define the notion of a smooth functor $\int_G \simp{X} \to D(R)$.
\begin{definition}
    A functor $M:\int_G \simp{X} \to D(R)$ is called smooth if for every simplex $x \in X$ and every $V:BG_x \to D(\C)$ the map 
    $$
    \begin{tikzcd}
      \map_{\int_G \simp{X}} (M, \Ran_{BG_x}^{\int_G\simp{X}} V^\infty) \arrow[r] & \map_{\int_G \simp{X}}(M , \Ran_{BG_x}^{\int_G \simp{X}} V)   
    \end{tikzcd}
    $$
    induced from the counit $V^\infty \to V$ is an equivalence. 
    Define $\MV(X)$ to be the full subcategory of $\Fun(\int_G \simp{X}, D(R))$ spanned by all smooth functors.
    We also call the objects of $\MV(X)$  Mayer--Vietoris resolutions (with coefficients in $R$).
\end{definition}
\begin{remark}
    \label{smoothRest}
    It follows immediately from the definition, that a resolution $M$ on $X$ is smooth precisely if for every $x \in X$ and $V\in D(G_x^\delta)$ the map 
\[\begin{tikzcd}[ampersand replacement=\&]
	{\map_{G_x}(M(x), V^\infty)} \&\& {\map_{G_x}(M(x), V)}
	\arrow[from=1-1, to=1-3]
\end{tikzcd}\]
    is an equivalence. In other words, $M$ is smooth if and only if the restriction of $M$ to $BG_x$ lies in the full subcategory $D(G_x)$ of $D(G^\delta_x)$ for every simplex $x\in X$.
\end{remark}
    
Next, we provide an easy way to construct (compact) Mayer--Vietoris resolutions. 
\begin{lemma}
    \label{smoothKan}
    For every simplex $x \in X$ and any smooth representation $V \in D(G_x)$ the functor $\Lan_{BG_x}^{\int_G \simp{X}} V$ obtained by left Kan extending along the full inclusion $BG_x \to \int_G \simp{X}$ is 
    smooth.
    If $V \in D(G_x)^\omega$ then $\Lan_{BG_x}^{\int_G \simp{X}}V$ is compact in $MV(X)$.
\end{lemma}
\begin{proof}
    Let $y\in X$ and $W:BG_y \to D(R)$. 
    We have to check that 
\[\begin{tikzcd}
	{\map_{\int_G \simp{X}} (\Lan_{BG_x}^{\int_G \simp{X}}V, \Ran_{BG_x}^{\int_G\simp{X}} W^\infty)} & {} & {} \\
	{\map_{\int_G \simp{X}}(\Lan_{BG_x}^{\int_G \simp{X}}V , \Ran_{BG_x}^{\int_G \simp{X}} W)}
	\arrow[from=1-1, to=2-1]
\end{tikzcd}\]
    is an equivalence.
    By the universal property of left Kan extensions, we have
    $$
    \map_{\int_G \simp{X}} (\Lan _{BG_x}^{\int_G \simp{X}}V, \Ran_{BG_y}^{\int_G \simp{X}}W)\simeq \map_{G_x} (V, (\Ran_{BG_y}^{\int_G\simp{X}}W)(x)).
    $$ 
    We can compute the underlying object of $(\Ran_{BG_y}^{\int_G\simp{X}}W)(x)$ by use of the limit formula for right Kan extensions. This yields 

    $$
    ( \Ran_{BG_y}^{\int_G \simp{X}}W)(x) = \lim_{x/BG_y}W.
    $$ 
    To understand the limit we need to compute the category $x/BG_y$. 
    The objects are given by pairs $(\partial,g)$ with $g\in G$ together with a map $\partial : [n] \to [m]$ in $\Delta^{\op}_{\mathrm{inj}}$ such that $g^{-1}y= \partial(x) $. 
    A morphism from $(\partial_0,g_0):x \to y$ to $(\partial_1,g_1):x \to y$  only exists if $\partial_0=\partial_1$ and is given by $h \in G_y$ such that $g_0 =g_1 h$. 
    This shows that $x/BG_y$ is a groupoid and every object has a trivial set of automorphisms.
    Hence, $x/BG_y$ is a set.
    For a fixed $\partial: [n] \to [m]$, two objects $(\partial,g_0): x \to y$ and $(\partial,g_1) :x \to y$ only differ by $g_0^{-1}g_1 \in G_y$, hence the number of isomorphism classes is bounded by the number of maps $\partial: [n] \to [m]$. 
    In particular, there are only finitely many and the limit is a finite product.
    Then for any $(\partial,g)\in x/BG_y$ conjugation with $g$ maps $G_x$ into $G_y$ and $G_x$ acts on $W$ via restriction along this conjugation map. Hence,  
     $$
     \lim _{x / BG_y}W \simeq \prod_{(d,g)\in x/BG_y} \res{G_y^g}{G_x}W^g.
     $$
    Then $\Lan^{\int_G \simp{X}}_{BG_x}V$ is smooth if and only if the induced map
\[\begin{tikzcd}[ampersand replacement=\&]
	{\map_{G_x}(V,\prod \res{G_y^g}{G_x}W^\infty)} \&\& {\map_{G_x}(V,\prod \res{G_y^g}{G_x}W)}
	\arrow[from=1-1, to=1-3]
\end{tikzcd}\] 
    is an equivalence for every $y \in X$ and $W : BG_y \to D(R)$.
    But this is true precisely if $V$ is in $D(G_x)$, that is if $V$ is smooth.
\end{proof}
The calculation in the proof also provides the following corollary that we will need again at a later point.
\begin{corollary}
    \label{Ranres}
    For simplices $x,y \in X$ the value of the right Kan extension is computed by 
    $$\res{\int_G \simp{X}}{BG_x}\Ran_{BG_y}^{\int_G \simp{X}}W= \prod_{x/BG_y} \res{G^g_y}{G_x}W^g.$$
    This product is always finite and hence compatible with filtered colimits. 
    Furthermore, this shows that the left adjoint can be computed as 
    $$
    \res{\int_G \simp{X}}{BG_y}\Lan_{BG_x}^{\int_G \simp{X}}V = \bigoplus_{x  /BG_y}\ind{G_x^{g^{-1}}}{G_y}V^{g^{-1}}.
    $$
\end{corollary}
\begin{proof}
    The first part follows directly from the previous proof. 
    For the second part first note that $    \res{\int_G \simp{X}}{BG_y}\Lan_{BG_x}^{\int_G \simp{X}}$ is left adjoint to  
    $\res{\int_G \simp{X}}{BG_x}\Ran_{BG_y}^{\int_G \simp{X}}$ because to composition of left adjoints is left adjoint to the composition of the corresponding right adjoints. 
    But from the explicit formula for the composition of the right adjoints, we know that the left adjoint has to be $\bigoplus_{x /BG_y}\ind{G_x^{g^{-1}}}{G_y}V^{g^{-1}}$.
\end{proof}
A consequence of the previous lemma is that left Kan extension and restriction along inclusions $BG_x \to \int_G \simp{X}$ restrict to an adjunction between smooth objects.
\begin{lemma}
    \label{Lancomp}
    By restricting the functors we obtain an adjunction 
\[\begin{tikzcd}
	{D(G_x)} && {MV(X).}
	\arrow[""{name=0, anchor=center, inner sep=0}, "{\Lan_{BG_x}^{\int_G \simp{X}}.}", shift left=2, curve={height=-6pt}, from=1-1, to=1-3]
	\arrow[""{name=1, anchor=center, inner sep=0}, "{\res{\int_G \simp{X}}{BG_x}}", shift left=2, curve={height=-6pt}, from=1-3, to=1-1]
	\arrow["\dashv"{anchor=center, rotate=-90}, draw=none, from=0, to=1]
\end{tikzcd}\]
    Small colimits exist in $\MV(X)$ and the inclusion $\MV(X) \to \Fun(\int_G \simp{X}, D(R))$ commutes with colimits, in particular the left adjoint preserves compact objects.
\end{lemma}
\begin{proof}
    Lemma \ref{smoothKan} and remark \ref{smoothRest} show that left Kan extension and restriction restrict to the full subcategories $D(G_x)$ and $\MV(X)$, and we obtain the desired adjunction.
    By definition $M \in \Fun(\int_G \simp{X} , D(R))$ is smooth if and only if the comparison map 
     $$
    \begin{tikzcd}
      \map_{\int_G \simp{X}} (M, \mathrm{Ran}_{BG_y}^{\int_G \simp{X}}V^\infty) \arrow[r] & \map_{\int_G \simp{X}}(M , \mathrm{Ran}_{BG_y}^{\int_G \simp{X}} V)   
    \end{tikzcd}
    $$
    is an equivalence for all $V: BG_y \to D(R)$. But $\map(-,-)$ is compatible with colimits in the first variable, hence a colimit of smooth objects in $\Fun(\int_G \simp{X}, D(R))$ is again smooth. Because $\MV(X)$ is a full subcategory this shows that the colimit of smooth objects in $\Fun(\int_G \simp{X}, D(R))$ also satisfies the universal property for the colimit in $\MV(X)$. In particular colimits in $\MV(X)$ exist and the inclusion $\MV(X) \to \Fun(\int_G \simp{X}, D(R))$ commutes with them.

    Since the restriction along $BG_x \to \int_G \simp{X}$ commutes with colimits in the entire functor category, it follows that the restriction to smooth objects also commutes with colimits.  
    In particular by Corollary \ref{Ranres} the right adjoint of  $\Lan_{BG_x}^{\int_G \simp{X}} : D(G_x) \to \MV(X)$ commutes with filtered colimits, hence it preserves compact objects.
\end{proof}
To invoke the adjoint functor theorem we need the following lemma. 
\begin{lemma}
    \label{MVcomp}
    Let $X$ be a $G$-semi-simplicial set. 
    The set of $\Lan_{BG_x}^{\int_G \simp{X}}V$ for simplices $x \in X$ and $V \in D(G_x)^\omega$ form a set of compact generators, in particular $\MV(X)$ is compactly generated.
\end{lemma}
\begin{proof}
    The category $D(G_x)$ is compactly generated for every $x \in X$. Let $C_x$ be the set of compact objects, in particular, $C_x$ generates $D(G_x)$. 
    Furthermore, every $x \in X$ induces a functor $BG_x \to \int_G \simp{X}$ and the left Kan extension along this functor preserves smooth, compact objects by \ref{Lancomp}.  
    Now let $C$ be the set that contains the image of $C_x$ under $D(D_x) \to MV(X)$ for all $x \in X$.
    We need to show that $C$ is jointly conservative. Let $M \in \MV(X)$ and $V_x \in C_x$. Then 
    \begin{equation*}
        \map_{\int_G \simp{X}}(\Lan_{BG_x}^{\int_G \simp{X}} V_x, M) \simeq \map_{D(G_x)}(V_x, M(x))
    \end{equation*}
    and $M(x)$ is zero if and only if the right-hand side is zero for all $V_x \in C_x$. 
    Hence, $M$ is zero in $\MV(X)$ if and only if it is pointwise zero if and only if the mapping spectra $\map_{\int_G \simp{X}}(\Lan_{BG_x}^{\int_G \simp{X}}V_x, M)$ are zero for all $x \in X$ and $V_x \in D(G_x)^\omega$. 
\end{proof}

\begin{remark}
    The same argument shows that the set of $\Lan_{BG_x}^{\int_G \simp{X}}V_x$ for $V_x \in D(G_x^\delta)^\omega$ form a set of compact generators of $\Fun (\int_G \simp{X}, D(R))$. This will allow us to use the adjoint functor theorem below.
\end{remark}

\begin{corollary}
    \label{smoothRes}
    The full inclusion $\MV(X) \to \Fun(\int_G \simp{X}, D(R))$ has a right adjoint $(-)^\infty: \Fun(\int_G \simp{X}, D(R)) \to \MV(X)$. In other words, $\MV(X)$ is a right Bousfield localization of $\Fun(\int_G \simp{X}, D(R))$.
\end{corollary}

\begin{proof}
    Both $\MV(X)$ and $\Fun(\int_G \simp{X}, D(R))$ are compactly generated by Lemma \ref{MVcomp} and the remark. Furthermore, the inclusion commutes with colimits by Lemma \ref{Lancomp}. Hence, by the adjoint functor theorem, there exists a right adjoint.
\end{proof}
The following is a standard result about Bousfield localizations. A proof can be found for instance in \cite[Corollary I.61.]{HebL}.
\begin{lemma}
    \label{cocompB}
    Let $L:\mathcal{C} \to \mathcal{D}$ be a left Bousfield localization with right adjoint $R$ and let $F: I \to \mathcal{D}$ be a diagram. If the right-hand side exists the following hold.
    \begin{equation*} 
            \lim_I F(i)= L(\lim_I RF(i)) \text{ and }\colim_I F(i)= L(\colim_I RF(i)) 
    \end{equation*}
    In particular if $\mathcal{C}$ is (co-) complete then $\mathcal{D}$ is (co-)complete.
    The analogous statement holds for right Bousfield localizations. 
\end{lemma}

\begin{corollary}
    \label{MVcolim}
    The category $\MV(X)$ is (co-) complete and stable. Colimits are computed pointwise and limits are the smoothening of the pointwise limits.
\end{corollary}
\begin{proof}
   By \ref{cocompB} we only need to show that $\Fun(\int_G \simp{X}, D(\C))$ is complete and cocomplete.
   Since $D(R)$ is complete and cocomplete, limits and colimits in functor categories to $D(R)$ exist and are computed pointwise.
   This shows that the category $\Fun( \int_G \simp{X}, D(R))$ is (co-) complete and stable.  
   Furthermore, \ref{cocompB} shows that cofibre and fibre sequences agree since they agree in $\Fun(\int_G \simp{X}, D(R))$, hence $\MV(X)$ is stable.
\end{proof}
\begin{lemma}
    The $t$-structure on $D(R)$ induces a $t$-structure on $\MV(X)$. A resolution is (co-)connected if and only if it is pointwise (co-)connected.
\end{lemma}
\begin{proof}
    Any adjunction between functors $F: \mathcal{D}\to \mathcal{E}$ and $G: \mathcal{E} \to \mathcal{D}$ induces an adjunction on functor categories when applying $\Fun(\mathcal{C}, -)$. 
    In particular the adjunctions
    \[\begin{tikzcd}
        {D_{\geq 0}(R)} && {D(R)} && {D_{\leq 0}(R)}
        \arrow[""{name=0, anchor=center, inner sep=0}, shift left=3, from=1-1, to=1-3, curve={height=-6pt}]
        \arrow[""{name=1, anchor=center, inner sep=0}, "{\tau_{\geq 0}}", shift left=3, from=1-3, to=1-1, curve={height=-6pt}]
        \arrow[""{name=2, anchor=center, inner sep=0}, "{\tau_{\leq 0}}", shift left=3, from=1-3, to=1-5, curve={height=-6pt}]
        \arrow[""{name=3, anchor=center, inner sep=0}, shift left=3, from=1-5, to=1-3, curve={height=-6pt}]
        \arrow["\dashv"{anchor=center, rotate=-90}, draw=none, from=0, to=1]
        \arrow["\dashv"{anchor=center, rotate=-90}, draw=none, from=2, to=3]
    \end{tikzcd}\]
    that correspond to the $t$-structure on $D(R)$ induce adjunctions after $\Fun(\int_G \simp{X}, -)$ and a $t$-structure on $\Fun(\int_G \simp{X}, D(R))$ such that (co-)connective objects are the pointwise (co-)connective functors. 
    We only need to check that this adjunction is compatible with smooth objects.
    But this follows directly from the explicit pointwise construction of the truncations in $D(R)$ since kernels and cockerels of smooth representations are smooth again.
\end{proof}

The following lemma generalizes Lemma \ref{Lancomp}.
It is a variant of the induction-restriction adjunction for representations parametrized over semi-simplicial sets and enjoys many of the same properties.

\begin{lemma}
    Let $f: X \to Y$ be a map of $G$-semi-simplicial sets. Then left Kan extension and restriction along $f$ induce an adjunction
\[\begin{tikzcd}
	{MV(X)} && {MV(Y)}
	\arrow[""{name=0, anchor=center, inner sep=0}, "{\Lan_f}", shift left=2, from=1-1, to=1-3, curve={height=-6pt}]
	\arrow[""{name=1, anchor=center, inner sep=0}, "{\res{f}{}}", shift left=2, from=1-3, to=1-1, curve={height=-6pt}]
	\arrow["\dashv"{anchor=center, rotate=-90}, draw=none, from=0, to=1]
\end{tikzcd}\]
    and the left adjoint preserves compact objects.
\end{lemma}
\begin{proof}
    We first show that the left Kan extension of smooth objects is again smooth.
    Recall from the proof of \ref{MVcomp} that $\MV(X)$ has a set of compact generators given by the left Kan extensions of compact objects from $D(G_x)$ along $BG_x \to \int_g \simp{X}$ for all simplices $x \in X$. 
    Consider the following commutative square and let $V_x \in D(G_x)$ be compact.
\[\begin{tikzcd}[ampersand replacement=\&]
	{\int_G \simp{X}} \&\& {\int_G \simp{Y}} \\
	\\
	{BG_x} \&\& {BG_{f(x)}}
	\arrow["f", from=1-1, to=1-3]
	\arrow[from=3-1, to=1-1]
	\arrow[from=3-3, to=1-3]
	\arrow[from=3-1, to=3-3]
\end{tikzcd}\]
    By the commutativity of the diagram, the left Kan extension of the upper composition is equivalent to the left Kan extension along the lower composition. 
    By left Kan extending $V_x$ along the lower map is precisely the induction $\ind{G_x}{G_{f(x)}}$ of $V_x$ to $G_{f(x)}$ and therefore again smooth and compact. Furthermore, Lemma \ref{smoothKan} shows that extending further to $\int_G \simp{Y}$ is still smooth.  
    This shows that $\Lan_f : \Fun(\int_G \simp{X}, D(R)) \to \Fun(\int_G \simp{X}, D(R))$ maps a set of smooth generators of $\MV(X)$ to smooth, compact objects in $\MV(Y)$. 
    But $\Lan_f$ also commutes with colimits, and the colimit of smooth objects is again smooth. 
    Since every object of $\MV(X)$ is a colimit of the compact generators $\Lan_f$ preserves smooth objects.
    In particular, it follows directly that $\Lan_f: \MV(X) \to \MV(Y)$ commutes with colimits, because colimits in $\MV(Y)$ are computed in $\Fun(\int_G \simp{Y}, D(R))$. 

    On the other hand, let $N \in \MV(Y)$. We have to show that 
     $$
    \begin{tikzcd}
      \map (\res{f}{}M, \Ran_{BG_x}^{\int_G\simp{X}} V_x^\infty) \arrow[r] & \map(\res{f}{}M , \Ran_{BG_x}^{\int_G \simp{X}} V_x)   
    \end{tikzcd}
    $$
    is an equivalence for all $x \in X$ and $V_x \in \Fun(BG_x, D(R))$.
    But by the universal property of right Kan extensions, we have 
    $$
    \map (\res{f}{}M, \Ran_{BG_x}^{\int_G \simp{X}}-) \simeq \map (\res{\int_G \simp{Y}}{BG_x}M, -)
    $$
    where we restrict along the composition $BG_x \to \int_G \simp{X} \to \int_G \simp{Y}$ on the right-hand side.
    By the commutativity, this is equivalent to restricting along $BG_x \to BG_{f(x)} \to \int_G \simp{Y}$ and therefore smooth.
    
    Lastly it follows formally that $\res{f}{}: \MV(Y) \to \MV(X)$ has a right adjoint 
    
\[\begin{tikzcd}[ampersand replacement=\&]
	{\MV(X)} \&\& {\Fun(\int_G \simp{Y}, D(R))} \&\& {\MV(Y).}
	\arrow["{\Ran_f}", from=1-1, to=1-3]
	\arrow["{(-)^\infty}", from=1-3, to=1-5]
\end{tikzcd}\]
    In particular this means that $\res{f}{}: \MV(Y) \to \MV(X)$ commutes with colimits and hence $\Lan_f: \MV(X) \to \MV(Y)$ preserves compact objects.    
\end{proof}
We can summarize the results we have established so far in the following way.
The formation of $\MV$ defines a functor $\MV : (ssSet^{BG})^{\infty} \to \mathrm{Pr}^L_{st, \omega}$ by $X \mapsto \MV(X)$ and induced maps by left Kan extensions.
Furthermore, limits and colimits can be computed in $\Fun(\int_G \simp{X}, D(R))$ followed by reflecting back to $\MV(X)$.

For any semi-simplicial set with $G$-action there is a canonical equivariant map to the terminal semi-simplicial set $*$ with the trivial $G$-action.
Note that by construction $\MV(*)=D(G)$.
Since this map will play an important role in the forthcoming sections we fix the following notation.
\begin{notation}
    Let $X$ be a semi-simplicial set with smooth $G$-action and let $X \to *$ be the canonical map.
    This induces a map 
\[\begin{tikzcd}
	{\MV(X)} && {D(G)}
	\arrow[from=1-1, to=1-3]
\end{tikzcd}\]
    that we will denote by $\alpha$.
    If there is more than one group involved will write $\alpha_G$ to avoid any confusion.
\end{notation}

So far we have only considered (maps between) $G$-semi-simplicial sets for actions for a fixed $G$. 
Similar to ordinary representation theory there are induction and restriction functors for any subgroup $H \subseteq G$ that preserves smooth objects. 

\begin{definition}
Let $H \subset G$ be an open subgroup. Then we define $\MV_H(X)$ to be the category of Mayer--Vietoris resolutions on $X$ with the restricted $H$-action.
In other words the full subcategory of $\Fun(\int_H \simp{X}, D(R))$ spanned by all smooth functors. 
\end{definition}
With the above definition, there is an adjunction between $\MV_H(X)$ and $\MV_G(X)$.
\begin{lemma}
    Let $(X \in ssSet^{BG})^\infty$ and let $H \subseteq G$ be an open subgroup.
    Then the functor $\int_H \simp{X} \to \int_G \simp{X}$ induces an adjunction 
\[\begin{tikzcd}
	{MV_H(X)} && {MV_G(X),}
	\arrow[""{name=0, anchor=center, inner sep=0}, "{\mathrm{ind}^G_H}", shift left=2, curve={height=-6pt}, from=1-1, to=1-3]
	\arrow[""{name=1, anchor=center, inner sep=0}, "{\mathrm{res}_G^H}", shift left=2, curve={height=-6pt}, from=1-3, to=1-1]
	\arrow["\dashv"{anchor=center, rotate=-90}, draw=none, from=0, to=1]
\end{tikzcd}\]
    such that $\ind{H}{G}$ preserves compact objects.
\end{lemma}
\begin{proof}
   For the first part, we only need to show that left Kan extension and restriction along $\int_H \simp{X} \to \int_G \simp{X}$ 
\[\begin{tikzcd}[ampersand replacement=\&]
	{\Fun(\int_H \simp{X}, D(R))} \&\& {\Fun(\int_G \simp{X}, D(R))}
	\arrow[""{name=0, anchor=center, inner sep=0}, "\Lan", shift left=2, from=1-1, to=1-3]
	\arrow[""{name=1, anchor=center, inner sep=0}, "{\res{}{}}", shift left=2, from=1-3, to=1-1]
	\arrow["\dashv"{anchor=center, rotate=-90}, draw=none, from=0, to=1]
\end{tikzcd}\] 
   preserve smooth objects. 
   Since both functors are compatible with colimits it suffices to check this on a set of generators. We only do this for the right adjoint, the argument for the left adjoint is analogous.
   Fix a simplex $x\in X$. Then we have the following commutative diagram.
\[\begin{tikzcd}[ampersand replacement=\&]
	{\int_H \simp{X}} \&\& {\int_G \simp{X}} \\
	\\
	{BH_x} \&\& {BG_x}
	\arrow[from=3-1, to=1-1]
	\arrow[from=1-1, to=1-3]
	\arrow[from=3-1, to=3-3]
	\arrow[from=3-3, to=1-3]
\end{tikzcd}\]
    where $G_x$ and $H_x=G_x \cap H$ are open subgroups of $G$.
    Let $V \in D(G_x)^\omega$ and consider $\Lan_{BG}^{\int_G \simp{X}}V$. 
    We need to check that the restriction along the top composition is smooth. But by commutativity, this is equivalent to restricting along the bottom composition which yields
    $$\res{\int_G \simp{X}}{BH_x}\Lan_{BG}^{\int_G \simp{X}}V \simeq \res{BG_x}{BH_x} \res{\int_G \simp{X}}{BG_x} \Lan_{BG_x}^{\int_G \simp{X}}V \simeq \res{BG_x}{BH_x}V.$$
    Because restrictions of smooth representations are smooth the right-hand side is smooth. 

\end{proof}

For the remainder of this section, we will collect some technical lemmas that we will frequently use later.
The first one provides us with an easy way to produce fibre sequences of Mayer--Vietoris resolutions.
\begin{notation}
Let $i: Y \to X$ be a smooth $G$-semi-simplicial subset. This induces an adjunction 
\[\begin{tikzcd}
	{\MV(Y)} && {\MV(X)}
	\arrow[""{name=0, anchor=center, inner sep=0}, "{\Lan_i}", shift left=3, curve={height=-6pt}, from=1-1, to=1-3]
	\arrow[""{name=1, anchor=center, inner sep=0}, "{\res{i}{}}", shift left=3, curve={height=-6pt}, from=1-3, to=1-1]
	\arrow["\dashv"{anchor=center, rotate=-90}, draw=none, from=0, to=1]
\end{tikzcd}\]
For any $M \in \MV(X)$ we write $M_Y:= \Lan_i\res{i}{}M$.
This comes with a canonical map $M_Y \to M$.
If $X\setminus Y$ is the complement of a $G$-semi-simplicial subset we define $M_{X\setminus Y}$ as the cofibre of the canonical map $M_Y \to M$.
Furthermore, we write $M(Y):= \alpha(M_Y)$ and $M(X \setminus Y):= \alpha (M_{X \setminus Y})$.
\end{notation}

\begin{lemma}
    \label{fibseq}
    Let $Y$ be a smooth semi-simplicial subset $X$ that is preserved under the action of an open subgroup $H\subseteq G$. Then we have an induced map $M_Y \to M$ in $\MV_H(X)$ that is an equivalence of every simplex $y \in Y$ and $M_Y$ is zero outside $Y$. 
    Furthermore, it is part of a bifibre sequence  
\[\begin{tikzcd}[ampersand replacement=\&]
	{M_Y} \&\& M \&\& {M_{X \setminus Y}.}
	\arrow[from=1-1, to=1-3]
	\arrow[from=1-3, to=1-5]
\end{tikzcd}\] 
    where $M_{X \setminus Y}$ is equivalent to $M$ on $X \setminus Y$ and zero otherwise.
    In particular, this induces a bifibre sequence in $D(H)$ 
\[\begin{tikzcd}[ampersand replacement=\&]
	{M(Y)} \&\& {M(X)} \&\& {M({X \setminus Y})}
	\arrow[from=1-1, to=1-3]
	\arrow[from=1-3, to=1-5]
\end{tikzcd}\]
    after realization.

\end{lemma}
\begin{proof}
    First we restrict $M: \int_G \simp{X} \to D(R)$ to $\int_H \simp{X}$.
   The resolution $M_Y$ is obtained by restricting $M$ to $\int_H \simp{Y}$ followed by left Kan extending back to $\int_H \simp{X}$ and the map $M_Y \to M$ is the counit of the restriction-extension adjunction. 
   Note that $\int_H \simp{Y}$ is the full subcategory of $\int_H \simp{X}$ spanned by the simplices of $Y$, hence the counit $M_Y \to M$ is an equivalence on all simplices of $Y$. 
   On the other hand, the pointwise formula for left Kan extensions allows us to compute the values of $M_Y$ outside $Y$ as the colimit 
   $$
   M_Y(x) = \colim_{(y\to x) \in \int_H   \simp{Y}/ x}M(y).
   $$
   But because $Y$ is a subcomplex and stable under $H$ this category is empty if $x \notin Y$ and hence the colimit is zero.
   By \ref{MVcolim} the cofibre of this map is computed pointwise, but since $M_Y \to M$ is an equivalence on $Y$ and zero otherwise we see that $M_{X\setminus Y}$ agrees with $M$ on $X \setminus Y$ and is zero everywhere else.
   The last claim follows because realization preserves colimits, in particular cofibre sequences.
\end{proof}

We will be particularly interested in computing the left Kan extension of resolutions along $X \to *$. 
In this case, it is sometimes useful to have an explicit formula for the extension.  
\begin{lemma}
    \label{RealFormula}
    Let $X$ be a semi-simplicial set with smooth $G$-action. Then for any $M \in \MV(X)^\heartsuit$ we can compute $\alpha(M)$ as the complex 

\[\begin{tikzcd}[ampersand replacement=\&]
	\ldots \&\& {\bigoplus_{[x] \in X_1/G}\ind{G_x}{G}M(x)} \&\& {\bigoplus_{[x] \in X_0/G}\ind{G_x}{G}M(x).}
	\arrow["{\sum_{i=0}^2(-1)^i\partial_i}", from=1-1, to=1-3]
	\arrow["{\sum_{i=0}^1(-1)^i\partial_i}", from=1-3, to=1-5]
\end{tikzcd}\]
In particular, we can compute the underlying complex of $\alpha(M)$ as  
\[\begin{tikzcd}[ampersand replacement=\&]
	\ldots \&\& {\bigoplus_{x \in X(1)}M(x)} \&\& {\bigoplus_{x \in X(0)}M(x).}
	\arrow["{\sum_{i=0}^2(-1)^i\partial_i}", from=1-1, to=1-3]
	\arrow["{\sum_{i=0}^1(-1)^i\partial_i}", from=1-3, to=1-5]
\end{tikzcd}\]
\end{lemma}

Before we prove the formula we need the following auxiliary result.
\begin{lemma}[Joyal]
    \label{BeckChav}
    Consider a cartesian square of categories
\[\begin{tikzcd}
	A && B \\
	\\
	C && D
	\arrow["{g'}", from=1-1, to=1-3]
	\arrow["f", from=1-3, to=3-3]
	\arrow["{f'}"', from=1-1, to=3-1]
	\arrow["g"', from=3-1, to=3-3]
\end{tikzcd}\]
such that $f$ is proper. Then for any cocomplete category $E$ the induced square obtained by applying $\Fun(-, E)$ is adjointable, i.e. for any $F \in \Fun(B, E)$ the comparison map 
\[\begin{tikzcd}
	{\Lan_{f'} \res{g'}{}F} && {\res{g}{}\Lan_{f}F}
	\arrow[from=1-1, to=1-3]
\end{tikzcd}\]
is an equivalence.
\end{lemma}
\begin{proof}
    Since $E$ is cocomplete every functor from $B$ to $E$ factors over the free cocompletion of $B$, i.e. the presheaf category $\mathcal{P}(B)$.
    This induces a commutative diagram 
\[\begin{tikzcd}
	A && B && E \\
	& {\mathcal{P}(A)} && {\mathcal{P}(B)} \\
	C && D \\
	& {\mathcal{P}(C)} && {\mathcal{P}(D)}
	\arrow[from=1-1, to=2-2]
	\arrow[from=1-3, to=2-4]
	\arrow[from=2-2, to=4-2]
	\arrow[from=2-4, to=4-4]
	\arrow[from=2-2, to=2-4]
	\arrow[from=4-2, to=4-4]
	\arrow["{f'}"{description, pos=0.7}, from=1-1, to=3-1]
	\arrow["{g'}"{description, pos=0.6}, from=1-1, to=1-3]
	\arrow["g"{description, pos=0.6}, from=3-1, to=3-3]
	\arrow["f"{description, pos=0.7}, from=1-3, to=3-3]
	\arrow[from=3-1, to=4-2]
	\arrow[from=3-3, to=4-4]
	\arrow["F"{description}, from=1-3, to=1-5]
	\arrow[from=2-4, to=1-5]
\end{tikzcd}\]
with maps induced by left Kan extension.
Because $f$ is proper the square of presheaf categories is adjointable by \cite[Proposition 11.6]{JoyalL}, hence  
\[\begin{tikzcd}
	{\Lan_{\mathcal{P}(f')} \res{\mathcal{P}(g')}{} \Lan \, F} && {\res{\mathcal{P}(g)}{} \Lan_{\mathcal{P}(f)} \Lan \, F}
	\arrow[from=1-1, to=1-3]
\end{tikzcd}\]
is an equivalence.
Pulling back along the Yoneda embedding $y_C: C \to \mathcal{P}(C)$ on both sides and using the commutativity together with the fact that $\Lan_{y_C}\res{y_C}{}  \to \mathrm{id}$ is an equivalence then proves the claim.
\end{proof}
Now we can proceed with the proof of \ref{RealFormula}.

\begin{proof}
    We only show the first assertion, the second claim follows with $G=\lbrace e \rbrace$.
    Given a functor $M: \int_G \simp{X} \to D(R)$ we want to compute the left Kan extension along $\int_G \simp{X} \to BG$. This can be factored as 
\[\begin{tikzcd}[ampersand replacement=\&]
	{\int_G \simp{X}} \&\& {\Delta^{op}_{inj}\times BG} \&\& BG,
	\arrow[from=1-1, to=1-3]
	\arrow[from=1-3, to=1-5]
\end{tikzcd}\]
    and we can compute the left Kan extension in two steps. We start by extending along $\int_G \simp{X} \to \Delta^{op}_{inj}\times BG$.
    Consider the following diagram 
\[\begin{tikzcd}[ampersand replacement=\&]
	F \&\& {\int_G \simp{X}} \&\& {D(R)} \\
	\\
	{[k]\times BG} \&\& {\Delta^{op}_{inj}\times BG}
	\arrow["{g'}"{description}, from=1-1, to=1-3]
	\arrow["f"{description}, from=1-3, to=3-3]
	\arrow["g"{description}, from=3-1, to=3-3]
	\arrow["{f'}"{description}, from=1-1, to=3-1]
	\arrow["M"{description}, from=1-3, to=1-5]
\end{tikzcd}\]
where $F$ is the pullback of $f$ and $g$.
Since $f:\int_G \simp{X} \to \Delta^{op}_{\mathrm{inj}} \times BG$ is a left fibration it is also proper \cite[Theorem 11.9.]{JoyalL}.
Then by \ref{BeckChav}, because $D(R)$ is cocomplete, the comparison map $\Lan_{f'}\res{g'}{}M \to \res{g}{}\Lan_{f}M$ is an equivalence. 
In other words, to understand the left Kan extension along $f$ evaluated at $[k]\times {*}$ including the $G$-action we need to compute the left Kan extension of the restriction along $f'$. 
The fibre $F$ is precisely the full subcategory of $\int_G \simp{X}$ containing all $k$-simplices.
That is, the objects are $k$-simplices of $X$, morphisms are given by the action of $G$, and the functor to $[k]\times BG$ is constant on objects and sends a morphism $(\id, g): x \to gx$ to $g$.  
This is in fact a groupoid and as such equivalent to $X(k)_{hG}\simeq \coprod_{x \in X(k)/G} BG_x$. 
Under this identification, $f'$ is the inclusion $BG_x \to BG$ on every factor.
But left Kan extending along $BG_x \to BG$ is precisely the induction of a $G_x$-module to a $G$-module, hence $\res{g}{}\Lan_fM \simeq \bigoplus_{[x] \in X_k/G}\ind{G_x}{G}M(x)$. 
The simplicial structure maps are given by the induction of the structure maps of $M$. 

In the second step, we compute the left Kan extension along $\Delta^{op}_{inj} \times BG \to BG$.
By adjoining the $BG$-factor to the right this is equivalent to computing realizations $\Fun(\Delta^{op}_{inj}, D(R)^{BG}) \to \Fun(BG, D(R))$.
But by Dold--Kan, this is precisely the Moore complex associated to a simplicial object in $(D(R)^{BG})^\heartsuit$ and hence in combination with the previous computation exactly what we claimed. 
\end{proof}

The previous lemma specializes in particular to the case where $X$ is $1$-dimensional.
We will use this case so frequently that it is worthwhile to state it explicitly.
\begin{corollary}
    \label{AssemFormula}
    Let $X$ be a $1$-dimensional semi-simplicial set with smooth $G$-action and let $M \in \MV(X)^\heartsuit$. Then 
    $\alpha (M)$ can be computed as the complex 
\[\begin{tikzcd}[ampersand replacement=\&]
	{\bigoplus_{[e] \in X_1/G} \ind{G_e}{G}M(e)} \&\& {\bigoplus_{[v] \in X_0/G} \ind{G_v}{G}M(v)}
	\arrow[from=1-1, to=1-3]
\end{tikzcd}\]
    and the differential is given by the difference of the induction of the maps $M(e) \to M(\partial_ie)$ with $i=0,1$.
    The underlying complex of $R$-modules can be realized as 
\[\begin{tikzcd}
	{\bigoplus_{e \in X_1}M(e)} && {\bigoplus_{v \in X_0}M(v)}
	\arrow[from=1-1, to=1-3]
\end{tikzcd}\]
    with differential given by the difference of the maps induced by the faces  $\partial_0$ and $\partial_1$.
\end{corollary}

Next, we give a precise description of the compact objects of $\MV(X)$. This will later allow us to compute the $K$-theory of this category.

\begin{lemma}
    \label{pointcomp}
    If the action of $G$ on $X$ is cocompact, then
    $M\in \MV(X)$ is compact if and only if for every $x\in X$ the restriction of $M$ along $BG_x \to \int_G \simp{X}$ is compact.
\end{lemma}
\begin{proof}
First, we show that restriction along $BG_x \to \int_G \simp{X}$ preserves compact objects.
Since restriction is left adjoint to right Kan extension it suffices to show that the latter commutes with filtered colimits.   
By \ref{MVcolim} colimits in $\MV(X)$ are computed pointwise. But remark \ref{Ranres} provides a pointwise formula for the right Kan extension which is compatible with colimits.
Therefore, compact objects in $\MV(X)$ are pointwise compact objects in $D(G_x)$.

Now for the other direction assume that all restrictions of $M$ are compact and let $n$ be the dimension of $X$ (since the action it cocompact $X$ has to be of finite dimension).
If $n=0$ then $\int_G \simp{X}$ is a finite coproduct of $BG_x$ and the claim is true.
Next, assume that that claim is true for all semi-simplicial sets up to dimension $n-1$. 
By \ref{fibseq} every $M \in \MV(X)$ is part of a bifibre sequence

\[\begin{tikzcd}
	{M_{\sk{n-1}}} && M && {M_{X_n},}
	\arrow[from=1-1, to=1-3]
	\arrow[from=1-3, to=1-5]
\end{tikzcd}\]
and $M_{\sk{n-1}}$ is the left Kan extension of a pointwise compact object on a semi-simplicial set of dimension $n-1$, hence by assumption compact in $\MV(X)$.
Because finite colimits of compact objects are compact again we only need to show that $M_{X_n}$ is compact. 

Again $M_{X_n}$ is part of a fibre sequence 

\[\begin{tikzcd}
	{\bigoplus_{[x] \in X_n/G} \Lan_{BG_x}^{\int_G \simp{X}} M(x)} && {M_{X_n}} && {\mathrm{fib}.}
	\arrow[from=1-3, to=1-5]
	\arrow[from=1-1, to=1-3]
\end{tikzcd}\]
The left map is the (direct sum) of the counit of the adjunction between left Kan extending and restricting along $BG_x \to \int_G \simp{X}$. In particular, since $BG_x$ is a full subcategory, it is an equivalence on $n$-simplices.
Furthermore, $M(x)$ was assumed to be compact, therefore $\Lan_{BG_x}^{\int_G \simp{X}} M(x)$ is also compact (and by the previous argument therefore pointwise compact).
Additionally, the coproduct is finite, because $G$ acts cocompactly on $X$, hence again compact. 
By Corollary \ref{MVcolim} the fibre of this map is computed pointwise. 
This shows that $\mathrm{fib}$ is zero on $n$-simplices and equal to $\bigoplus_{x \in X_n/G} \Lan_{BG_x}^{\int_G \simp{X}}M(x)$ on simplices of dimension less than $n$, therefore pointwise compact.
By our induction assumption, this shows that $\mathrm{fib}$ is compact in $\MV(X)$ and therefore $M_{X_n}$ is compact in $\MV(X)$.	
\end{proof}

If the stabilizers of the $G$-action are all compact, the description of the compact objects allows us to show that the $t$-structure on $\MV(X)$ descents to a $t$-structure on $\MV(X)^\omega$. This is where the regularity of $R$ comes into play. 
\begin{lemma}
    \label{tcomp}
    If $G$ acts with compact stabilizers the $t$-structure on $\MV(X)$ restricts to a $t$-structure on the full subcategory of compact objects $\MV(X)^\omega$ and the heart of this $t$-structure is Noetherian. 
\end{lemma}

\begin{proof}
    We only need to show that both truncation functors $\tau_{\geq 0}$ and $\tau_{\leq 0}$ preserve compact objects.
    By Lemma \ref{pointcomp} $M \in \MV(X)$ is compact if and only if all restrictions to stabilizers are compact.
    Furthermore, by definition of the $t$-structure on $\MV(X)$ both $\tau_{\geq 0}M$ and $\tau_{\leq 0}M$ are obtained by truncating pointwise.
    But since all $G_x$ are compact, the regularity of $R$ implies that for each $x \in X$ the category $D(G_x)$ is regular by \cite[Theorem 7.2]{BL22}.
    Hence, any finitely presented representation can be resolved by finitely many finitely generated projective representations. 
    In particular the truncations on $D(G_x)$ for $x\in X$ preserves compact objects and restrict to a $t$-structure on $D(G_x)^\omega$. 
    This shows that if $M$ was pointwise compact then $\tau_{\geq 0}M$ and $\tau_{\leq 0}M$ are pointwise compact and therefore by Lemma \ref{pointcomp} compact in $\MV(X)$. 
    Lastly, the heart of the induced $t$-structure is Noetherian for every $x\in X$ the category $D(G_x)^\omega$ is Noetherian again by \cite[Theorem 7.2]{BL22} and $\int_G \simp{X}$ has only finitely many isomorphism classes of objects since the $G$-action on $X$ is cocompact. 
    Hence, if any $M \in \MV(X)^\omega$ would admit an infinite non-trivial chain of subobjects this would induce a non-trivial chain for at least one $ x \in X$ and therefore contradict that $D(G_x)^\omega$ is Noetherian. 
\end{proof}
Lastly, we want to use the results so far to describe how $\alpha$ interacts with the $t$-structure on $\MV(X)^\omega$. 
\begin{lemma}
    \label{alphaconnective}
    Let $X$ be finite-dimensional and let $M\in \MV(X)$ be a Mayer--Vietoris resolution.
    \begin{enumerate}
        \item If $M \in M(X)_{ \geq n}$, then $\alpha(M) \in D(G)_{\geq n}$.
        \item If $M \in \MV(X)_{\leq n}$, then $\alpha(M) \in D(G)_{\leq (n+\dim(X))}$.
    \end{enumerate} 
\end{lemma}
\begin{proof} 
    For both claims observe that the connectivity of $\alpha(M) \in D(G)$ only depends on the underlying object in $D(R)$.
    But by \ref{BeckChav} this can be computed as the colimit of the restriction of $M$ along $\simp{X} \to \int_G \simp{X}$.
    Since by definition of the $t$-structure, the restriction of $M$ is pointwise in $D(R)_{\geq n}$ and the inclusion $D(R)_{\geq n} \to D(R)$ commutes with colimits the first claim follows immediately.
    For the second claim let $m:= \dim (X)$. 
    By the definition of $\simp{X}$, there are at most $m$ non-trivial composable morphisms in $\simp{X}$.
    Hence, all cells of dimension greater than $m$ are degenerate. 
    This implies that $\simp{X}$ is $m$-skeletal, and we can construct $\simp{X}$ in $m$ steps by iteratively attaching $k$-cells in the following way.
\[\begin{tikzcd}
	{\coprod_{\simp{X}_k} \partial \Delta^k} && {\sk{k-1}(\simp{X})} \\
	\\
	{\coprod_{\simp{X}_k} \Delta^k} && {\sk{k}(\simp{X}) }
	\arrow[from=1-1, to=1-3]
	\arrow[from=1-3, to=3-3]
	\arrow[from=1-1, to=3-1]
	\arrow[from=3-1, to=3-3]
\end{tikzcd}\]
    Here, $\simp{X}_k = \Hom_{\mathrm{sSets}}(\Delta^k, \simp{X})$ is the set of $k$-simplices in $\simp{X}$.
    By \cite[Propositions 4.3.26 and 4.3.27]{LaIntro}, this allows us to construct the colimit of a functor $M: \simp{X} \to D(R)_{\leq n}$ in $m$ steps by iteratively taking pushouts as follows. 
\[\begin{tikzcd}
	{\bigoplus_{\simp{X}_k}\colim_{\partial \Delta^k}M} && {\colim_{\sk{k-1}(\simp{X})}M} \\
	\\
	{\bigoplus_{\simp{X}_k} \colim_{\Delta^k}M} && {\colim_{\sk{k}(\simp{X})}M}
	\arrow[from=1-1, to=1-3]
	\arrow[from=1-1, to=3-1]
	\arrow[from=1-3, to=3-3]
	\arrow[from=3-1, to=3-3]
\end{tikzcd}\]
    By induction, we can assume that the claim is true for $k-1$.
    Furthermore, since $\Delta^k$ has a terminal object, the colimit of $M$ restricted to $\Delta^k$ is the evaluation of $M$ at this terminal vertex, in particular, it is in $D(R)_{\leq n}$.
    Then, by the long exact sequence of homotopy groups associated to a pushout, we see that the coconnectivity can increase by at most one. 
    Then the second part of the lemma follows by induction.
\end{proof}
\subsection{Relative resolutions}
There is also a version of $\MV(X)$ relative to a collection of full (co-)complete subcategories $\mathcal{D}_x \subseteq D(G_x)$ for every $x\in X$ that are compatible with the $G$-action.
This can be encoded by a full subcategory of smooth functors 
\[\begin{tikzcd}
	{\int_G X_n} && {D(R).}
	\arrow[from=1-1, to=1-3]
\end{tikzcd}\]
Note that here $\int_G X_n$ are precisely the homotopy orbits $(X_n)_{hG}$ of the $G$-action on the set $X_n$.
\begin{definition}
    An admissible collection is a collection $\mathcal{D} = \lbrace \mathcal{D}_n \rbrace _{n \in \N}$ of full, compactly generated, t-exact subcategories $\mathcal{D}_n$ of $\MV(X_n) \subseteq  \Fun(\int_G X_n, D(R))$ such that all inclusions $\mathcal{D}_n \to \MV(X_n)$ admit a right adjoint and preserve compact objects.
    Furthermore, we require that $\mathcal{D}_{n+1}$ is mapped to $\mathcal{D}_{n}$ by left Kan extension along every face map $\partial_i: X_{n+1} \to X_n$.
\end{definition}

\begin{remark}
    In particular such a collection $\mathcal{D}$ gives us a full subcategory $\mathcal{D}(x):=\mathcal{D}_n(x)$ of $D(G_x)$ for $x \in X_n$ by restricting along $BG_x \to (X_n)_{hG}$.
\end{remark}

\begin{definition}
    Let $\mathcal{D} $ be an admissible collection.
    Then we define $\MV(X, \mathcal{D})$ to be the full subcategory of $\MV(X)$ spanned by all $M \in \MV(X)$ such that $M(x) \in \mathcal{D}(x)$ for all $x\in X$. 
\end{definition}
\begin{lemma}
    \label{tcomprel}
    The $t$-structure on $\MV(X)$ restricts to a $t$-structure on $\MV(X, \mathcal{D})$. Furthermore, if all stabilizers are compact this $t$-structure descents to a $t$-structure on $\MV(X, \mathcal{D})^\omega$.
\end{lemma}
\begin{proof}
    Since all subcategories $\mathcal{D}_n$ are $t$-exact, the truncation functors on $\MV(X)$, which are computed pointwise, restrict to $\MV(X, \mathcal{D})$ and hence induce a $t$-structure. 
    By assumption all inclusions $\mathcal{D}_n \to \MV(X_n)$ preserve colimits and compact objects. 
    This implies that the compact objects in $\mathcal{D}_n$ are precisely the compact objects in $\MV(X_n)$ that lie in the essential image of the inclusion. 
    Because we assumed $R$ to be regular, Lemma \ref{tcomp} shows that the truncations restrict to $\MV(X)^\omega$.
    Hence, both $\tau_{\geq 0}$ and $\tau_{\leq 0}$ restrict to $\MV(X, \mathcal{D})^\omega$.
\end{proof}
Since all $\mathcal{D}(x)$ colocalization and hence (co-)complete and (co-)limits in $\MV(X)$ are computed pointwise we immediately obtain the following corollary.
\begin{corollary}
    The category $\MV(X, \mathcal{D})$ is (co-)complete.
\end{corollary}

The category $\MV(X, \mathcal{D})$ has mostly the same properties as the unrestricted version $\MV(X)$.
\begin{lemma}
    \label{RelLan}
Let $\mathcal{D}$ be an admissible collection. Then $\mathcal{D}(x)$ is mapped to $\MV(X, \mathcal{D})$ by left Kan extension, and compact objects in $\mathcal{D}(x)$ are mapped to compact objects in $\MV(X, \mathcal{D})$.
This provides a set of compact generators and $\MV(X, \mathcal{D})$ is compactly generated.
\end{lemma}
\begin{proof}
    Assume that $V \in D(x)^\omega$.
    We only need to show that $\Lan_{BG_x}^{\int_G \simp{X}}V$ is indeed in $\MV(X, \mathcal{D})$, then the proof is the same as in Lemma \ref{MVcomp}.
    For this we need to compute $(\Lan_{BG_x}^{\int_G \simp{X}}V)(y)$ for every $y \in X$.
    By Corollary \ref{Ranres} we have 
    $$
    (\Lan_{BG_x}^{\int_G \simp{X}}V)(y) = \bigoplus_{ x /BG_y} \ind{G_x^{g-1}}{G_y}V^{g^{-1}}.
    $$
    But by assumption, $\mathcal{D}$ is closed under conjugation and induction, hence $\Lan_{BG_x}^{\int_G \simp{X}}V \in \MV(X, \mathcal{D})$.
\end{proof}

\begin{lemma}
    \label{MVLoc}
    Let $\mathcal{D}$ be an admissible collection.
    Then $\MV(X, \mathcal{D})$ is a right Bousfield localization of $\MV(X)$ with right adjoint $\pr{\mathcal{D}}$.
    The right adjoint can be computed by pointwise applying the right adjoint $D(G_x) \to \mathcal{D}(x)$. 
\end{lemma}
\begin{proof} 
    Since both $\MV(X)$ and $\MV(X, \mathcal{D})$ are compactly generated and the inclusion commutes with colimits the existence of the right adjoint follows from the adjoint functor theorem.
    To explicitly compute the right adjoint let $\pr{\mathcal{D}}: \MV(X) \to \MV(X, \mathcal{D})$ and $\pr{\mathcal{D}(x)} : D(G_x) \to \mathcal{D}(x)$ be the right adjoints to the inclusions.
    By Lemma \ref{RelLan} the following square commutes. 
\[\begin{tikzcd}
	{\mathcal{D}(x)} && {D(G_x)} \\
	\\
	{\MV(X, \mathcal{D})} && {\MV(X)}
	\arrow["\Lan", from=1-3, to=3-3]
	\arrow[from=1-1, to=1-3]
	\arrow[from=3-1, to=3-3]
	\arrow["\Lan"', from=1-1, to=3-1]
\end{tikzcd}\]
    But then the associated square of right adjoints also commutes, hence 
    $$ \res{\int_G \simp{X}}{BG_x} \pr{\mathcal{D}} M \simeq \pr{\mathcal{D}(x)} \res{\int_G \simp{X}}{BG_x}M$$
    for all $M \in \MV(X)$.
\end{proof}

\begin{lemma}
    \label{relcolimit}
    Let $\mathcal{D}$ be an admissible collection such that for all $n$ the right adjoint to the inclusion $\mathcal{D}_n \to \MV(X_n)$ commutes with all colimits.
    Then $\pr{\mathcal{D}}: \MV(X) \to \MV(X, \mathcal{D})$ also commutes with all colimits.
\end{lemma}
\begin{proof}
   Let $F: I \to \MV(X)$ be a diagram in $\MV(X)$ and consider the colimits comparison map 
\[\begin{tikzcd}
	{\colim_I \pr{\mathcal{D}}F(i)} && {\pr{\mathcal{D}}\colim_IF(i).}
	\arrow[from=1-1, to=1-3]
\end{tikzcd}\]
    We need to check that this map is an equivalence.
    To see this we check that it is a pointwise equivalence. 
    The square
\[\begin{tikzcd}
	{\MV(X)} && {\MV(X, \mathcal{D})} \\
	\\
	{D(G_x)} && {\mathcal{D}(x)}
	\arrow[from=1-1, to=3-1]
	\arrow["{\pr{\mathcal{D}}}"', from=1-1, to=1-3]
	\arrow[from=1-3, to=3-3]
	\arrow["{\pr{\mathcal{D}(x)}}", from=3-1, to=3-3]
\end{tikzcd}\]
    is commutative by Lemma \ref{MVLoc}. 
    But the bottom composition commutes with all colimits for all $x \in X$ by assumption.
    This shows that the colimit comparison map is a pointwise equivalence.
\end{proof}

\begin{lemma}
    \label{relmaps}
    Let $f: X \to Y$ be a $G$-equivariant map $ \mathcal{D}^X $ and $ \mathcal{D}^Y$ admissible collections on $X$ and $Y$. If $\mathcal{D}^X$ is mapped to $\mathcal{D}^Y$ by left Kan extension, then $f$ induces a functor 
\[\begin{tikzcd}[ampersand replacement=\&]
	{\MV(X,  \mathcal{D}^X )} \&\& {\MV(Y, \mathcal{D}^Y ).}
	\arrow["{\Lan_f}", from=1-1, to=1-3]
\end{tikzcd}\]
\end{lemma}
\begin{proof}
  Because $\Lan_f: \MV(X) \to \MV(Y)$ is compatible with colimits it suffices to prove the claim for a set of compact generators. 
  But the assumption precisely ensures that 
\[\begin{tikzcd}
	{\mathcal{D}^X(x)} && {\mathcal{D}^Y(f(x))} \\
	\\
	{\MV(X, \mathcal{D}^X)} && {\MV(Y, \mathcal{D}^Y)}
	\arrow["\Lan", from=1-1, to=1-3]
	\arrow[from=1-1, to=3-1]
	\arrow["{\Lan_f}", from=3-1, to=3-3]
	\arrow[from=1-3, to=3-3]
\end{tikzcd}\]
    commutes. In particular the set of compact generators from Lemma \ref{RelLan} is mapped to $\MV(Y, \mathcal{D}^Y)$.
\end{proof}

\subsection{Semi-simplicial structure on trees}
To apply the machinery that we developed so far to Bruhat--Tits trees for reductive groups of rank one, we need to prove that they are always semi-simplicial and that the $G$-action is compatible with the semi-simplicial structure.
This is true for any tree with transitive $G$-action on its edges.
We say that $G$ acts without inversion if for any $g\in G$ and $x\in X$ such that $gx=x$ all faces of $x$ are also fixed by $g$. 
The following lemma shows that we can always assume that $G$ acts without inversion.
\begin{lemma}
    Let $X$ be a tree with $G$-action such that $G$ acts transitive on the set of edges, and assume that the action of $G$ inverts edges.
    Then $G$ acts on the barycentric subdivision of $X$ without inversion and the action is still transitive on edges.
    Furthermore, if $G$ acts with compact stabilizers on $X$, then the action on the barycentric subdivision has compact stabilizers. 
\end{lemma}
\begin{remark}
    The barycentric subdivision splits every edge $e$ into two edges $e_0$ and $e_1$, but the (pointwise) stabilizers do not change. That is $G_e=G_{e_0}=G_{e_1}$. 
    The stabilizer of the barycentre is precisely given by those elements that either fixed or interchanged the vertices of $e$.
\end{remark}
\begin{proof}
   If $g \in G$ interchanges the vertices of an edge $e$ that is split into $e_0$ and $e_1$ by the barycentric subdivision, then $g$ precisely interchanges $e_0$ and $e_1$. 
   By transitivity of the action on $X$, this holds for every edge $e$.
   Therefore, $G$ does still act transitively on edges.
   Furthermore, $G$ can not invert any edges of the subdivision as this would move the barycentre of an edge of $X$ to a vertex of $X$ before subdivision, which does not come from a simplicial map on $X$.
   For the last part, we only have to check that the stabilizer of the barycentre is compact. 
   Let $b$ be the barycentre of any edge $e$. 
   Then $G_b$ are precisely the elements that permute the boundary of the edge and $G_e$ is the subgroup of elements that leave the boundary unchanged. 
   Therefore, $G_e \subseteq G_b$ has index two, where the two cosets can be represented by the identity and any element that interchanges the vertices of the boundary.
   In particular, $G_b$ is compact if $G_e$ is compact.
\end{proof}

\begin{lemma}
    Let $X$ be a tree with $G$-action such that $G$ acts without inversion and transitively on edges.
    Then there is a semi-simplicial set with $G$-action $ssX$ such that the set $ssX_n$ contains precisely the $n$-simplices of $X$ with the induced $G$-action. Furthermore, we have $\partial_ix=f$ for some face map $\partial_i$ if and only if $f$ is a face of $x$ in $X$. 
\end{lemma}
\begin{proof}
    Let $ssX_n$ be the set of $n$-simplices of $X$. 
    We only need to define the corresponding face maps $\partial_i$corresponding. 
    Fix any edge $e$ with stabilizer $G_e$. Then the boundary $\partial e$ has two elements $v_0$ and $v_1$, and we define $\partial_i = v_i$. 
    Since $G$ does act without inversion, this is compatible with the $G_e$ action on $e$, that is $\partial _i(ge)=\partial_ie$ for any $g \in G_e$.
    Let $e'$ be any other edge. 
    By assumption, there exists $g \in G$ such that $ge' = e$.
    Define $\partial_ie' = \partial_i(ge)$.
    This is well-defined since any two $g,g' \in G$ with $ge' = g'e'=e$ only differ by an element of $G_e$ and the latter is compatible with our definition of $\partial_i$. 
    It then follows immediately that this satisfies the desired properties.
\end{proof}
As a consequence, we obtain the following corollary for the Bruhat--Tits tree.
\begin{corollary}
\label{BruhatTitsssSet}
    Let $X$ be the Bruhat--Tits tree of a reductive, algebraic group of rank one.
    Then, after subdividing if $G$ inverts edges, $X$ admits the structure of a semi-simplicial set such that $G$ acts without inversion and compact stabilizers.
\end{corollary}
From now on we will not differentiate between the tree $X$ and its associated semi-simplicial set $ssX$ and in both cases just write $X$.

\section{The \texorpdfstring{$K$}{}-theory of \texorpdfstring{$G$}{}}
The goal of this section is to prove of formula for the $K$-theory spectrum of reductive algebraic groups of rank one as a colimit of compact open subgroups that arise as stabilizers of the corresponding Bruhat--Tits building.

To do this we proceed in two steps.
In the first subsection, we first compute the $K$-theory spectrum of $\MV(X, \mathcal{D})$ for any admissible collection $\mathcal{D}$. 
After this, we show that $D(G)$ is a left Bousfield localization of $\MV(X)$ and use this to construct an exact sequence of presentable, compactly generated stable infinity categories.
This part does not assume anything about the dimension of $X$.

In the second subsection we compute the kernel of the localization $\MV(X) \to D(G)$. 
This will crucially use that $X$ is a tree.
The key ideas of this proof are based on work by Waldhausen published in \cite{WH78a} and \cite{WH78b}.
Without any additional work, the same proof also allows us to compute the kernel of $\MV(X, \mathcal{D}) \to D(G)$ for any admissible collection. This we will need in section $4$.

\subsection{A fibre sequence of K-theory spectra}
We start this section with the computation of the $K$-theory spectrum of $\MV(X, \mathcal{D})^\omega$ for any admissible collection $\mathcal{D}$.
We will always assume that the $G$-action on the semi-simplicial set $X$ is smooth.
\begin{proposition}
    \label{KMV}
    Let $X$ be a finite-dimensional semi-simplicial set with cocompact smooth $G$-action and $\mathcal{D}=\lbrace \mathcal{D}_n \rbrace_{n \in \N}$ an admissible collection.
    Then 
    $$
    K(\MV(X, \mathcal{D})^\omega) \simeq \bigoplus_{[x] \in X/G} K(\mathcal{D}_{(x)}(x)^\omega).
    $$
    In particular 
    $$
    K(\MV(X)^\omega) \simeq \bigoplus_{[x] \in X/G} K(G_x).
    $$

\end{proposition}
\begin{proof}
    The proof is by induction on the dimension of $X$.
    Recall from Lemma \ref{pointcomp} that $M \in \MV(X,  \mathcal{D})$ is compact precisely if all restrictions along $BG_x \to \int_G \simp{X}$ for $x \in X$ are compact in $\mathcal{D}(x) \subseteq D(G_x)$. 
    First, assume that $X$ has dimension $0$. 
    Then $\int_G \simp{X} \simeq X_{hG}$ which decomposes into a disjoint union of $BG_x$ with $x \in X_0/G$.
    If $M \in \MV(X, \mathcal{D})^\omega$ then it decomposes into a direct sum of compact objects in $\mathcal{D}^\omega(x) \subseteq D(G_x)$ and vice versa.
    Hence, the $K$-theory spectrum is a direct sum as claimed.

    Now assume that the claim is true up to dimension $n$ and $X$ has dimension $n+1$. 
    By Lemma \ref{fibseq} the inclusion $\sk{n}X \to X$ induces a fibre sequence 
\[\begin{tikzcd}
	{M_{\sk{n}X}} && M && {M_{X \setminus \sk{n}X}}
	\arrow[from=1-1, to=1-3]
	\arrow[from=1-3, to=1-5]
\end{tikzcd}\]
    for every $M \in \MV(X)$ and if $M \in \MV(X, \mathcal{D})$ then the fibre sequence is in $\MV(X, \mathcal{D})$. By \ref{pointcomp} the middle term is compact if and only if the outer terms are compact.
    The objects of the form $M_{X \setminus \sk{n}X}$ are precisely the resolutions that are only supported on $X_{n+1}$. 
    This induces a Verdier sequence 
\[\begin{tikzcd}
	{\MV(\sk{n}X, \mathcal{D}\vert \sk{n}X)} && {\MV(X, \mathcal{D})} && {\MV(X_{n+1}, \mathcal{D} \vert X_{n+1}).}
	\arrow["\Lan", from=1-1, to=1-3]
	\arrow["{\res{}{}}", from=1-3, to=1-5]
\end{tikzcd}\]
    Furthermore the left Kan extension $\MV(\sk{n} X) \to \MV(X)$ has a right adjoint given by restriction that also preserves compact objects by \ref{pointcomp}.
    Therefore, the restriction functor $\MV(X, \mathcal{D} ) \to \MV(X_{n+1}, \mathcal{D} \vert X_{n+1})$ also has a right adjoint that preserves compact objects by \cite[Lemma A.2.5.]{calmes2020hermitian}. 
    Hence, the sequence is split, and by additivity
    $$
    K(\MV(X, \mathcal{D})^\omega) \simeq K(\MV(\sk{n}X, \mathcal{D}\vert \sk{n}X)^\omega) \bigoplus K(\MV(X_{n+1}, \mathcal{D} \vert X_{n+1})^\omega).
    $$ 
    Now $\sk{n}X$ and $X_{n+1}$ are both of dimension less or equal than $n$ and the claim follows by the induction hypothesis.
\end{proof}

An object of $\MV(X)$ can be thought of as a collection of complexes of representations of compact open subgroups parametrized over $X$. 
These representations can then be glued together, taking the geometry of $X$ into account, to obtain an object of $D(G)$.
This gluing process is formalized by the induced map of $X \to *$ on $\MV(-)$ where the latter carries the trivial action. 
If the underlying space of $X$ is contractible we can obtain all objects of $D(G)$ this way as the next proposition shows.
\begin{proposition}
    \label{MVBous}
    Let $X$ be a $G$-semi-simplicial set such that the underlying semi-simplicial set is contractible.
    Then the functor $\alpha: \MV(X) \to D(G)$ induced by the $G$-map $X \to *$ is a left Bousfield localization.
\end{proposition}
\begin{proof}
   Recall that the functor is given by the left Kan extension along $p:\int_G \simp{X} \to BG$.
   This has a right adjoint given by restriction. 
   Hence, we only need to show that $\res{p}{}: D(G) \to \MV(X)$ is fully faithful or equivalently that the counit $\Lan_p \res{p}{} \to \id$ of the adjunction is an equivalence.
   The forgetful functor $D(G) \to D(R)$ is conservative because it is given by evaluation an object of $D(G) \simeq Fun^\infty(BG, D(R))$ 
   on the single object of $BG$ and a natural transformation in is an equivalence if and only if it is an equivalence on every object.
   Therefore, it suffices to show that the counit is an equivalence after forgetting the $G$-action, i.e. restricting along $* \to BG$.
   The functor $p: \int_G \simp{X} \to BG$ is proper because it is the composition \cite[Proposition 11.4]{JoyalL} of the left fibration \cite[Theorem 11.9]{JoyalL} given by the unstraightening of $X: \Delta^{op}_{inj} \times BG \to Sets $ and the projection $\Delta^{\op}_{inj}\times BG \to BG$ \cite[Corollary 11.5]{JoyalL}. 
   Hence, by \ref{BeckChav} the pullback square 
\[\begin{tikzcd}
	{\simp{X}} && {\int_G \simp{X}} \\
	\\
	{*} && BG
	\arrow["p", from=1-3, to=3-3]
	\arrow[from=3-1, to=3-3]
	\arrow[from=1-1, to=1-3]
	\arrow["{p'}"', from=1-1, to=3-1]
\end{tikzcd}\]
    satisfies the Beck--Chevalley condition after applying $\Fun(-, D(R))$.
    This shows that
\[\begin{tikzcd}
	{\Lan_{p'} \res{p'}{} \res{*}{} \simeq\res{*}{} \Lan_p \res{p}{}} && {\res{*}{}}
	\arrow[from=1-1, to=1-3]
\end{tikzcd}\] 
    where $\Lan_{p'} \res{p'}{} \res{*}{} \to \res{*}{}$ is given by the counit $\Lan_{p'} \res{p'}{} \to \id$ applied to $\res{*}{}$.
    Therefore, we only need to check that  $\Lan_{p'} \circ \res{}{} \to \id$ is an equivalence.
    But left Kan extension to the point is precisely taking the colimit, and we only need to check that for every $V \in D(G)$ the canonical map 
\[\begin{tikzcd}
	{\colim _{\simp{X}}V} && V
	\arrow[from=1-1, to=1-3]
\end{tikzcd}\]
    is an equivalence. 
    This is induced by tensoring $V$ with the map of anima 
\[\begin{tikzcd}
	{\colim_{\simp{X}} *} && {*}.
	\arrow[from=1-1, to=1-3]
\end{tikzcd}\]
    But $\colim_{\simp{X}} *\simeq X$ and the map is an equivalence in $\An$, hence it induces an equivalence after tensoring with $V$.
 \end{proof}
 The previous proposition allows us to construct an exact sequence of presentable stable $\infty$-categories.
 \begin{definition}
    \label{DefMV0}
    Define $\MV_0(X, \mathcal{D})$ to be the kernel of $\MV(X, \mathcal{D}) \to D(G)$. 
 \end{definition}
 \begin{remark}
    We will later prove \ref{Kercompg} that $\MV_0(X, \mathcal{D})$ is again compactly generated and therefore in $Pr^{L}_{st, \omega}$. 
    In particular, we have an exact sequence 
    $$
    \begin{tikzcd}
        \MV_0(X) \arrow[r] & \MV(X) \arrow[r] & D(G)
    \end{tikzcd}
    $$
    in $Pr^L_{st, \omega}$.

    In general, for an admissible collection $\mathcal{D}$ we do not have a replacement for the Bousfield localization on the right-hand side in general. 
    If we define $D(G, \mathcal{D})$ to be the essential image of $\MV(X, \mathcal{D}) \to D(G)$ we again obtain an adjunction $\alpha \dashv \pr{\mathcal{D}} \circ \res{}{}$ with $\pr{\mathcal{D}}$ as in Lemma \ref{MVLoc}. 
    However, in general, the right adjoint does not need to be fully faithful.
\end{remark}
The above sequence induces a short exact sequence of $K$-theory spectra of the categories of compact objects in the following way.
\begin{remark}
    \label{RmPrComp}
    There is an equivalence of categories $Pr^{L }_{\omega, st}$ and $Cat^{\mathrm{perf}}_{\mathrm{ex}}$ between presentable, compactly generated, stable $\infty$-categories with colimit preserving functors that preserve compact objects and essentially small, idempotent complete, stable $\infty$-categories with exact functors, induced by taking compact objects and $\mathrm{Ind}$-completion \cite[5.5.7.8]{HTT} and \cite[Theorem A.3.11]{calmes2020hermitian}.
    Using this, we can apply $(-)^\omega$ to the above sequence to obtain an exact sequence of perfect small stable $\infty$-categories. 
    $K$-theory sends exact sequences of small stable $\infty$-categories to bifibre sequences of spectra, therefore we obtain the bifibre sequence
    $$
    \begin{tikzcd}
        \K{\MV_0(X)^\omega} \arrow{r} & \K{\MV(X)^\omega} \arrow{r} & \K{G}.
    \end{tikzcd}
    $$
\end{remark}

\subsection{Computation of the kernel of \texorpdfstring{$\alpha$}{}}
From now on we will assume that $X$ is a tree and that $G$ acts smoothly and transitively on edges. 
In addition, we will assume that all stabilizers are compact.
By Corollary \ref{BruhatTitsssSet} $X$ admits the structure of a semi-simplicial set such that these assumptions are satisfied, and we can apply the machinery we have developed so far.
We will show that $\MV_0(X, \mathcal{D})$ is compactly generated and prove the main result of this section which is the following theorem. 
\begin{theorem}
    \label{RelKer}
    Let $G$ be a reductive, algebraic group with a smooth action on a tree $X$ with compact stabilizers, that is transitive on edges and without inversion. Let $\mathcal{D}$ be an admissible collection.
    Then $K(\MV_0(X, \mathcal{D})^\omega) \simeq K(\mathcal{D}(e)^\omega) \bigoplus K(\mathcal{D}(e)^\omega)$ for any fixed edge $e\in X$.
\end{theorem}
The arguments in this section will rely on the existence of a bounded $t$-structure on $\MV_0(X, \mathcal{D})^\omega$.
In particular, we want to apply the (non-connective) Theorem of the Heart and D\'evissage.
\begin{remark}
    The right adjoint of the functor $\alpha: \MV(X) \to D(G)$ preserves filtered colimits. 
    This implies that the right adjoint of the inclusion also preserves filtered colimits since it can be constructed as the fibre of the unit $\id \to \res{}{} \circ \alpha$, see \cite[Lemma A.2.5]{calmes2020hermitian}.
    Hence, the compact objects $\MV_0(X, \mathcal{D})^\omega$ are precisely the compact objects of $\MV(X, \mathcal{D})$ that are in the kernel of $\alpha$.
    In particular, we can use Lemma \ref{pointcomp} to check if an object is compact.
\end{remark}
\begin{lemma}
    \label{tker}
    Let $\mathcal{D}$ be an admissible collection.
    The $t$-structure on $\MV(X, \mathcal{D})$ restricts along the inclusion $\MV_0(X, \mathcal{D}) \to \MV(X, \mathcal{D})$ to a $t$-structure on $\MV_0(X, \mathcal{D})$. 
    In addition, this restricts further to a bounded $t$-structure on  $\MV_0(X, \mathcal{D})^\omega$ such that the heart $\MV_0(X, \mathcal{D})^{\omega, \heartsuit}$ is Noetherian.
\end{lemma}
\begin{proof}
    For the first part, we only need to show that the truncation functors $\tau_{\geq 0}$ and $\tau_{\leq 0}$ restrict to the kernel.
    Let $M \in \MV_0(X, \mathcal{D})$. 
    Then we have a bifibre sequence 
\[\begin{tikzcd}
	{M_{\geq 0}} && M && {M_{<0}}
	\arrow[from=1-3, to=1-5]
	\arrow[from=1-1, to=1-3]
\end{tikzcd}\]    
    in $\MV(X, \mathcal{D})$ and hence a fibre sequence 
\[\begin{tikzcd}
	{\alpha(M_{\geq 0})} && {\alpha(M)} && {\alpha(M_{<0}).}
	\arrow[from=1-3, to=1-5]
	\arrow[from=1-1, to=1-3]
\end{tikzcd}\]
    Since $\alpha(M)\simeq 0$ by assumption, we see that $\pi_{i+1}(\alpha(M_{<0})) \cong \pi_i(\alpha(M_{\geq 0 }))$.
    But by \ref{alphaconnective} we know that $\alpha(M_{\geq 0}) \in D(G)_{\geq 0}$.
    Furthermore, $\alpha(M_{<0})\in D(G)_{\leq 0}$ again by \ref{alphaconnective}.
    Then the above isomorphism of homotopy groups shows that $\alpha (M_{<0})=\alpha(M_{\geq 0})=0 $. 
    The claim that this restricts to a $t$-structure on compact objects follows immediately since both truncations preserve compact objects by $\ref{tcomprel}$. 
    Furthermore, since the $t$-structure on $\MV(X, \mathcal{D})^\omega$ is bounded we again have 
    $$
    \MV_0(X, \mathcal{D})^\omega = \bigcup_{a \leq b} \MV_0(X, \mathcal{D})^\omega_{[a,b]}.
    $$
    Lastly $\MV_0(X, \mathcal{D})^{\omega, \heartsuit}$ is Noetherian since it is a full subcategory of $\MV(X, \mathcal{D})^{\omega, \heartsuit}$.
    The latter is Noetherian since all $D(x)^{\omega, \heartsuit}$ are noetherian.
\end{proof}

\begin{remark}
    If the dimension of $X$ is greater than one it is not clear that the $t$-structure on $\MV(X, \mathcal{D})^\omega$ restricts to a (bounded) $t$-structure on $\MV_0(X, \mathcal{D})^\omega$. 
    In fact the existence of a bounded $t$-structure on $\MV_0(X, \mathcal{D})^\omega$ implies by an argument by Schlichting, adapted to stable $\infty$-categories in \cite{AGH19}, that $K_{-1}(\MV_0(X, \mathcal{D})^\omega)=0$. 
    Then the long exact sequence of homotopy groups yields that the Farrell--Jones assembly map with respect to the compact open subgroups of $G$ is surjective on $K_0$ which is highly non-trivial.
    However, as recently shown by Bartels and Lück in \cite{bartels2023algebraic}, the assembly map is an equivalence.
    Therefore, the surjectivity on $K_0$ is not an obstruction and the $t$-structure might still restrict to $\MV_0(X, \mathcal{D})^\omega$.
\end{remark}
Before we proceed we make the following observation. For any edge $e$ the complement $X \setminus e$ has two path components, each containing exactly one vertex of $e$.
If $x$ is a vertex of $e$, let $\Gamma_x(e)$ be the path component of $X \setminus e$ containing $x$.
Then $\Gamma_x(e)$ inherits an action of $G_e$ because $G$ acts without inverting edges. 
Furthermore, denote by $\Gamma_x^n(e)$ the intersection of the ball of radius $n$ around $x$ and $\Gamma_x(e)$. 
Here the distance between two vertices is the number of edges between them.
The distance between an edge $e$ and a vertex $v$ is the maximum of the distances between the vertices of $e$ and $v$. 
We obtain a chain of inclusions with $\colim_n \Gamma_x^n(e) = \Gamma_x(e)$. 

\begin{lemma}
    \label{FiltGamma}
    Let $M \in \MV(X, \mathcal{D})$ and let $e$ be an edge of $X$ with vertex $x$. Then 
    $\colim_n M_{\Gamma_x^n(e)} = M_{\Gamma_x(e)}$ in $\MV_{G_e}(X, \mathcal{D})$.
\end{lemma}
\begin{proof}
    If $M \in MV(X, \mathcal{D})$, then $M_{\Gamma^n_x(e)} \in \MV_{G_e}(X, \mathcal{D})$. 
    Since by \ref{MVLoc} colimits in $\MV_{G_e}(X, \mathcal{D})$ can be computed in $\MV_{G_e}(X)$ it suffices to only consider the case $M \in \MV(X)$.
    The full inclusion $\MV(X) \to \Fun(\int_G \simp{X}\op, D(R))$ is a left adjoint \ref{smoothRes} and hence commutes with colimits.
    Therefore, we can compute the colimit of the diagram in $\Fun(\int_G\simp{X}^{op}, D(R))$.
    This is done pointwise.
    For any simplex in $\Gamma_x(e)$ there is an integer $N\geq 0$ such that $\Gamma_x^N(e)$ already contains the given simplex. Since the subcomplexes are contained in each other, this is also true for all $n\geq N$. Hence, for any simplex of $\Gamma_x(e)$ the colimit diagram becomes constant after finitely many stages and the value agrees with the value of $M_{\Gamma_x(e)}$. For any simplex outside $\Gamma_x(e)$ the same is also true since both sides are zero.
    This shows that $\colim_n M_{\Gamma_x^n(e)} = M_{\Gamma_x(e)}$. 
\end{proof}
Next, we introduce a decomposition of $M(e)$ for every $M \in \MV(X, \mathcal{D})$ given an edge $e \in X_1$.
Consider the chain of inclusions from before
\[\begin{tikzcd}
	{x = \Gamma^0_x(e)} && {\Gamma^1_x(e)} && {\Gamma^2_x(e)} && \ldots .
	\arrow[from=1-1, to=1-3]
	\arrow[from=1-3, to=1-5]
	\arrow[from=1-5, to=1-7]
\end{tikzcd}\]
This induces a diagram in $D(G_e)$ 
\begin{equation}    
    \label{NilFilt}
\begin{tikzcd}[ampersand replacement=\&]
	{M(e)} \\
	\\
	{M(\Gamma_x^0(e))} \&\& {M(\Gamma_x^1(e))} \&\& {M(\Gamma^2_x(e))} \&\& {\ldots }
	\arrow[from=3-1, to=3-3]
	\arrow[from=3-3, to=3-5]
	\arrow[from=3-5, to=3-7]
	\arrow[from=1-1, to=3-1]
	\arrow[from=1-1, to=3-3]
	\arrow[from=1-1, to=3-5]
\end{tikzcd}
\end{equation}
where the vertical maps are given by the map $M(e) \to M(x)$ followed by $M(x) \to M(\Gamma^n_x(e))$. Since $\alpha$ is compatible with colimits, the colimit of the lower row is $M(\Gamma_x(e))$.
For any edge $e$ we obtain by Lemma \ref{fibseq} a bifibre sequence of $G_e$-equivariant resolutions
\[\begin{tikzcd}[ampersand replacement=\&]
	{M_{\Gamma_{\partial_0e}(e)} \bigoplus M_{\Gamma_{\partial_1e}(e)}} \&\& M \&\& {M_e}.
	\arrow[from=1-1, to=1-3]
	\arrow[from=1-3, to=1-5]
\end{tikzcd}\]
From now on assume that $M \in \MV_0(X, \mathcal{D})^\heartsuit$. Then $\alpha(M) \simeq 0$ and the above sequence induces an equivalence 

\begin{equation}
\label{KerDecomp}   
\begin{tikzcd}[ampersand replacement=\&]
	\alpha(M_e)[1]\simeq {M(e)} \&\& {M(\Gamma_{\partial_0e}(e))\bigoplus M(\Gamma_{\partial_1e}(e)).}
	\arrow[from=1-1, to=1-3]
\end{tikzcd}
\end{equation}
In particular, the right-hand side is also in $\MV(X, \mathcal{D})^\heartsuit$.
Set $M_i(e) := \ker (M(e) \to M(\Gamma_{\partial_ie}(e))) \in \mathcal{D}(e)^\heartsuit \subseteq D(G_e)^\heartsuit = \rep{G_e}$.
Then $M(e) = M_0(e) \bigoplus M_1(e)$ and if $M(e)$ is compact then $M_0(e)$ and $M_1(e)$ are also compact.  
In that case  
\begin{equation*}
    \Map (M_i(e) , M(\Gamma_{\partial_ie}(e))) \simeq \colim_n \, \Map(M_i(e), M(\Gamma_{\partial_ie}^n(e))).
\end{equation*}
A property of filtered colimits in anima is that points that lie in the same path component of the colimit are already in the same path component at some finite stage.
This implies that already one of the compositions of the inclusion $M_i(e) \to M(e)$ with the vertical maps of \eqref{NilFilt} has to be homotopic to zero (and hence zero since $\Map(M_i(e), M(\Gamma^n_{\partial_ie}(e)))$ is discrete) at some finite stage. 
This allows us to define a chain of full subcategories of $\MV_0(X, \mathcal{D})^{\omega, \heartsuit}$.
\begin{definition}
For $n \geq 0$ let $\mathcal{C}(\mathcal{D})^n$ be the full subcategory of $\MV_0(X, \mathcal{D})^{\heartsuit}$ spanned by $M\in \MV(X, \mathcal{D})^{\heartsuit}$ such that for some edge $e$ (and hence every edge) both maps 
$M_i(e) \to M(\Gamma^n_{\partial_ie}(e))$ for $i=0,1$ are zero.
\end{definition}
\begin{remark}
    The previous argument also shows that on compact objects we have an equivalence $\colim_n\,  \mathcal{C}(\mathcal{D})^{n, \omega} = \MV_0(X, \mathcal{D})^{ \heartsuit, \omega}$. 
\end{remark}
Before we proceed with the computation of the $K$-theory spectrum we need the following technical lemma.
\begin{lemma}
    \label{lemmadiff}
    The map $M(e) \to M(\Gamma^n_{\partial_ie}(e))$ in \eqref{KerDecomp} factors into the composition of the structure map $M(e) \to M(\partial_ie)$ and the map induced from the inclusion $M(\partial_ie) \to M(\Gamma^n_{\partial_ie}(e))$.
\end{lemma}
\begin{proof}
    It suffices to show the case $n=0$, the other cases follow by post composing with the inclusion to $M(\Gamma^n_{\partial_ie}(e))$.
    For this consider the fibre sequence 
\[\begin{tikzcd}[ampersand replacement=\&]
	{M({\partial_ie})} \&\& {M({e \cup \partial_ie})} \&\& {M(e)}.
	\arrow[from=1-1, to=1-3]
	\arrow[from=1-3, to=1-5]
\end{tikzcd}\]
 By (the $G_e$-equivariant version of) \ref{AssemFormula} the middle complex is equivalent to $M(e) \to M(\partial_ie)$ with differential given by the structure map. With that we obtain 
\[\begin{tikzcd}[ampersand replacement=\&]
	0 \&\& {M(e)} \&\& {M(e)} \\
	{M(\partial_ie)} \&\& {M(\partial_ie)} \&\& 0
	\arrow["{\partial_i}", from=1-3, to=2-3]
	\arrow[from=1-1, to=2-1]
	\arrow[from=1-1, to=1-3]
	\arrow["\id"{description}, from=2-1, to=2-3]
	\arrow["\id"{description}, from=1-3, to=1-5]
	\arrow[from=2-3, to=2-5]
	\arrow[from=1-5, to=2-5]
\end{tikzcd}\]
and taking cofibre once yields
\[\begin{tikzcd}[ampersand replacement=\&]
	{M(e)} \&\& {M(e)} \&\& {M(\partial_ie)} \\
	{M(\partial_ie)} \&\& 0 \&\& 0.
	\arrow["{\partial_i}", from=1-1, to=2-1]
	\arrow["\id"{description}, from=1-1, to=1-3]
	\arrow[from=2-1, to=2-3]
	\arrow[from=1-3, to=2-3]
	\arrow["{\partial_i}"{description}, from=1-3, to=1-5]
	\arrow[from=2-3, to=2-5]
	\arrow[from=1-5, to=2-5]
\end{tikzcd}\] 
Hence we see that the differential in the cofibre sequence is precisely given by the structure map.
\end{proof}
Now we can define a filtration for every $M \in \MV_0(X, \mathcal{D})^{\omega, \heartsuit}$ that will later allow us to apply Quillen's D\'evissage theorem.
Note that the heart $\MV(X, \mathcal{D})^\heartsuit$ is an (ordinary) Abelian category, therefore we do not need to worry about coherencies when constructing the filtration.
\begin{construction}
    Fix $M \in \MV_0(X, \mathcal{D})^{\heartsuit}$. 
    We use the above decomposition to define a filtration $M^n$ of $M$ by first specifying the value of $M^n$ on every simplex for every $n \geq 0$.
    By definition $M_i(e)$ is the kernel of $M(e) \to M(\Gamma_{\partial_ie}(e))$.
    Analogously define $M_i^n(e) := \ker (M(e) \to M(\Gamma_{\partial_ie}^n(e)))$ and $M^n(e):= M_0^n(e) \bigoplus M^n_1(e)$. 
    For a vertex $v \in X$, let $B_{n+1}(v)$ be the ball of radius $n$ around $v$ and define 
    $M^n(v):= \ker (M(v) \to M(B_{n+1}(v)))$.  
    Note that these are again in $\mathcal{D}(e)^\heartsuit$ resp.\ $\mathcal{D}(v)^\heartsuit$ and compact if $M$ was compact. 
    Next, we check that the structure maps of $M$ restrict to $M^n(e)$ and $M^n(v)$ and therefore define a resolution $M^n$.
    \begin{lemma}
        \label{FiltKer}
        For any morphism $e \to v$ in $\int_G \simp{X}$ the induced map $M(e) \to M(v)$ restricts to a map $M^n(e) \to M^n(v)$.
        In particular $M^n$ with the restricted maps is in $\MV(X, \mathcal{D})^{\heartsuit}$.
        If $M$ was compact, then $M^n \in \MV(X, \mathcal{D})^{\heartsuit, \omega}$. 
        Furthermore, there are maps $M^n \to M^{n+1}$ such that $\colim \, M^n = M$.
    \end{lemma}
    \begin{proof} 
    Since any morphism in $\int_G \simp{X}$ can be decomposed into an inverse face inclusion $e \to v$ and the action by a group element it suffices to treat this cases separately.
    For the latter case the claim follows directly since $g\Gamma^n_{\partial_ie}(e) = \Gamma^n_{\partial_ige}(ge)$ and $gB_{n+1}(v)=B_{n+1}(gv)$ for every $g \in G$. 
    For the case of an inverse inclusion let $e$ be an edge with vertex $\partial_0e=v$.
    By Lemma \ref{lemmadiff} we have the following 
    commutative square
\begin{equation}
    \label{incFilt}
\begin{tikzcd}
	{M_0^n(e) \oplus M_1^n(e)} & {M(\Gamma_v^n(e))} \\
	{M(v)} & {M(B_{n+1}(v))}
	\arrow[from=1-2, to=2-2]
	\arrow[from=2-1, to=2-2]
	\arrow["{\partial_e^v}"', from=1-1, to=2-1]
	\arrow["{}", from=1-1, to=1-2]
\end{tikzcd}
\end{equation}
    where $\partial_e^v$ is the map that $M$ assigns to the inverse inclusion $e \to v$ and the horizontal map on the top is again given by \eqref{KerDecomp}.
    By definition $M_0^n(e)$ is mapped to zero under the upper map, hence $\partial_e^vM_0^n(e) \subseteq M^n(v)$.
    For $M_1^v(e)$ let $B_{n}(e) = \Gamma_{\partial_01}^n(e) \cup e \cup \Gamma_{\partial_1e}^n(e)$ and consider the associated fibre sequence 
\[\begin{tikzcd}
	{M(e)} &&{M(\Gamma_0^n(e))\bigoplus M(\Gamma_1^n(e))} && {M(B_{n}(e))} .
	\arrow[from=1-3, to=1-5]
	\arrow[from=1-1, to=1-3]
\end{tikzcd}\] 
    This shows that booth compositions $M(e) \to  M(\Gamma_{\partial_0e}^n(e)) \to M(B_{n}(e))$ and $M(e) \to  M(\Gamma_{\partial_1e}^n(e)) \to M(B_{n}(e))$ agree (up to a sign).
    By post composing with the inclusion $M(B_n(e)) \to M(B_{n+1}(v))$ we see that the upper composition in \ref{incFilt} is equivalent to the composition $M(e) \to M(\Gamma_{\partial_1e}^n(e)) \to M(B_{n+1}(v))$ and hence also 
    zero on $M_1^n(e)$. This shows that the structure maps of $M$ restrict to $M^n$ and this defines a resolution.
    Furthermore, the inclusions of kernels induce maps $M^n \to M^{n+1}$ for every $n \geq 0$. 
    Since any $x \in M_i(e)$ lies in the kernel of the map to $M(\Gamma_{\partial_ie}(e))$ it has to be in the kernel at some finite stage $M(\Gamma^n_{\partial_ie}(e))$ and therefore in $M_i^n(e)$.
    The same argument holds for vertices. 
    This shows that $\colim _n\,  M^n = M$. 
\end{proof}

\end{construction}

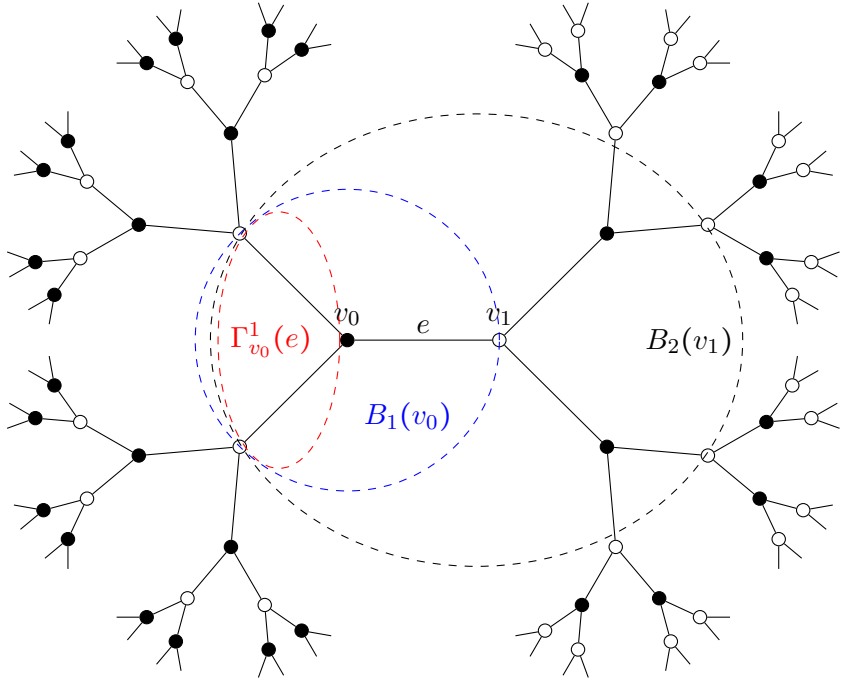
\begin{figure}

\centering
\begin{tikzpicture}
   [
   grow cyclic,
   level distance=2cm,
   nodes={circle,draw,inner sep=+0pt, minimum size=5pt},  
   level/.style={
   level distance/.expanded=\ifnum#1>1 \tikzleveldistance/1.5\else\tikzleveldistance\fi},
   level 1/.style={sibling angle=90},
   level 2/.style={sibling angle=80},
   level 3/.style={sibling angle=70},
   level 4/.style={sibling angle=50},
 ]
 \begin{scope}[xshift=2cm]
   \node[label=above:$v_1$] (A) {}
   child foreach \cntI in {1,...,2} {
     node[fill, color=black] {}
     child foreach \cntII in {1,...,2} { 
       node {}
       child foreach \cntIII in {1,...,2} {
         node[fill, color=black] {}
         child foreach \cntIV in {1,...,2} {
           node {}
           child foreach \cntV in {1,...,2} {}
         }
       }
     }
   };
 \end{scope}

 \begin{scope}[rotate=180]

   \node[fill, color=black,label=above:$v_0$] (B) {}
   child foreach \cntVI in {1,...,2} {
     node {}
     child foreach \cntVII in {1,...,2} { 
       node[fill, color=black] {}
       child foreach \cntVIII in {1,...,2} {
         node {}
         child foreach \cntIX in {1,...,2} {
           node[fill, color=black] {}
           child foreach \cntX in {1,...,2} {}
         }
       }
     }
   };
 \end{scope}

 \draw (A) -- (B);
 \node[style={draw=none},align=center, below] at (1,0.3) {$e$};

 \draw[black, dashed] (1.7,0) ellipse (3.5cm and 3cm);
 \draw[red, dashed] (-0.9,0) ellipse (0.8cm and 1.7cm);

 \draw[blue, dashed] (B) circle (2cm);

 \node[style={draw=none},align=center, below, color=red] at (-1, 0.6) {$\Gamma^1_{v_0}(e)$};

\node[style={draw=none}, align=center, below, color=blue] at (0.8,-0.4) {$B_1(v_0)$};
  \node[style={draw=none},align=center, below, color=black] at (4.5, 0.6) {$B_2(v_1)$};
\end{tikzpicture}
\caption{$\Gamma_{v_0}^1(e)$, $B_1(v_0)$ and $B_2(v_1)$ for $\mathrm{SL}_2(\Q_2)$.}
\end{figure}

\begin{lemma}
    \label{DevReq}
    
    The resolution $M^n$ is in the kernel of $\alpha$ for every $n \geq 0$.
    Furthermore, the cofibre of the map 
    $ M^n \to M^{n+1}$ is in $\mathcal{C}^0(\mathcal{D})$ and compact if $M$ was compact. 
\end{lemma}
\begin{proof}
    By \ref{AssemFormula} the underlying complex of $\alpha(M)$ can be realized as
    
    \[\begin{tikzcd}
    	{\bigoplus_{e \in X_1}M(e)} && {\bigoplus_{v \in X_0}M(v)}
	    \arrow[from=1-1, to=1-3]
    \end{tikzcd}\] 
    and the map $\alpha( M^n) \to \alpha(M)$ is then given by the inclusions of submodules. 
    The latter is contractible by assumption.
    From that, it follows by a short diagram chase that $H_1(\alpha(M^n))=0$. 

    Let $v \in X_0$ and let $S(v)=\lbrace e \in X_1 \, \vert \, v\subseteq e \rbrace$ be the set of edges neighbouring $v$. Then, by \ref{fibseq}, we have a fibre sequence 
\[\begin{tikzcd}[ampersand replacement=\&]
	{\bigoplus_{e \in S(v)} M(e)} \&\& {M(v) \bigoplus M(B_{n+1}(v)\setminus \overset{\circ}{B}_1(v))} \&\& {M(B_{n+1}(v))} .
	\arrow[from=1-1, to=1-3]
	\arrow[from=1-3, to=1-5]
\end{tikzcd}\]
If $x\in M(v)$ is in the kernel of the map to $M(B_{n+1}(v))$, then by exactness it has a preimage  $y \in \bigoplus _{e \in S(v)}M(e)$ and $y$ is mapped to $0$ in $M(B_{n+1}(v) \setminus \overset{\circ}{B}_1(v))$. But $B_{n+1}(v) \setminus \overset{\circ}{B}_1(v)$ decomposes into $\coprod_{e \in S(v)}\Gamma^n_{\partial^{-v}}(e)$ where $\partial^{-v}e$ is the vertex of $e$ that is not $v$.
This precisely means that $y$ already lies in $\bigoplus _{e \in S(v)}M^n(e)$.
In particular, the differential is also surjective therefore $\alpha(M^n)\simeq 0$. 

Lastly we show that the cofibre of $M^n \to M^{n+1}$ is in $\mathcal{C}(D)^0$.
Because $M^n$ and $M^{n+1}$ are in the kernel of $\alpha$, the cofibre lies in the kernel as well.
Since degree-wise injective maps are cofibrations, we can compute the cofibre by taking the degree-wise cokernel, and we only need to check that $x \in M_i^{n+1}(e)$ is mapped to $\ker (M(\partial_ie) \to M(B_{n+1}(\partial_ie)))$ under the map induced from $e \to \partial_ie$.
But since $\Gamma_{\partial_ie}^{n+1}(e) \subseteq B_{n+1}(\partial_ie)$ this follows immediately.
\end{proof}

As a final step, we show that the category $\mathcal{C}(\mathcal{D})^0$ is equivalent to $\mathcal{D}(e)^\heartsuit \oplus \mathcal{D}(e)^\heartsuit$.
Let $X^{\mathrm{bary}}$ be the barycentric subdivision of $X$ with the induced $G$-action.
If $\mathcal{D}$ is an admissible collection on $X$, we can extend it to an admissible collection on $X^{\mathrm{bary}}$ by assigning the category $\mathcal{D}(e)$ to the barycentre of an edge $e$.

\begin{definition}
    Let $\mathcal{B}(\mathcal{D})^0$ be the full subcategory of $\MV_0(X^{\mathrm{bary}}, \mathcal{D})^{\heartsuit}$ spanned by the resolutions that vanish on the barycentres.
\end{definition}

\begin{lemma}
    \label{compc0}
    There is an equivalence of categories $F:\mathcal{B}(\mathcal{D})^{0} \to \mathcal{C}(\mathcal{D})^{0}$ such that $F(M)(e)=M(e_0) \oplus M(e_1)$ 
    for any $M \in \mathcal{B}(\mathcal{D})^0$ and edge $e \in X$ that is split into edges $e_0$ and $e_1$ with $\partial_1e \subset e_0$ and $\partial_0e \subset e_1$ by the barycentric subdivision and the identity on vertices. 
\end{lemma}
\begin{proof} 
    For any $M \in \mathcal{B}(\mathcal{D})^0$ we define the resolution $F(M)$ on for a vertex $v \in X$ to be $F(M)(v):=M(v)$ and for an edge $e \in X$, that is subdivided into $e_0$ and $e_1$, to be $F(M)(e):=M(e_0) \oplus M(e_1)$. 
    For $g \in G$ with face map $g e \to \partial_ige$
    define $F(M)(e) \to b(M)(\partial_ige)$ to be equal to $M(e_j) \to M(\partial_ige)$ for $j \in \lbrace 0,1 \rbrace $ with $j \neq i$ and zero otherwise.
    
    To see that this is an equivalence we also construct a functor $G: \mathcal{C}(\mathcal{D})^{0} \to \mathcal{B}(\mathcal{D})^0$ in the other direction.
    For $M \in \mathcal{C}(\mathcal{D})^0$ on vertices we define $G(M(v)):=M(v)$ for $v \in X$ and zero on the barycentres.
    On edges, define $G(M)(e_i):=\ker( M(e) \to M(\partial_ie))$.
    The value of $G(M)$ on morphisms into vertices of $X$ is the restriction of the corresponding map in $M$ and zero if we map into the barycentre.
    It then follows immediately that $G \circ F = \id$.
    For the other direction, note that for $M \in \mathcal{C}(\mathcal{D})^{0}$ we always have $M(e) = \ker ( M(e) \to M(\partial_0)) \oplus \ker ( M(e) \to M(\partial_1))$ by \eqref{KerDecomp} and the definition of $\mathcal{C}(\mathcal{D})^{0}$.
\end{proof}
Let $B \subseteq \int_G \simp{X^{\mathrm{bary}}}$ be the full subcategory spanned by all simplices except the barycentres. 
We can think of $B$ as splitting the edges of $X$ in half without adding a new endpoint.
Note that this does not come from a semi-simplicial set, but it is still a category.
In particular, we can right Kan extend functors along $B \subseteq \int_G \simp{X^{\mathrm{bary}}}$.  
\begin{lemma}
    \label{BBary}
    A resolution $M \in \MV(X^{\mathrm{bary}}, \mathcal{D})$ vanishes on all barycentres if and only if it is right Kan extended from its restriction to $B$.
   In particular, $\mathcal{B}(\mathcal{D})^{0}$ is equivalent to the full subcategory of $\MV_0(X^{\mathrm{bary}}, \mathcal{D})^{\heartsuit}$ of resolutions that are right Kan extended from their restriction to $B$.
\end{lemma}
\begin{proof}
    Let $M \in \MV_0(X^{\mathrm{bary}}, \mathcal{D})^\heartsuit$ be any resolution and consider the natural map
\[\begin{tikzcd}
	M && {\Ran_B^{\int_G \simp{X^{\mathrm{bary}}}} \circ \res{\int_G \simp{X^{\mathrm{bary}}}}{B} M} .
	\arrow[from=1-1, to=1-3]
\end{tikzcd}\]
    Since $B$ is a full subcategory, the map is an equivalence on objects of $B$.
    Furthermore, we can compute the evaluation on the right-hand side on a barycentre $b$ as a limit over the slice $b/B$.
    But this category is always empty, hence the limit is zero.
    This shows that a resolution is precisely right Kan extended from its restriction to $B$ if it vanishes on barycentres.
\end{proof}

From now on we fix an edge $e \in X$ that is subdivided into $e_0$ and $e_1$ in $X^{\mathrm{bary}}$ with $\partial_1e \subset e_0$ and $\partial_0e \subset e_1$. Note that both $e_0$ and $e_1$ have stabilizer $G_e$.
This yields two functors $BG_{e_i} \to B$.
\begin{proposition}
    \label{doubleedgeker}
    The two left Kan extensions along $ BG_{e_i} \to B$ induce an equivalence of categories
\[\begin{tikzcd}
	{\mathcal{D}(e)^{ \heartsuit} \oplus \mathcal{D}(e)^{ \heartsuit}} && {\mathcal{B}(\mathcal{D})^0.}
	\arrow[ from=1-1, to=1-3]
\end{tikzcd}\]
\end{proposition}
\begin{proof}
    We identify $\mathcal{B}(\mathcal{D})^0$ with a full subcategory of $\Fun(B, D(R)^{ \heartsuit})$ as in Lemma \ref{BBary}, namely those functors with right Kan extension in $\MV_0(X^{\mathrm{bary}}, \mathcal{D})$. 
    The first step is to check that $\Lan_{BG_{e_i}}^BV$ is indeed in the kernel of $\alpha$ when extended to $\MV(X^{\mathrm{bary}}, \mathcal{D})^{\heartsuit}$. 
    By post composing $BG_{e_i} \to B$ with $B \subseteq \int \simp{X^{\mathrm{bary}}}$  we obtain two functors $BG_{e_i} \xrightarrow{}B \to \int_G \simp{X^{\mathrm{bary}}}$.
    There is a canonical natural transformation, induced by the universal property of Kan extensions, that extends the identity on $\Lan_{\iota_i}V$
\[\begin{tikzcd}
	{\Lan_{BG_{e_i}}^{\int_G \simp{X^{\mathrm{bary}}}}V} && {\Ran_{B}^{\int_G \simp{X^{\mathrm{bary}}}} \Lan_{BG_{e_i}}^B V.}
	\arrow[from=1-1, to=1-3]
\end{tikzcd}\]
    It is the identity on simplices in $B$, and the right-hand side is zero on the barycentres.
    Therefore, the fibre $\mathrm{fib}$ is concentrated on the barycentres and there equal to $\Lan_{BG_{e_i}}^{\int_G \simp{X^{\mathrm{bary}}}}V$.
    In particular if $V \in \mathcal{D}(e)$, this shows that all terms are in $\MV(X^{\mathrm{bary}}, \mathcal{D})$.
    Using Lemma \ref{Ranres} for left Kan extensions we compute 
    $$(\Lan_{BG_{e_i}}^{\int_G \simp{X^{\mathrm{bary}}}}V)(b) =\bigoplus_{(\partial, g) \in  e_i / BG_b}\ind{(G_{e_i})^{g^{-1}}}{G_b}V^{g^{-1}}  $$
    for any barycentre $b$.
   Recall that $e_i/ BG_b$ is a groupoid with trivial morphism spaces, and since $G$ acts without inversion, there is exactly one object $(\partial , g)$ up to isomorphism and $G_b=(G_{e_i})^{g^{-1}}$. 
   Hence, the right-hand side is equivalent to $V^{g^{-1}}$. 
   Note that this is precisely the value of left Kan extension of $V$ along the functor $BG_{e_b} \to \int_G \simp{X^{\mathrm{bary}}}$ where $e_b$ is the barycentre of the edge $e$, and the map from the fibre corresponds to the canonical map 
    \[\begin{tikzcd}
	{\Lan_{BG_{e_b}}^{\int_G \simp{X^{\mathrm{bary}}}}V} && {\Lan_{BG_{e_i}}^{\int_G \simp{X^{\mathrm{bary}}}}V}
	\arrow[from=1-1, to=1-3]
    \end{tikzcd}\]
    induced from the identity on $BG_{e_b}$.
    In particular, this map becomes an equivalence after applying $\alpha$ and therefore $\Ran_B^{\int_G \simp{X^{\mathrm{bary}}}} \Lan_{BG_{e_i}}^BV$ lies in the kernel of $\alpha$.
  We also have a functor in the other direction by restricting along the inclusions $BG_{e_i} \to B$, and, by the universal property of Kan extensions, this is right adjoint to $\Lan_{BG_{e_i}}^B$.
   Furthermore, the unit of the adjunction is an equivalence since the inclusion is full. 
   It remains to be shown that the counit of the adjunction is also an equivalence.
   Let $M \in \mathcal{B}(\mathcal{D})^0$ and consider the counit
\[\begin{tikzcd}
	{\Lan_{BG_{e_0}}^BM(e_0) \oplus \Lan_{BG_{e_1}}^BM(e_1)} && M.
	\arrow[from=1-1, to=1-3]
\end{tikzcd}\]
    Again, because the inclusions are full, this is an equivalence on $e_0$ and $e_1$. 
    In particular, the cofibre is trivial on edges and in the kernel of $\alpha$. 
    But then Lemma \ref{AssemFormula} shows that it has to be also trivial on vertices and therefore zero.
    Hence, the counit is an equivalence.
\end{proof}

As a consequence of the results we have established so far, we can now prove that $\MV_0(X, \mathcal{D})$ is compactly generated.
\begin{proposition}
    \label{Kercompg}
    The kernel $\MV_0(X, \mathcal{D})$ is compactly generated and hence in $Pr^{L}_{st, \omega}$. 
\end{proposition}
\begin{proof}
   Since the kernel of an exact, colimit preserving functor between stable categories is always stable and cocomplete, the last claim follows immediately once we have shown that $\MV_0(X, \mathcal{D})$ is compactly generated.  
   Therefore, it suffices to find a set of compact objects $\mathcal{G}$ in $\MV_0(X, \mathcal{D})$ that is jointly conservative, that is $M \in \MV_0(X, \mathcal{D})$ is contractible if and only if the mapping spectrum $\map(G, M)$ vanishes for all $G \in \mathcal{G}$. 
   By Lemma \ref{tker} the point-wise $t$-structure on $\MV(X, \mathcal{D})$ restricts to a $t$-structure on $\MV_0(X, \mathcal{D})$.
   In particular, the $t$-structure is $t$-complete and $t$-cocomplete. 
   This means that for any $M \in \MV_0(X, \mathcal{D})$ we have $M = \colim _n \, \tau_{\geq n }M$ and $M = \lim _n \, \tau_{\leq n }M$. 
   This implies that we only have to check that the objects $C$ are jointly conservative on objects in $\MV_0(X, \mathcal{D})^\heartsuit$.
   By Lemma \ref{FiltKer} every resolution $M\in \MV_0(X, \mathcal{D})^\heartsuit$ admits a filtration $F^nM$ with $n \geq 0$ such that the cofibre of $F^nM \to F^{n+1}M$ lies in $\mathcal{C}(\mathcal{D})^0$ for all $n\geq -1$ with $F^{-1}M=0$.
   Then $M\simeq 0$ if and only if $F^{n+1}M / F^nM \simeq 0$ for all $n\geq -1$.
   But by Lemma \ref{compc0} and Proposition \ref{doubleedgeker} we know that $\mathcal{C}(\mathcal{D})^0\simeq \mathcal{D}(e)^\heartsuit \oplus \mathcal{D}(e)^\heartsuit$.
   Therefore, the set of compact objects in $\mathcal{D}(e)^\heartsuit \oplus \mathcal{D}(e)^\heartsuit$ regarded as objects in $\MV_0(X, \mathcal{D})$ will be jointly conservative.
\end{proof}

\begin{corollary}
    \label{D0MV0}
    The composition 
\[\begin{tikzcd}[ampersand replacement=\&]
	{\mathcal{C}(\mathcal{D})^{0,\omega}} \&\& {\MV_0(X)^{\heartsuit, \omega}} \&\& {\MV_0(X)^\omega}
	\arrow[from=1-1, to=1-3]
	\arrow[from=1-3, to=1-5]
\end{tikzcd}\]
    is an equivalence on $K$-theory. 
\end{corollary}
\begin{proof}
    By Lemma \ref{DevReq} the first map satisfies the assumptions for Quillen's D\'evissage theorem and is therefore an equivalence on $K$-theory spectra.  
    The second map becomes an equivalence on $K$-theory by the non-connective the heart \cite[Theorem 1.3]{AGH19} since the $t$-structure on $\MV_0(X)^\omega$ is bounded and the heart is Noetherian. 
\end{proof}
\subsection{A formula for the \texorpdfstring{$K$}{}-theory spectrum}
We can combine the results of the previous sections to obtain the following theorem.
\begin{theorem}
    \label{Kseq}
    Let $G$ be a reductive, algebraic group with a smooth action on a tree $X$ that is transitive, without inversion on edges, and has compact stabilizers.
    Then, for any edge $e \in X$ the following holds. 
    \begin{enumerate}
        \item If $X_0$ has one $G$-orbit, then 
\[\begin{tikzcd}
	{K(G_e)} && {K(G_{\partial_0e})} && {K(G)}
	\arrow["{\ind{}{}}", from=1-3, to=1-5]
	\arrow["{\ind{}{}(\partial_0) - \ind{}{}(\partial_1)^g}", from=1-1, to=1-3]
\end{tikzcd}\]
            with maps induced by induction and $g \in G$ such that $g\partial_1e = \partial_0e$, is a cofibre sequence.
        \item If $X_0$ has two $G$-orbits, then 
\[\begin{tikzcd}
	{K(G_e)} && {K(G_{\partial_0e})} \\
	\\
	{K(G_{\partial_1e})} && {K(G)}
	\arrow["{\ind{}{}}", from=1-3, to=3-3]
	\arrow["{\ind{}{}}", from=1-1, to=1-3]
	\arrow["{\ind{}{}}"', from=1-1, to=3-1]
	\arrow["{\ind{}{}}"', from=3-1, to=3-3]
\end{tikzcd}\]
   with maps induced by induction, is a pushout square.

    \end{enumerate}
    In particular $K(G)$ is connective.
\end{theorem}
\begin{proof}
    By definition of $\MV_0(X)$ \ref{DefMV0}, Proposition \ref{Kercompg} and Proposition \ref{MVBous}, there is an exact sequence of categories 
\[\begin{tikzcd}
	{\MV_0(X)} && {\MV(X)} && {D(G)}
	\arrow[from=1-1, to=1-3]
	\arrow[from=1-3, to=1-5]
\end{tikzcd}\]
    and hence a bifibre sequence of $K$-theory spectra.
    Since the action of $G$ on $X$ is transitive on edges Proposition \ref{KMV} yields  
    $$
    K(\MV(X)^\omega) \simeq \left(\bigoplus_{[v] \in X_0/G} K(G_{v})\right) \bigoplus K(G_e), 
    $$
    and we can always assume that $v \subset  e$.
    Corollary \ref{D0MV0} shows
    $$
    K(\MV_0(X)) \simeq K(G_e) \bigoplus K(G_e)
    $$ 
    such that the maps into $K(\MV(X)^\omega)$ are induced by induction.
    We can rearrange the fibre sequence to a pullback square as shown in the bottom part of the following diagram, where $\nabla$ is the anti-diagonal and the middle arrow is induced by the induction along the face maps $\partial_i$ from $G_e$ to $G_{\partial_ie}$.
\[\begin{tikzcd}
	{K(G_e)} && {K(G_e)} \\
	\\
	{K(G_e) \bigoplus K(G_e)} && \bigoplus_{[v] \in X_0/G}K(G_v) \\
	\\
	{K(G_e)} && {K(G)}
	\arrow["{\partial_0 + \partial_1}", from=3-1, to=3-3]
	\arrow[from=5-1, to=5-3]
	\arrow[from=3-3, to=5-3]
	\arrow["\nabla"{description}, from=3-1, to=5-1]
	\arrow["{=}"{description}, from=1-1, to=1-3]
	\arrow["\Delta"{description}, from=1-1, to=3-1]
	\arrow[ from=1-3, to=3-3]
\end{tikzcd}\]
    Note that if $\partial_0e = g\partial_1e$ for some $g\in G$, then this induces an equivalence $K(G_{\partial_0e}) \simeq K(G_{\partial_1e})$ that we use to ensure that both inductions above have the same target.
    Since it is a pullback square the fibres of the vertical maps are equivalences.
    But the fibre of the left-hand map is $K(G_e)$, where $\Delta$ is the diagonal with a negative sign on one factor, hence also the fibre on the right. 
    Then the sequence on the right-hand side is a fibre sequence, hence a cofibre sequence.
    Unraveling this sequence depending on the number of $G$-orbits of $X_0$ yields the desired formulas.

    Lastly using Lemma \ref{tcomp} applied to $*$ with the trivial $G_x$-action we see that $D(G_e)^\omega$ admits a bounded $t$-structure with Noetherian heart. 
    Then by \cite[Theorem 1.2]{AGH19} the spectrum $K(G_x)$ is connective. 
    Hence, $K(G)$ is the pushout of connective spectra and therefore connective.
\end{proof}

\begin{remark}
    We want to close this section with a remark about groups with reduced Bruhat--Tits building of rank one.
    If we fix a central character, the same method can be adapted to compute the $K$-theory for groups that act without inversion on a tree such that all stabilizers are compact modulo centre and the action is transitive on edges.
    For example, this hypothesis is satisfied if the reduced Bruhat--Tits building is a tree.
    In that case, we have to require that all Mayer--Vietoris resolutions take values in the derived category $D(G_x, \chi)$ with coefficients in a regular with ring $R$ with $\Q \subseteq R$ and fixed central character $\chi$.
    This is again equivalent to modules over a certain regular Hecke-algebra by \cite[Theorem 7.2]{BL22}.
    Both, Kan extension and restriction, preserve this condition, and all arguments can be carried out without modification. 
    Note that the resolutions with a fixed central character do not define an admissible collection because $D(G_x, \chi)$ is not a full subcategory of $D(G_x)$. 
    In particular, we can not simply apply the machinery we have developed for relative resolutions.
\end{remark}

\section{The \texorpdfstring{$K$}{}-theory of principal series Bernstein blocks}
The goal of this section is to describe the $K$-theory of all principal series Bernstein blocks for split reductive groups of rank one in terms of compact open subgroups. 
A crucial ingredient to achieve this is the existence of types based on the work of Roche in \cite{RO98}. 
The section is divided into two subsections. 

In the first subsection, we will recall some results about the existence and construction of types.
Roche also provides a complete description of the set of intertwiners for principal series types which will be a key ingredient in the proofs. 

In the second subsection, we prove the formula for the $K$-theory spectra of Bernstein blocks. 
This will be done by first defining an admissible collection of subcategories associated to a type. 
Then we again construct a Verdier sequence such that we can compute the $K$-theory of two of the three terms. 
The main task to achieve this will be to show that the relative version of $\alpha$ that we define is again a left Bousfield localization. 
This will occupy the majority of the second part. 

\subsection{A recollection of the construction of principal series types}
In this section, we recall the construction of certain types for principal series Bernstein blocks over $\C$ done by Roche in \cite{RO98}.
The results of this section are due to Roche or direct consequences of such. 
Let $\mathbb{G}$ be a connected, split reductive algebraic group over a non-Archimedean local field $F$ with maximal $F$-split torus $\mathbb{T}$. 
Write $T$ for the $F$-rational points of $\mathbb{T}$. This has a unique maximal compact subgroup $\mathbb{T}(\mathcal{O}_F)$ which we denote by $T^0$.  
Depending on the group we have to make some assumptions on the residual characteristic of the field $F$.
Let $\Phi$ be the root system associated to $(G,T)$ and let $p$ be the residual characteristic of $F$. 
If $\Phi$ is irreducible we require that $p$ satisfies the following requirements depending on the type of $\Phi$ \cite{RO98}[before Theorem 4.15].
\begin{itemize}
    \item If $\Phi=A_n$ then $p>n+1$,
    \item if $\Phi=B_n, C_n ,D_n$ then $p \neq 2$, 
    \item if $\Phi=F_4$ then $p \neq 2,3$,
    \item if $\Phi=G_2, E_6$ then $p \neq 2,3,5$,  
    \item if $\Phi = E_7, E_8$ then $p \neq 2,3,5,7$.
\end{itemize}
Otherwise, if $\Phi$ is not irreducible we need to exclude all primes corresponding to the irreducible factors.
These restrictions are not optimal, in particular, if $G$ is $\mathrm{GL}_n(F)$ or $\mathrm{SL}_n(F)$ we do not need any additional assumptions on the residual characteristic \cite{RO98}[remark 4.14].
Lastly, we also fix a positive set of roots $\Phi^+$.

\begin{remark}
    \label{splitrank1gps}
    We will mostly be interested in root systems associated to reductive groups of rank one.
    In that case $X^*(T)=\Hom(T, \mathbb{G}_m)$ has dimension one and there is either no root or only one pair of roots $\pm \alpha$.
    A choice of positive roots then corresponds to choosing either $\alpha$ or $-\alpha$.
    The classification of connected, split reductive groups tells us that there are precisely three groups of rank one. These are $\mathbb{G}_m$, $\mathrm{SL_2}$ and $\mathrm{PGL}_2$.
\end{remark}
The choice of a positive set of roots also provides us with a standard apartment together with a distinguished simplex of maximal dimension that we will need in the next subsection. 
Furthermore, this specifies an Iwahori subgroup $I$ as a finite index subgroup in the stabilizer of this simplex and all types will be constructed as characters on open subgroups of $I$.
For the construction of principal series types, we start with a smooth character $\chi: T^0 \to \C^\times$. 

\begin{theorem}[Roche]
    \label{TypeRoche}
    Let $\chi: T^0 \to \C^\times$ be a smooth character. There is a compact open subgroup $J_\chi \subseteq I$ containing $T^0$ and an extension of $\chi$ to a character $\rho_\chi: J_\chi \to \C^\times$ which is a type.
\end{theorem}
\begin{remark}
    \label{RemLeviType}
    Let $\widetilde{\chi}: T \to \C^\times$ be any extension of $\chi$.
    This determines a Bernstein block of $G$. Any other extension of $\chi$ only differs by twisting with an unramified character, therefore all extensions determine the same Bernstein block and $\rho_\chi$ is a type for this block by \cite[Theorem $7.0$]{RO98}.
\end{remark}
\begin{lemma}[Roche]
    \label{IwahoridecompType}
    The group $J_\chi$ has a decomposition $J_\chi=U'T^0\overline{U'}$ and $\rho_\chi$ is trivial on $U'$ and $\overline{U'}$.
\end{lemma}
\begin{example}
    \label{TypesSL}
    Let us explicitly apply the above construction to obtain types for the principal series blocks of $\mathrm{SL}_2(\mathbb{Q}_p)$. 
    A maximal torus is given by $\lbrace \mathrm{diag}(z, z^{-1}) \mid z \in \mathbb{Q}_p^\times \rbrace $ and $T^0$ are all elements in $T$ with $z \in \mathbb{Z}_p^\times$.
    Now let $\chi: \Z_p^\times \to \mathbb{C}^\times$ be a smooth character and let $n$ be minimal such that $1 + n\mathbb{Z}^p$ is contained in the kernel of $\chi$. 
    Then $\chi$ extends to a character $\rho_\chi$ on 
    \begin{equation*}
       J_\chi = 
       \begin{pmatrix}
        \Z^\times_p &  p^{\lfloor n/2 \rfloor }\Z_p \\
        p^{\lfloor (n+1)/2 \rfloor} \Z_p & \Z^\times_p
       \end{pmatrix} \cap \mathrm{SL}_2(\mathbb{Q}_p)
    \end{equation*} 
    that is trivial away from the diagonal (here $\lfloor x \rfloor$ denotes the largest integer smaller than $x$).
    If we start with any smooth character $\chi$ of $T= \Q^\times_p$ we can restrict it to $T^0$ and apply the above construction. This way we obtain a type for every principal series Bernstein block of $\mathrm{SL}_2(\Q_p)$ (see remark \ref{RemLeviType} and \cite[example $3.5$]{RO98}).
\end{example}
Recall that the extended Weyl group of $G$ is $W = N(T)/T^0$. 
In particular, we can take an element $\omega \in W$ and conjugate a character $\chi: T^0 \to \C^\times$. 
\begin{definition}
    For a character $\chi : T^0 \to \C^\times$ we define 
    $ W_\chi \coloneqq \lbrace \omega \in W \mid \chi^\omega = \chi \rbrace $ and $N(T)_\chi \coloneqq \lbrace \omega \in N(T) \mid \chi ^\omega = \chi \rbrace$.
\end{definition}
\begin{lemma}[Roche]
    \label{RocheIntertwiner}
    Let $\chi: T^0 \to \C^\times$ be a character with associated type $\rho_\chi$. Then $\mathcal{I}(\rho_\chi) \coloneq \lbrace g \in  G \, \vert \, \Hom_{J_{\chi}^g \cap J_\chi}(\rho^g_\chi , \rho_\chi)\neq 0 \rbrace =  J_\chi W_\chi J_\chi$ are the intertwiners of $\rho_\chi$ with itself. 
 \end{lemma}

We can also apply the above theorem to the conjugated character $\chi^\omega: T^0 \to \C^\times$ to obtain another type $\rho_{\chi^\omega}: J_{\chi^\omega} \to \C^\times$. 
\begin{remark}
    Note that the conjugated character $\chi^\omega: T^0 \to \C^\times$ only depends on the image of $\omega$ in the finite Weyl group $N(T)/T$ since conjugation with elements in $T$ acts trivially.
\end{remark}
\begin{lemma}
    Let $\chi : T^0 \to \C^\times $ be a smooth character and $\omega \in W$. If $T$ has rank $1$ then $J_\chi = J_{\chi^\omega}$.
\end{lemma}
\begin{remark}
    The proof of this lemma requires some details about the construction of $J_\chi$ that we did not recall. For a complete treatment see \cite[section $3$]{RO98}. 
\end{remark}
\begin{proof}
    For any root $\alpha \in \Phi$ we denote by $c_\alpha$ the conductor of the composition $\chi \circ \alpha^\vee$ where $\alpha^\vee$ is the dual root to $\alpha$. 
    The construction of $J_\chi$ depends only on the function $f_\chi: \Phi \to \Z$ that is defined as 
    
    \begin{equation*}
        f_\chi(\alpha) =
        \begin{cases}
        \lfloor {c_\alpha/2}\rfloor  &\text{ if }\alpha \in \Phi^+ \\
        \lfloor {(c_\alpha +1)/2} \rfloor & \text{ if }\alpha \in \Phi^-.
        \end{cases}
    \end{equation*} 
   If $T$ has rank $1$, then $\mathrm{Aut}(T)$ has $2$ elements ($t \mapsto t$ and $t \mapsto t^{-1}$).
   But $N(T)$ acts by conjugation on $T$ with kernel $C(T)$, hence $W$ maps injectively to $\mathrm{Aut}(T)$. Therefore, $W$ has at most $2$ elements, and conjugating with either element does not change the kernel. 
   This implies that $\chi$ and $\chi^\omega$ have the same kernel and hence also the pre-composition with any dual root $\alpha^\vee$.
   In other words $f_\chi = f_{\chi^\omega}$ and therefore $J_\chi = J_{\chi^\omega}$. 
\end{proof}

\begin{lemma}
    \label{chiinter}
    Let $\omega \in W$.
    Then any element of $ \omega W_\chi $ intertwines $\rho_\chi$ and $\rho_{\chi^\omega}$. 
\end{lemma}
\begin{proof}
    By \ref{IwahoridecompType} the group $J_\chi$ admits an Iwahori decomposition $J_\chi=U'_\chi T\overline{U'}_\chi$ such that $\rho_\chi$ and is trivial on $U'_\chi$ and $\overline{U'}_\chi$ and equal to $\chi$ on $T^0$.
    Similarly, for $\rho_{\chi^\omega}$. 
    Then for any $\omega^{-1}g \in \omega^{-1}W_\chi$ the character $(\rho_{\chi^\omega})^{\omega^{-1}g}$ is defined on $J_{\chi^\omega}^{\omega^{-1}g} = (U'_{\chi^\omega})^{\omega^{-1}g}(T^0)^{\omega^{-1}g}(\overline{U'}_{\chi^\omega})^{\omega^{-1}g}$ where it is trivial outside $(T^0)^{\omega^{-1}g}=T^0$ and equal to $(\chi^\omega)^{\omega^{-1}g}=\chi$ on $T^0$. 
    In particular $\rho_\chi$ and $(\rho_{\chi^\omega})^{\omega^{-1}g}$ agree on the intersection $J_\chi \cap J_{\chi^\omega}^{\omega^{-1}g}$, hence $\omega^{-1}g$ in an intertwiner. 
\end{proof}

\begin{lemma}
    Let $\chi: T^0  \to \C^\times$ be a smooth character and $\omega \in W$. Then $\rho_\chi$ and $\rho_{\chi^\omega}$ are types for the same Bernstein block.
\end{lemma}
\begin{proof}
Let $\widetilde{\chi}: T \to \C^\times$ be any extension of $\chi$.
Then $(\widetilde{\chi})^\omega$ is an extension of $\chi^\omega$.
Since $\widetilde{\chi}$ and $(\widetilde{\chi})^\omega$ only differ by conjugation they are both in the same inertia class and hence determine the same Bernstein block.
Since the types associated types $\rho_\chi$ resp.\ $\rho_{\chi^\omega}$ correspond to the Bernstein blocks of $\widetilde{\chi}$ resp.\ $(\widetilde{\chi})^\omega$ they determine the same block.
\end{proof}

\subsection{The colimit associated to a principal series Bernstein block}
In this section, $X$ will be the Bruhat--Tits tree of a connected split reductive algebraic group $G$ of rank one.
By the classification of connected split reductive groups, these are (the $F$-points of) either $\mathbb{G}_m$, $\mathrm{PGL}_2$ or $\mathrm{SL}_2$.
Furthermore, we will continue to use the notation from the previous section. 
In particular, we still fix a root system $\Phi$, a set of positive roots $\Phi^+$ and restrict the residual characteristic accordingly.
From now on we will only work with complex representations, that is, $R= \C$.
By \ref{BruhatTitsssSet} we can again assume that $G$ acts without inversion and compact stabilizers on $X$. 
Furthermore, let $I$ be the standard Iwahori subgroup associated to the set of positive roots. Then $I$ is canonically contained in the stabilizer of an edge $e$ (before subdivision) as a finite index open subgroup, and we fix this edge for the rest of this section. 
If we have to subdivide $X$ to ensure that the action does not invert edges then $e$ is split into $e_0$ and $e_1$ and $G_e=G_{e_0}=G_{e_1}$. We fix $e_0$ and write again $e$ instead of $e_0$. 

Let $(J, \rho_\chi)$ be a type associated to a smooth character $\chi: T^0 \to \C^\times$ as before. 
This gives us a functor $\rho_\chi: BJ \to D(\C)$.
Because $J \subseteq I\subseteq G_e$ we obtain a functor
\[\begin{tikzcd}
	BJ && BG_e && {\int_G \simp{X}.}
	\arrow[from=1-1, to=1-3]
	\arrow[from=1-3, to=1-5]
\end{tikzcd}\]
By left Kan extending this induces a smooth functor 
$$
\mathrm{L}\rho_\chi^e: \int_G \simp{X}\op \to D(\C).
$$ 
With this notation, we define a collection of subcategories.
\begin{definition}
    \label{defrel}
    Let $\mathcal{D}_n(\chi)$ be the smallest full stable subcategory of $\MV(X_n)$  that is closed under colimits and contains $(\mathrm{L}\rho_{\chi^\omega}^e)_{X_n}$ for all $\omega \in W$. 
    We define $\mathcal{D}(\chi):= \lbrace \mathcal{D}_n (\chi) \rbrace_{n \in \N}$.
\end{definition}

\begin{notation}
    \label{NotationRelCat}
    From now on, we will use the notation $D(G_x, \rho_\chi):= \mathcal{D}(\chi)(x)$ with associated projection $\pr{\chi} := \pr{\mathcal{D}(\chi)}$.
    Furthermore, let $D(G, \rho_\chi)$ be the Bernstein block corresponding to the type $\rho_\chi$.
\end{notation}
For each $x\in X$, we can give an explicit description for the category $D(G_x, \rho_\chi) \subseteq D(G_x)$.

\begin{lemma}
    \label{Dchi}
    For every $x\in X$, the category $D(G_x, \rho_\chi)$  is the smallest full stable subcategory of $D(G_x)$ that contains $\ind{(J_{\chi^\omega})^g}{G_x}(\rho_{\chi^\omega})^g$ for all $\omega \in W$ and $g \in G$ such that $x \subseteq ge$ and is closed under colimits.
\end{lemma}
\begin{remark}
    \label{numberinds}
    The number of distinct $\ind{(J_{\chi^\omega})^g}{G_x}( \rho_{\chi^\omega})^g$ in the lemma is finite. 
    The proof will show that for a fixed $\omega$ there is, up to isomorphism, exactly one if $x$ is an edge.
    If $x$ is a vertex there are two precisely if both vertices of $e$ lie in the same $G$-orbit and one otherwise.   
\end{remark}
\begin{proof}
    Since $(X_n)_{hG}$ decomposes into a disjoint union of $BG_x$'s, the category $\MV(X_n)$ is equivalent to a product of $D(G_x)$'s by choosing a representative $x$ for every $G$-orbit and restrict along the inclusions $BG_x \to (X_n)_{hG}$.
    Hence, $D(G_x, \chi)$ is the smallest subcategory of $D(G_x)$ that is closed under colimits and contains $(\Lan \rho_{\chi^\omega}^e)_{X_n}(x)$ for all $\omega \in W$, and we only need to compute the evaluation of the left Kan extension to prove the claim. 
    This can be done by use of Corollary \ref{Ranres} which yields 
    $$  
L\rho^e_{\chi^\omega}(x) \simeq \bigoplus _{(\partial, g) \in e/BG_x} \ind{(J_{\chi^\omega})^g}{G_x}(\rho_{\chi^\omega})^g.
    $$ 
    But when $X$ is the Bruhat--Tits tree, we can say more about $e/BG_x$.
    Recall that $e/BG_x$ is a groupoid with trivial automorphism spaces, therefore it is a set. 
    Note that $e/BG_x$ is empty if $\dim(x)>\dim(e)$, by definition of $\int_G \simp{X}$. If $\dim(e)= \dim(x)$, it has up to isomorphism exactly one object given by $(\id, g)$ with $g \in G$ such that $gx=y$, because $G$ acts transitively on edges.
    Hence, the only non-trivial case is $\dim(x)<\dim(e)$. 
    If $(\partial,g), (\partial',g') \in e/BG_x$ then $ g \partial(e) =x= g'\partial'(e)$ or equivalently $\partial(e) = \partial '(g^{-1}g'e)$.
    Note that there can only be an isomorphism between $(\partial , g)$ and $(\partial ' , g')$ if $\partial = \partial '$.
    If the latter is the case, then $g^{-1}g' \in G_x$ and $g^{-1}g'$ induces an isomorphism.
    Therefore, $BG_x/ e$ has two elements precisely if $\partial_0(e)$ and $\partial_1(e)$ are in the same $G$-orbit and one element otherwise.
    \end{proof}
Using that all $G_x$ are compact, we also obtain the following.
\begin{lemma}
    \label{relpwcolim}
    For every $n$, the right adjoint to the inclusion $\mathcal{D}(\chi)_n \to \MV(X_n)$ commutes with all colimits and preserves compact objects. 
\end{lemma}
\begin{proof}
   Since $\int_G \simp{X_n} = (X_n)_{hG}$ decomposes into a disjoint union of $BG_x$ it suffices to check the claim for all inclusions $D(G_x, \rho_\chi) \to D(G_x)$. 
    For this, first note that, since $G_x$ is compact, every representation decomposes into irreducibles and every irreducible has endomorphism ring equal to $\C$ by Schur's lemma. 
    Hence, on the level of categories, this implies
    $$
    D(G_x) \simeq \prod_{\mathrm{Irr}(G_x)}D(\C).
    $$
    Then $D(G_x, \rho_\chi)$ corresponds precisely to the product indexed over the irreducibles appearing in $D(G_x, \rho_\chi)$, and the right adjoint to the inclusion is precisely the projection 
\[\begin{tikzcd}
	{\prod_{\mathrm{Irr}(G_x)}D(\C)} && {\prod_{\mathrm{Irr}(D(G_x, \chi))}D(\C).}
	\arrow[from=1-1, to=1-3]
\end{tikzcd}\]
    In particular, it commutes with all colimits and preserves compact objects.
\end{proof}
This immediately yields the following corollary.
\begin{corollary}
    \label{Dchicomp}
        The category of compact objects  $D(G_x, \rho_\chi)^\omega$ is the smallest subcategory of $D(G_x)$ that contains 
        $\ind{(J_{\chi^\omega})^g}{G_x}(\rho_{\chi^\omega})^g$ for all $\omega \in W$ and $g \in G$ such that $x \subseteq ge$
         and is closed under finite colimits and retracts.
\end{corollary}
\begin{proof}
    By \ref{relpwcolim}, the right adjoint to inclusion $D(G_x, \rho_\chi) \to D(G_x)$ commutes with all colimits, in particular an object in $D(G_x, \rho_\chi)$ is compact if and only if it is compact in $D(G_x)$. 
    But the latter are precisely the finite colimits of irreducibles of $G_x$. 
\end{proof}
Using the above description we can prove that $\mathcal{D}(\chi)$, with the restricted $t$-structures from $\MV(X_n)$, is admissible.
\begin{lemma}
    The collection $\mathcal{D}(\chi)$ is admissible.
\end{lemma}
\begin{proof}
   By construction and Lemma \ref{Dchi} and Lemma \ref{relpwcolim} all $\mathcal{D}_n(\chi)$ are full cocomplete, stable, $t$-exact subcategories (with respect to the restricted $t$-structure of $MV(X_n)$) and the inclusion commutes with colimits and preserves compact objects. 
   In particular, it admits a right adjoint.
   Lastly Lemma \ref{Dchicomp} shows that $D(G_x,\rho_ \chi)$ is compactly generated, by Lemma \ref{relpwcolim} that the inclusion $D(G_x,\rho_ \chi)\to D(G_x)$ preserves compact objects and and by Lemma \ref{Dchi} induction maps $\mathcal{D}_{n+1}(\chi)$ into $\mathcal{D}_n(\chi)$.
\end{proof}
The admissible collection $\mathcal{D}(\chi)$ gives rise to a full subcategory of $\MV(X)$.
\begin{definition}
    We define $\MV(X, \rho_\chi)$ to be the full subcategory $\MV(X, \mathcal{D}(\chi))$ of $\MV(X)$ that is associated to the admissible collection $\mathcal{D}(\chi)$.
\end{definition}

\begin{corollary}
    \label{MVchicomp}
    The category $\MV(X, \rho_\chi)$ is compactly generated and the full inclusion $\MV(X, \rho_\chi) \to \MV(X)$ has a colimit preserving right adjoint. In particular $\MV(X, \rho_\chi)$ is both complete and cocomplete 
    Furthermore, the left Kan extensions of $\ind{J_{\chi^\omega}}{G_x} \rho_{\chi^\omega}$ to $\int_G \simp{X}$ for $\omega \in W$ form a set of compact generators.
    The compact objects in $\MV(X, \rho_\chi)$ are precisely the compact objects in $\MV(X)$ that lie in $\MV(X, \rho_\chi)$.
\end{corollary}
\begin{proof}
    This follows directly from Lemma \ref{RelLan} and \ref{MVLoc} together with Lemma \ref{cocompB}. 
    The claim that the right adjoint commutes with colimits follows from Lemma \ref{relcolimit} and Lemma \ref{relpwcolim}.
    Lastly, since the right adjoint commutes with colimits, the inclusion preserves compact objects, and any object in $\MV(X, \rho_\chi)$ that is compact in $\MV(X)$ was already compact in $\MV(X, \rho_\chi)$.
\end{proof}

 With these definitions in place, we obtain for every principal series block a relative version of the functor $\alpha: \MV(X) \to D(G)$.

\begin{lemma}
    The functor $\alpha : \MV(X) \to D(G)$ restricts to a functor 
\[\begin{tikzcd}
	{\MV(X, \rho_\chi)} && {D(G,\rho_\chi)}
	\arrow[from=1-1, to=1-3]
\end{tikzcd}\]
\end{lemma}
\begin{proof}
    By Lemma \ref{Dchi} we see that $\mathcal{D}(\chi)$ is mapped to $D(G, \rho_\chi)$.
    Therefore, the claim follows from Lemma \ref{relmaps}.
\end{proof}

\begin{lemma}
    The functor $\MV(X, \chi) \to D(G, \chi)$ has a right adjoint $\pr{\chi} \circ \res{}{}$.
\[\begin{tikzcd}[ampersand replacement=\&]
	{\MV(X,\rho_\chi)} \&\& {D(G,\rho_\chi)}
	\arrow[""{name=0, anchor=center, inner sep=0}, "\alpha", shift left=2, from=1-1, to=1-3, curve={height=-6pt}]
	\arrow[""{name=1, anchor=center, inner sep=0}, "{\pr{\chi} \circ \res{}{}}", shift left=2, from=1-3, to=1-1, curve={height=-6pt}]
	\arrow["\dashv"{anchor=center, rotate=-90}, draw=none, from=0, to=1]
\end{tikzcd}\]
\end{lemma}
\begin{proof}
    This follows completely formal since compositions of left adjoints are again left adjoint with right adjoint functor given by the composition of the right adjoints.
\end{proof}
Again we want to show that this adjunction is a left Bousfield localization. 
This will be the main task for most of this section.
The claim is equivalent to the counit of the adjunction being an equivalence.
To show this, first note that $\ind{J}{G}\rho_\chi$ is a compact generator of $D(G, \rho_\chi)$ and the counit  
\[\begin{tikzcd}
    \label{relcounit}
	{\alpha \circ \pr{\chi} \circ \, \res{}{}} && \id
	\arrow[from=1-1, to=1-3]
\end{tikzcd}\]
is compatible with colimits.
Hence, we only need to check that the counit is an equivalence for $\ind{J}{G}\rho_\chi$.
By remark \ref{splitrank1gps} there are essentially two cases for our proposes, namely $\mathbb{G}_m$ where the ordinary Weyl group is trivial, and $\mathrm{PGL}_2$ and $\mathrm{SL}_2$ where the Weyl group is cyclic of order two.
In principle, we could apply the same strategy for both cases.
However, the case of $\mathbb{G}_m$ also admits a straightforward proof that omits most technicalities, and we will therefore include it here.
\begin{figure}
    \centering
\begin{tikzpicture}[
  grow cyclic,
  level distance=1.3cm,
  level/.style={
    level distance/.expanded=\ifnum#1>1 \tikzleveldistance/1\else\tikzleveldistance\fi,
    nodes/.expanded={\ifodd#1 fill=none\else fill=none\fi}
  },
  level 1/.style={sibling angle=180},
  level 2/.style={sibling angle=0},
  level 3/.style={sibling angle=0},
  level 4/.style={sibling angle=0},
  nodes={circle,draw,inner sep=+0pt, minimum size=5pt},
  ]
\path[rotate=90]
  node {}
  child foreach \cntI in {1,...,2} {
    node {}
    child foreach \cntII in {1,...,1} { 
      node {}
      child foreach \cntIII in {1,...,1} {
        node {}
        child foreach \cntIV in {1,...,1} {
          node {}
          child foreach \cntV in {1,...,1} {}
        }
      }
    }
  };
\end{tikzpicture}
\caption{Bruhat--Tits building of $T= \mathbb{G}_m(F)$.}
\end{figure}
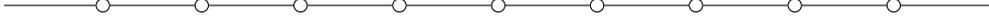
\begin{lemma}
    For $T = \mathbb{G}_m(F)$, the counit \ref{relcounit} is an equivalence and $\MV(X, \rho_\chi) \to D(T, \rho_\chi)$ is a left Bousfield localization. 
\end{lemma}
\begin{proof}
    Recall that the Bruhat--Tits building of $T$ is constructed as $X_*(T) \otimes_\Z \R$, where $X_*(T)$ is a free Abelian group of rank one with $T$-action.
    In particular, this induces a simplicial structure with vertices $X_*(T)$ that is compatible with the $T$-action.
    More precisely, any $t \in T$ acts by translating by the negative of its valuation $-\mathrm{val}(t)$.
    In particular, all simplices have the same stabilizer $\mathcal{O}_F^\times = T^0$.
    Furthermore, for any type $\rho_\chi: J_\chi  \to \C^\times$, the group $J_\chi$ has to be equal to $T^0$ since it contains $T^0$ by Lemma \ref{IwahoridecompType} and $T^0$ is maximal among compact open subgroups of $T$.
    Lemma \ref{Dchi} shows that the associated admissible collection is constant and all simplices have the same projection $\pr{\chi}: D(G_x)  \to D(G_x, \rho_\chi)$.
    Applying this to $\ind{T^0}{T} \rho_\chi $ using Macky's formula, we compute 
   \begin{equation*}
    \res{T}{T^0}(\ind{T^0}{T} \rho_\chi) = \bigoplus _{t \in T^0 \backslash T / T^0} \ind{T^0 \cap (T^0)^t}{T^0} (\rho_\chi)^t = \bigoplus _{T^0 \backslash T / T^0} \rho_\chi.
   \end{equation*} 
   This shows that $\res{}{}(\ind{T^0}{T}\rho_\chi)$ is already in $\MV(X, \rho_\chi)$ and therefore $\pr{\chi} \circ \res{}{}(\ind{T^0}{T}\rho_\chi) = \res{}{}(\ind{T^0}{T} \rho_\chi)$. 
   Then $\alpha : \MV(X, \rho_\chi) \to D(T, \rho_\chi)$ is a Bousfield localization by Proposition \ref{MVBous}.
\end{proof}
For simplicity, we will from now on assume that $G$ has non-trivial (finite) Weyl group, as this covers both remaining cases.
In particular, this implies that the Bruhat--Tits tree has two orbits of vertices after possibly subdividing.
\begin{remark}
    Since we do know all connected split reductive groups of rank one, we can explicitly say in which cases we have to subdivide. 
    For $\mathrm{SL}_2(F)$, the action does not invert edges, hence we do not have to subdivide.
    However, the action of $\mathrm{PGL}_2(F)$ on its Bruhat--Tits tree does indeed invert edges, therefore we will use the subdivided tree in this case.
\end{remark}
In these cases, we can also give a geometric description of the extended Weyl group.
\begin{lemma}
    \label{Weyl}
    The extended Weyl group $W$ is freely generated by the reflections along the vertices of any edge in the (subdivided) standard apartment associated to the maximal torus.
    In particular, it acts free and transitive on the edges of the (subdivided) standard apartment.
    Furthermore, both generators are mapped to non-trivial elements in the finite Weyl group $N(T)/T$. 
\end{lemma}
\begin{proof}
    Let $e$ be any edge of the standard apartment associated to $T$.
    By \cite[Proposition 6.6.3]{KP23}, applied to the case where $G$ is split, there is a decomposition
\begin{equation*}
    W = W^{\mathrm{aff}} \rtimes \Omega 
\end{equation*}
    where $W^{\mathrm{aff}}$ is the affine Weyl group, which is generated by the reflections at the vertices of $e$, and $\Omega = W \cap \hat{G_e}$, where $\hat{G_e}$ are those elements of $G$ that leave $e$ invariant as a set, i.e. fix $e$ or interchange its vertices. 
    In particular, $W^{\mathrm{aff}}$ always acts free and transitive on the set of edges of $\mathcal{A}$, and $\Omega$ contains at most two elements.
    If $G$ acts without inversion on $X$, then $W$ acts without inversion on $\mathcal{A}$ and $\Omega$ is trivial.
    On the other hand, if $G$ inverts any edge, then, by transitivity, there exists $g \in G$ that inverts $e$.
    Furthermore, $G$ acts transitively on the set of apartments, and we can map $g \mathcal{A}$ back to $\mathcal{A}$ while fixing the intersection $g\mathcal{A} \cap \mathcal{A}$. 
    Therefore, we can assume that $g$ fixes $\mathcal{A}$, hence $g \in N(T)$ by \cite[Corollary 7.4.9]{BT72} and $g$ represents an element in the Weyl group.
    This shows that if $G$ inverts an edge, the Weyl group also inverts edges in $\mathcal{A}$. 
    Hence, $\Omega$ is non-trivial and generated by the element that interchanges the vertices of $e$, or equivalently the reflection at the barycentre of $e$.
    But then, $W$ acts free transitive after barycentric subdivision on $\mathcal{A}$.
    Now, let $\omega_i$ be the reflections at the vertices of $e$ and let $b$ be the reflection at the barycentre of $e$. 
    These are a set of generators for $W$, but $\omega_1 = b \omega_0 b$. 
    Therefore, $\omega_0$ and $b$ already generate $W$.
    Lastly, note that the claim that both generators map to non-trivial elements in $N(T)/T$ is equivalent to the assertion that they are not contained in $T/T^0$. 
    But $T/T^0 \simeq \Z$, therefore it has no elements of order two and can not contain the reflections.
\end{proof}
Let $M:=L\rho^e_\chi$ be the left Kan extension 
of $\rho_\chi: BJ \to D(\C)$ along the inclusion $BJ \to \int_G\simp{X}$. Then, by transitivity of left Kan extensions, $\alpha (M) \simeq \ind{J}{G}\rho_\chi$.   
Note that since the compact generator is in the image of $\alpha$, the counit evaluated at $\alpha (M)$ is part of the diagram 
\[\begin{tikzcd}
	{\alpha (M)} && {\alpha \circ \pr{\chi} \circ \res{}{} \circ  \alpha (M)} && {\alpha (M),}
	\arrow[from=1-1, to=1-3]
	\arrow[from=1-3, to=1-5]
\end{tikzcd}\]
where the composite is the identity by the triangle identities of the adjunction.
But then, showing that the right-hand map is an equivalence is equivalent to showing that the left-hand map is an equivalence, and the latter will be the goal for the remaining part of the section.
The strategy to do this will be to interpolate between $M$ and $\pr{\chi} \circ \res{}{} \circ \alpha (M)$ with a filtration $F^\chi_nM$ such that $F^\chi_0=M$ and $\colim_nF^\chi_nM = \pr{\chi} \circ \res{}{} \circ \alpha (M)$.
We then show that all filtration quotients become contractible after applying $\alpha$, and hence the filtration becomes constant after $\alpha$. In particular, $\alpha(M) = \alpha \circ \pr{\chi} \circ \res{}{} \circ \alpha (M)$.
We will first start by constructing a filtration without taking the character $\chi$ into account. 
For ease of notation we define $L:=G_e$ and $K_i := G_{\partial_i e}$.
\begin{construction}
    Let $V \in D(L)^\heartsuit$ be any representation and for ease of notation we will write $M:= \Lan_{BL}^{\int_G \simp{X}}V$.
    The goal is to construct a filtration of $\res{}{} \circ \alpha(M) = \res{}{}( \ind{L}{G}V$). 
    For any simplex $x \in X$ (of dimension zero or one) we define $E^n_x$ to be the set of all edges with distance less or equal $n$ to $x$. 
    Then $E^{n}_x \subset E^{n+1}_x$, and if $e$ is an edge with vertex $v$ we have $E^n_e \subset E^n_v$.  
    Let $M_{E^n_x}\in \MV_{G_x}(X)$ be the resolution that is equal to $M$ on $E^n_x$ and zero otherwise. Applying $\alpha$ yields a sequence 

\[\begin{tikzcd}[ampersand replacement=\&]
	\alpha \left( M_{E^0_x}\right) \& \alpha \left( M_{E^1_x} \right) \&  \alpha \left( M_{E^2_x}\right) \& \alpha \left( M_{E^3_x} \right) \& \ldots
	\arrow[from=1-1, to=1-2]
	\arrow[from=1-2, to=1-3]
	\arrow[from=1-3, to=1-4]
	\arrow[from=1-4, to=1-5]
\end{tikzcd}\]
    in $D(G_x)$, and every term maps canonically to $\alpha \left( M_{E^\infty_x}\right) \simeq \res{G}{G_x} \ind{J}{G}\rho_\chi[1]$, where the equivalence holds by Lemma \ref{AssemFormula}.  
    Note that $\alpha \left(M_{E^n_x}\right)[-1]$ is again in the heart by Lemma \ref{AssemFormula}. 
    Define $F_nV(x):= \alpha \left(M_{E^n_x}\right)[-1]$. 
    The inclusion $E^n_e \subseteq E^n_v$ for $v \subset e$ induces a map 
\[\begin{tikzcd}[ampersand replacement=\&]
	{F^nV(e)} \&\& {F^nV(v).}
	\arrow[from=1-1, to=1-3]
\end{tikzcd}\]
    Together, this defines a functor $F_nV : \int_G \simp{X} \to  D(\C)^\heartsuit$. Furthermore, the inclusions $E^n_x \to E^{n+1}_x$ induce natural transformations 
    $F_{n}V \to F_{n+1}V$. 

\end{construction}
The following properties of the filtration are a direct consequence of the construction.
\begin{lemma}
    \label{FiltProp}
    Let $V \in D(L)^\heartsuit$. Then $F_0V = \Lan_{BL}^{\int_G \simp{X}}V$ and $\colim_n F_nV = \res{}{} \circ \alpha  (F_0V)$.
    Furthermore, the canonical map $F_0V \to \colim_nF_nV$ is the unit of the adjunction $ \alpha \dashv \res{}{}$.
\end{lemma}
\begin{proof}
    We first show that $F_0V = \Lan_{BL}^{\int_G \simp{X}}V$. 
    By construction, both sides agree on edges and the universal property of the left Kan extension yields a map $\Lan_{BL}^{\int_G \simp{X}}V \to F_0V$. 
    By Lemma \ref{Ranres} we can explicitly calculate the right-hand side. 
    This shows that on $\partial_ie$ the left Kan extension is $(\Lan_{BL}^{\int_G \simp{X}}V)(\partial_ie) = \ind{L}{K_i}V$ and the map induced from $e \to \partial_ie$ is the canonical one.
    But on the other hand we have $F_0V(e)=\alpha(E_e^0)[-1]=V$ and $F_0V(\partial_ie)=\alpha (E_{\partial_ie}^0)[-1] = \ind{L}{K_i}V$.
    By \ref{AssemFormula} the inclusion $E_e^0 \to E_{\partial_ie}^0$ corresponds to the map 
      \[\begin{tikzcd}
	V && {\bigoplus_{g \in L \backslash K / L}\ind{L \cap L^g}{L}V^g}
	\arrow[from=1-1, to=1-3]
\end{tikzcd}\] 
    that includes $V$ as the summand associated to $1 \in L \backslash K /L$ where we identify $K/L$ with the $K$-orbit of $e$.
    Hence, $F_0V = \Lan_{BL}^{\int_G \simp{X}}V$.

    To show that $\colim_n F_nV  \simeq \res{}{} \circ \alpha (F_0V)$, note that 
    colimits in $\MV(X)$ are computed pointwise. For every simplex $x$ the set $E^n_x$ converges to the set of all edges for $n \to \infty$ and with that $\colim F^nV(x) = \res{G}{G_x}\alpha_G(M_{E^\infty_x})[-1]= \res{G}{G_x} \ind{L}{G}\rho_\chi$ and all maps are identities.

    For the last claim we use the adjunction $\Lan_{BL}^{\int_G \simp{X}} \dashv \res{\int_G \simp{X}}{BL}$.  
    Since $F_0 V = \Lan_{BL}^{\int_G \simp{X}}V$, this implies that we only need to check that both maps agree after evaluation on $e$.
    The map induced from the filtration corresponds to the map $V =\alpha (E_e^0)[-1] \to \alpha (E_e^\infty)[-1]= \ind{L}{G}V$ induced by the inclusion of $L$-sets $E_e^0 \to E_e^\infty$.
    But by \ref{AssemFormula} this is precisely the map 

\[\begin{tikzcd}
	V && {\bigoplus_{g \in L \backslash G / L}\ind{L \cap L^g}{L}V^g}
	\arrow[from=1-1, to=1-3]
\end{tikzcd}\] 
    that maps $V$ onto the summand corresponding to $1 \in G$ when identifying $G/L$ with the set of edges.
    For the second map, it follows formally that this is the canonical map $V \to \res{G}{G_e} \ind{G_e}{G}V$.
    Hence, both maps are the same. 
\end{proof}

 Let us note the following properties of the quotients of this filtration.
\begin{lemma}
    \label{Filtquot}
    Let $V \in D(L)^\heartsuit$ be a smooth representation and let $e$ be an edge with vertices $v_i:= \partial_ie$. Then for $n \geq 0$ the complex 
    $$
    \alpha (F_{n+1} V/F_n V) 
    $$
    is equivalent to the two-term complex induced by the map
    $$
    \begin{tikzcd}
      \ind{G_e}{G}(F_{n+1}V/F_nV)(e) \arrow[r] & \ind{G_{v_1}}{G}(F_{n+1}V/F_nV)(v_1) \oplus \ind{G_{v_2}}{G}(F_{n+1}V/F_nV)(v_2) 
    \end{tikzcd}
    $$
    This can be explicitly computed by the complex 
    $$
    \begin{tikzcd}
        0 \arrow[r] & \ind{L \cap L^{g_1}}{G}V^{g_1} \oplus \ind{L \cap L^{g_2}}{G}V^{g_2} \arrow[r] & \ind{K_1 \cap L^{g_1}}{G}V^{g_1} \oplus \ind{K_2 \cap L^{g_2}}{G}V^{g_2} \arrow[r] & 0  
    \end{tikzcd}
    $$
    with $g_0, g_1$ such that $g_0e$ and $g_1e$ represent the two $L$-orbits  of edges with distance $n+1$ to $e$. Then $K_i \cap L^{g_i} = L \cap L^{g_i}$ for $i=0,1$ and the differential is the identity.
\end{lemma}
\begin{proof}
    We start by noting that the natural transformation $F_n V \to F_{n+1}V$ is by construction pointwise injective. 
    Hence, the cofibre is again pointwise concentrated in degree $0$ and therefore in the heart.
    Now the first part follows directly from \ref{AssemFormula}.

    For the second part, we need to 
    compute the cofibre $F_{n+1}/F_n$ at $v_1$, $v_2$ and $e$ together with the induced maps. 
    Again, this is computed pointwise.
    For a simplex $x$, the cofibre is given by 
    $\alpha (M_{ E^{n+1}_x}) / \alpha (M_{E^n_x})$ which is 
    equivalent to $\alpha (M_{E^{n+1}_x \setminus E^n_x})$.
    To compute the latter we want to apply the $L$-equivariant version of \ref{AssemFormula}. 
    First, consider the edge $e$.
    Note that for $n\geq 0$ there are two $L$-orbits of edges with distance exactly $n+1$ to $e$. 
    This is because by \cite{BT72}[Section 7.4.] $G$ acts transitively on pairs $(\mathcal{A}', e')$ where $\mathcal{A}'$ is an apartment containing $e'$. In particular, $L$ acts transitively on the set of apartments containing $e$ and we can represent any $L$-orbit by an edge in a fixed apartment $\mathcal{A}$. Since $\mathcal{A}$ is one-dimensional and $G$ acts without inversion, there are exactly two $L$-orbits of edges with distance $n+1$ to $e$.
    Let $e_1$ and $e_2$ be representatives of the two orbits. 
    Then there exist $g_1, g_2 \in G$ such that $g_ie=e_i$ for $i=1,2$
    and the $L$-stabilizer of $e_i$ is $L \cap L^{g_i}$. 
    Lastly by Lemma \ref{Ranres} we can compute $\res{L^{g_i}}{L \cap L^{g_i}}M(e_i) = \res{L^{g_i}}{L \cap L^{g_i}}V^{g_i}$. 
    This yields an equivalence 
    $$\alpha_L(M_{E^{n+1}_e \setminus E_e^n})[-1] \simeq \ind{L\cap L^{g_1}}{L}V^{g_1} \oplus \ind{L \cap L^{g_2}}{L}V^{g_2}.$$

    For $v_i$, there is exactly one $K_i$ orbit of edges with distance $n+1$ and this orbit contains precisely one of the two edges $e_1$ and $e_2$. 
    Assume that it contains $e_i$.
    Then by the same argument as before 
    $$
    \alpha_{K_i}(M_{E^{n+1}_{v_i} \setminus E_{v_i}^n})[-1]= \ind{K_i\cap L^{g_i}}{K_i}V^{g_i}. 
    $$
    For the differential, note that $L \cap L^{g_i} \subseteq K_i \cap L^{g_i}$.
    Furthermore, $e$ is part of the geodesic connecting $v_i$ and $e_i$. 
    Any $g\in G$ that fixes $v_i$ and $e_i$ also fixes the whole geodesic between them, in particular, it fixes $e$ and hence $g \in L$.
    Therefore, $L \cap L^{g_i} =K_i \cap L^{g_i}$ and the induced map from the inclusion $E_e \to E_{v_i}$ is the canonical map $\ind{L \cap L^{g_i}}{L}V^{g_i} \to \ind{K_i \cap L^{g_i}}{K_i}V^{g_i}$.
\end{proof}  

This immediately implies the following corollary.

\begin{corollary}
    For any $V \in D(L)^\heartsuit$ the quotient $F_{n+1}V / F_nV$ is in the kernel of $\alpha$. 
\end{corollary}

This yields another way to show that the map $\MV(X) \to D(G)$ is a left Bousfield localization that we are going to sketch since it uses the same strategy that we want to employ for the relative case.
\begin{remark}
    The previous corollary proves that for every $V\in D(L)^\heartsuit$ the map
    $F_nV \to F_{n+1}V$ is an $\alpha$-equivalence.
    This implies that $\alpha (F_0V) \to \alpha( \colim_n \, F_nV) \simeq \alpha \circ   \res{}{} ( \ind{L}{G}V)$ is also an equivalence where the latter equivalence holds by Lemma \ref{FiltProp}.
    The counit $\alpha \circ  \res{}{} (  \ind{L}{G}V) \to \ind{L}{G}V$ is part of the composition 
\[\begin{tikzcd}[ampersand replacement=\&]
	{ \alpha (F_0V)} \&\& {\alpha \circ \res{}{} \circ \alpha (F_0V)  } \&\& {\alpha (F_0V)}
	\arrow[from=1-1, to=1-3]
	\arrow[from=1-3, to=1-5]
\end{tikzcd}\]
where we use that $\ind{L}{G}V=\alpha(F_0V)$, again by Lemma \ref{FiltProp}.
The first map is an equivalence by the previous discussion and the composition is an equivalence by the triangle identity of the adjunction $\alpha \dashv \res{}{}$.
Hence, the counit applied to $\alpha(F_0V)$ is also an equivalence.
Applying this to $V= \ind{C}{L}1$ for every compact open subgroup $C \subseteq L$ shows that 
the counit is an equivalence for a set of compact generators of $D(G)$.
Since both sides are compatible with colimits, this implies that
it is an equivalence in general.
Therefore, $\MV(X) \to D(G)$ is a left Bousfield localization.
\end{remark}

Next, we want to consider a relative version of the filtration.
For this, let $\chi: T^0 \to \C^\times$ be a smooth character with associated type $\rho_\chi$.
This induces filtration by applying $\pr{\chi} : \MV(X) \to \MV(X, \rho_\chi)$ to $F_nV$.  
\begin{definition}
    Let $\chi: T^0 \to \C^\times $ be a smooth character with associated type $\rho_\chi$. Let $\mathcal{D}(\chi)$ be the corresponding admissible collection with projection $\pr{\chi}: \MV(X) \to \MV(X, \rho_\chi)$.
    Then, for any $V \in D(L, \rho_\chi)$, we define
    $$
    F_n^\chi V:=\pr{\chi}F_n V. 
    $$
\end{definition}
Similar to the absolute case before, the relative filtration converges to $\pr{\chi} \circ \res{}{} \circ \ind{L}{G}V$.

\begin{lemma}
    \label{relFilt}
    Let $V \in D(L, \rho_\chi)^\heartsuit$. Then $F_0^\chi V = \Lan^{\int_G \simp{X}}_{BL} V$ and $\colim_n F^\chi_nV = \res{}{} \circ \alpha (F_0^\chi V)$.
    Furthermore, the canonical map $F_0^\chi V \to \colim _n F_n^\chi V$ is the unit of the adjunction $\alpha \dashv \pr{\chi} \circ \res{}{}$.
\end{lemma}
\begin{proof}
    By definition, we have $F_0^\chi V \simeq \pr{\chi} \Lan_{BL}^{\int_G \simp{X}}V$.
    But Lemma \ref{relmaps} shows that $\Lan_{BL}^{\int_G \simp{X}}V \in \MV(X, \rho_\chi)$ and therefore $\pr{\chi}\Lan_{BL}^{\int_G \simp{X}}V = \Lan_{BL}^{\int_G \simp{X}}V$.
    Furthermore, since $\pr{\chi}$ commutes with colimits by Lemma \ref{MVchicomp}, it follows immediately that $\colim_n F^\chi_nV = \pr{\chi} \res{}{} \circ \alpha (F^\chi_0V)$.

    Lastly, since $\MV(X, \rho_\chi)$ is a colocalization of $\MV(X)$ and $F_0V \in \MV(X, \rho_\chi)$, it follows that we obtain the unit of the adjunction $\alpha \dashv \pr{\chi} \circ \res{}{}$ by applying $\pr{\chi}$ to the unit of the adjunction $\alpha \dashv \res{}{}$ and restricting it to $\MV(X, \rho_\chi)$.   \end{proof}

We want to employ the same strategy as before to show that the counit of the adjunction between $\alpha$ and $\pr{\chi} \circ \res{}{}$ is an equivalence. 
The main task for this strategy is to prove that the quotients of the relative filtration are also contractible after $\alpha$.
This will be the goal of the remainder of this section.
By Lemma \ref{Filtquot} together with the description of $\pr{\chi}$ from Lemma \ref{MVLoc}, this is equivalent to the following proposition.
\begin{proposition}
    \label{projind}
        Let $K=K_0$ or $K=K_1$ and let $g \in G$ such that $K \cap L^g = L \cap L^g$. Then the map 
        $$ 
        \begin{tikzcd}
            \ind{L}{K} \pr{\chi} \ind{L\cap L^g}{L}V^g \arrow[r] & \pr{\chi} \ind{K \cap L^g}{K}V^g
        \end{tikzcd}
        $$
        with $V = \ind{J}{L}\rho_\chi$ is an equivalence.
\end{proposition}
\begin{remark}
    \label{mapinj}
    Note that the injectivity follows immediately. Since the admissible collection is closed under induction, the map is part of a commutative square 
\[\begin{tikzcd}
	{\ind{L}{K} \pr{\chi} \ind{L \cap L^g}{L} V^g} && {\pr{\chi} \ind{K\cap L^g}{K}V^g} \\
	\\
	{\ind{L}{K} \ind{L \cap L^g}{L}V^g} && {\ind{K \cap L^g}{K}V^g}
	\arrow[from=1-3, to=3-3]
	\arrow[from=1-1, to=3-1]
	\arrow[from=1-1, to=1-3]
	\arrow["{=}", from=3-1, to=3-3]
\end{tikzcd}\]
where both vertical maps are injective.
\end{remark}
Before we prove the proposition, we need to establish some auxiliary results.

\begin{lemma}
    \label{diagram}
    Let $W=N(T)/T^0$ be the extended Weyl group with respect to $T$, let $L$ and $J$ be as above and let $K=K_0$ or $K =K_1$.
    Then the following diagram commutes.
\[\begin{tikzcd}
	N(T)/T^0 & {J \backslash N(T) /J} & {J\backslash G /J} \\
	N(T)/T^0 & {L \backslash N(T) /L} & {L \backslash G /L} \\
	{K\backslash N(T)/T^0}& {K \backslash N(T) /L} & {K \backslash G /L}
	\arrow[from=1-1, to=1-2]
	\arrow[from=1-2, to=1-3]
	\arrow["\pr{J}^L", from=1-3, to=2-3]
	\arrow["\pr{L}^K" ,from=2-3, to=3-3]
	\arrow[from=3-2, to=3-3]
	\arrow["(5)",from=3-1, to=3-2]
	\arrow["(2)", from=2-1, to=2-2]
	\arrow["(3)", from=2-2, to=2-3]
	\arrow["(1)", from=1-2, to=2-2]
	\arrow["\id"{description}, from=1-1, to=2-1]
	\arrow[ from=2-1, to=3-1]
	\arrow["(4)", from=2-2, to=3-2]
\end{tikzcd}\]
    The arrows $(1)$, $(2)$, $(3)$ and $(5)$ are bijective, $(4)$ is a $2:1$ surjection and $K\backslash N(T)/T^0$ is the quotient of the $\Z/2$-action by the element $\omega$ that corresponds to the reflection of the standard apartment at the vertex of $e$ stabilized by $K$ on $N(T)/T^0$.  
\end{lemma}
\begin{proof}
    Consider the composition $N(T)/T^0 \to L \backslash N(T) /L \to L \backslash G / L$ first. 
    By definition $L$ is the stabilizer of the edge $e$ and $G$ acts transitive on edges, hence the left $G$-set $G/L$ can be identified with the set of edges of $X$.
    Therefore, $L\backslash G /L$ is equivalent to the set of $L$-orbits of edges of $X$.
    Let $\mathcal{A}$ be the standard apartment associated to $T$. In particular, this contains the edge $e$. Let $e'$ be any other edge.
    Any two edges lie in a common apartment, hence there is an apartment $\mathcal{A}'$ that contains both $e$ and $e'$.
    By \cite{BT72}[Section 7.4.], the group $G$ acts transitively on the set of pairs $(\mathcal{A}'', e')$ where $\mathcal{A}''$ is an apartment containing the edge $e'$. 
    Note that this is still true after subdividing $X$ if the $G$-action inverts edges.
    In particular there exists $g \in G$ such that $g\mathcal{A}'=\mathcal{A}$ and $g$ fixes $e$, hence $g \in L$. 
    Therefore, we can assume that every orbit is represented by an edge in $\mathcal{A}$.
    , The group $G$ acts by isometries and there are exactly two edges in $\mathcal{A}$ with the same, positive distance to $e$.
    These can not be mapped onto each other while fixing $e$ as this would invert the orientation. 
    In particular, they are not in the same $L$ orbit.
    By Lemma \ref{Weyl}, the extended Weyl group acts simply transitive on the edges of $\mathcal{A}$ and the map $W=N(T)/T^0 \to L \backslash G / L$ sends $\omega$ to the edge $\omega e$.  
    This shows that $N(T)/T^0 \to L \backslash G /L$ is bijective, hence $N(T)/T^0 \to L \backslash W / L$ is also bijective. 
    Now, by the two-out-of-three property, it follows that $L \backslash N(T)/  L \to L \backslash G / L$ is also bijective.
    By the commutativity of the diagram, the composition $N(T)/T^0 \to J \backslash N(T) / J \to L \backslash N(T) / L$ is also bijective. 
    Again this implies that $N(T)/N^0 \to J \backslash W / J$ is injective and therefore bijective, hence $J \backslash N(T) / J \to L \backslash N(T) / L$ has to be bijective. 

    The only thing that remains to be shown is that $(5)$ is bijective and $L \backslash N(T) /L \to K \backslash N(T) /L$ is $2:1$. 
    Note that because $N(T)/T^0 \to L \backslash N(T) / L$ is bijective,
    the map $N(T)/T^0 \to N(T)/L$ has to be injective. Taking $K$-orbits on both sides we see that $K \backslash N(T)/T^0 \to K \backslash N(T)/L$ is also injective. It is also clearly surjective, hence bijective. 
    The last claim that needs to be checked is that $L \backslash N(T)/L \to K\backslash N(T)/L$ is $2:1$. 
    From what we have shown so far, this is equivalent to showing that $N(T)/T^0 \to K \backslash N(T)/T^0$ is $2:1$.
    When we identify $W =N(T)/T^0$ with the edges of the standard apartment, this map corresponds to the quotient map of the action by of $N(T) \cap K$.
    Any element in $N(T) \cap K$ has to fix the vertex $v\subseteq e$ that is stabilized by $K$. 
    Furthermore, the group $W$ acts simply transitive on $\mathcal{A}$, and therefore $N(T)$ acts transitively with kernel $T^0$, hence the only element (up to multiplication with $T^0$) in the intersection acting non-trivially is the reflection $\omega$ at the vertex $v$.
    This precisely identifies the two edges with equal distance, hence the map in question is $2:1$. 
\end{proof}
\begin{corollary}
    If $W_\chi \neq W$, then $\pr{L}^K\vert _{N(T)_\chi}:L \backslash N(T)_ \chi / L \to K \backslash N(T)_\chi/ L $ is bijective.
\end{corollary}

\begin{proof}
    Let $\omega$ be the reflection at the vertex with stabilizer $K$. Then, by Lemma \ref{diagram}, the map $L \backslash N(T) / L \to K \backslash N(T) / L$ can be identified with the quotient map $W \to \omega \backslash W$.
    If $W_\chi \neq W$, then $W$ decomposes into $W_\chi \sqcup \omega W_\chi$ and the claim follows.
\end{proof}
\begin{lemma}
    $\ind{J}{L}\rho_{\chi^\omega} $ is irreducible.
\end{lemma}
\begin{proof}
    Since $L$ is compact every $L$-representation decomposes into a direct sum of irreducibles.
    Hence, it is enough to show that $\End(\ind{J}{L}\rho_{ \chi^\omega})$ is one-dimensional. 
    By Macky's formula 
    $$ \End_L(\ind{J}{L}\rho_{\chi^\omega}) \cong \Hom_J(\rho_{\chi^\omega}, \ind{J}{L}\rho_{\chi^\omega}) \cong \bigoplus_{[g] \in J \backslash L / J}
    \Hom_{J }(\rho_{\chi^\omega}, \ind{J \cap J^g}{J} (\rho_{\chi^\omega})^g). $$
    By the induction-restriction adjunction, we obtain an equivalence
\[\begin{tikzcd}
	{\Hom_{J }(\rho_{\chi^\omega} , \ind{J \cap J^g}{J}( \rho_{\chi^\omega}) ^g)} && {\Hom_{J \cap J^g}(\rho_{\chi^\omega} ,(\rho_{ \chi^\omega})^g)}
	\arrow["\simeq", tail reversed, from=1-1, to=1-3]
\end{tikzcd}\]
    and since $\rho_{\chi^\omega}$ is a character the dimension of is either one or zero.
    It is one precisely if $g$ is an intertwiner of $\rho_{\chi^\omega}$. 
    But by Lemma \ref{RocheIntertwiner}, $g \in J \backslash L / J$ intertwines $\rho_{\chi^\omega}$ if and only if $g \in J \backslash N(T)_\chi \cap L /J$. Since $J\backslash N(T)_\chi \cap L / J \cong L \backslash N(T)_\chi \cap L / L$ 
    by \ref{diagram} there is exactly one such intertwiner. 
\end{proof}
\begin{lemma}
    \label{numbirred}
    For $K=K_i$ for $i=0,1$ the number of irreducible components of $\ind{J}{K} \rho_{\chi^\omega}$ is equal to $  \vert (\pr{J}^K)^{-1}([1]) \cap J \backslash N(T)_\chi /J \vert $. More precisely, we have 
    one if $W_\chi \neq W$, and two distinct irreducibles if $W_\chi = W$.  
\end{lemma}
\begin{proof}
    Again, we compute the dimension of $\End_K(\ind{J}{K} \rho_{\chi^\omega})$ using Macky's formula. 
    $$ \End_K(\ind{J}{K} \rho_{\chi^\omega}) \cong \bigoplus_{[g] \in J \backslash K / J} \Hom_{J \cap J^g}(\rho_{\chi^\omega}, (\rho_{\chi^\omega})^g).$$
    The latter $\Hom$-sets are either zero or one-dimensional. By Lemma \ref{RocheIntertwiner}, they are one-dimensional if and only if $[g]$ intertwines $\rho_{\chi^\omega}$, i.e. if and only if $[g] \in (\pr{J}^K)^{-1}([1]) \cap J \backslash N(T)_\chi /J$. 
\end{proof}
\begin{lemma}
    \label{sameind}
    For any $\omega \in W$ there is an isomorphism $\ind{J}{K}\rho_\chi \simeq \ind{J}{K}\rho_{\chi^\omega}$. 
\end{lemma}
\begin{proof}
   If $\omega \in W_\chi$ then by definition $\chi = \chi^\omega$ and there is nothing to show. 
   Therefore, let us assume $\omega \notin W_\chi$, in particular $W_\chi \subsetneq W$.
   Furthermore, by Lemma \ref{Weyl} we can assume that $\omega$ is the reflection along the vertex that has stabilizer $K$, since the conjugation of $\chi$ by $\omega$ only depends on the image of $\omega$ in the finite Weyl group.
   Then by \ref{numbirred} both $\ind{J}{K} \rho_\chi$ and $\ind{J}{K}\rho_{\chi^\omega}$ are irreducible and hence we only need to show that there is a non-zero map between them. 
   Again, we calculate
   \begin{equation*}
    \Hom_K(\ind{J}{K}\rho_\chi , \ind{J}{K}\rho_{\chi^\omega}) \cong \Hom_J(\rho_\chi , \ind{J}{K}\rho_{\chi^\omega}) \cong \bigoplus_{[g] \in J \backslash K /J}\Hom_J(\rho_\chi , \ind{J\cap J^g}{J} {\rho_{\chi^\omega}}^g).
   \end{equation*}
   Simplifying further we obtain $\Hom_J(\rho_{\chi} , \ind{J\cap J^g}{J} {\rho_{\chi^\omega}}^g)\cong \Hom_{J\cap J^g}(\rho_\chi , {\rho_{\chi^\omega}}^g)$.
    The latter space is non-zero precisely if $g$ is an intertwiner of $\rho_\chi$ and $\rho_{\chi^\omega}$. By Lemma \ref{diagram} the set $(\pr{J}^{K})^{-1}([1]) \cap J \backslash \omega N(T)_\chi /J$ is non-empty, hence there is an intertwiner in $J\backslash K /J$ by Lemma \ref{chiinter} which induces a non-trivial map.  
\end{proof}

\begin{proof}[Proof of \ref{projind}]
    Since induction maps $D(L, \rho_\chi)$ to $D(K, \rho_\chi)$, it follows immediately that the map is injective (see remark \ref{mapinj}).
    For the surjectivity, first note that all representations appearing are finitely generated, and it suffices to count the number of irreducible representations on both sides since we already know that the map is injective. 
    We start by computing this number for the right-hand side.
    Since $\ind{J}{K}\rho_\chi = \ind{J}{K} \rho_{\chi^\omega}$ (Lemma \ref{sameind}) decomposes into one or two distinct irreducibles by Lemma \ref{numbirred} and every irreducible representation contained in $D(K, \rho_\chi)$ appears as a summand in $\ind{J}{K}\rho_\chi$, the desired number is exactly the dimension of the $\C$-vector space $\Hom_K(\ind{J}{K}\rho_\chi, \ind{K\cap L^g}{K}V^g)$.
    Using that $\pr{\chi}$ is right adjoint to the inclusion $D(K, \rho_\chi) \to D(K)$ we obtain 
         $$
         \Hom_K(\ind{J}{K}\rho_\chi , \pr{\chi}\ind{K \cap L^g}{K}V^g) \simeq \Hom_J(\rho_\chi, \ind{K \cap L^g}{K}V^g).
        $$
    With Macky's formula, we can decompose $\res{K}{J}\ind{K\cap L^g}{K} V^g$ as follows. The representation $\ind{K \cap L^g}{K}V^g$ is a direct summand of 
    $$\bigoplus_{[h] \in K \backslash G / L}\ind{K \cap L^h}{K}V^h =\res{G}{K} \ind{L}{G} \ind{J}{L} \rho_\chi.$$ 
    Restricting further to $J$ and applying (the proof of) Macky's formula again we also obtain a finer decomposition
    $$ 
    \res{G}{J}\ind{J}{G}\rho_\chi = \bigoplus_{[h] \in J \backslash G /J}\ind{J\cap J^h}{J}\rho_\chi^h
    $$
    such that  $\ind{K \cap L^g}{K}V^g\subseteq \res{G}{K}\ind{J}{G}\rho_\chi$ decomposes further into the summands corresponding to the preimage of $[g]$ under $\pr{J}^{K}:J \backslash G / J \to K \backslash G / L$. 
    With that we compute 
    $$
    \Hom_J(\rho_{\chi}, \ind{K \cap I^g}{K}V^g) \cong \bigoplus_{[h] \in (\pr{J}^{K})^{-1}([g])} \Hom_{J \cap J^h}(\rho_\chi, {\rho_{\chi}^h}).
    $$
    By Lemma \ref{RocheIntertwiner}, $h$ is an intertwiner of $\rho_\chi$ if and only if $h\in J \backslash N(T)_\chi /J$. Thus, the dimension is equal to 
    $ \vert J \backslash N(T)_\chi / J\cap (\pr{J}^{K})^{-1}([g]) \vert$. 

    For the left-hand side we start by considering $\Hom_J(\rho_{\chi^\omega}, \ind{L\cap L^g}{L}V^g)$ for any $\omega \in W$. 
    By an analogous calculation as above, we obtain
    $$\Hom_J(\rho_{\chi^\omega},  \ind{L \cap L^g}{L}V^g) \cong \bigoplus_{[h] \in (\pr{J}^{L})^{-1}([g])} \Hom_{J \cap J^h}(\rho_{\chi^\omega}, (\rho_\chi)^h).$$
    By Lemma \ref{chiinter} this has dimension at least $\vert (J \backslash \omega W_\chi / J\cap (\pr{J}^{L})^{-1}([g])) \vert$. 

    Now, we distinguish between the two cases $W_\chi =W$ and $W_\chi \neq W$.
    In the first case, we computed that $\pr{\chi}\ind{L\cap L^g}{L}V^g$ contains at least one copy of $\rho_\chi$, which by Lemma  \ref{numbirred} splits into two irreducibles after induction to $K$. 
    On the other hand by the above computation, $\pr{\chi}\ind{K \cap L^g}{K}V^g$ has also two irreducible components.  
    In the case $W_\chi \neq W$, we have $W = W_\chi \cup \omega W_\chi$ as a set, where $\omega$ represents any non-trivial element in the finite Weyl group.
    Hence, either $J \backslash N(T)_\chi /J \cap (\pr{J}^{L})^{-1}([g])$ or 
    $J \backslash \omega N(T)_\chi /J \cap (\pr{J}^{L})^{-1}([g])$ is non-empty and $\pr{\chi}\ind{L \cap L^g}{L}V^g$ contains at least one irreducible.  
    This representation stays irreducible by \ref{numbirred}, and we have again the same number of irreducibles on both sides.
    \end{proof}
    As a consequence of the above, we can finally prove that $\alpha$ restricts to a left Bousfield localization on $\MV(X, \rho_\chi)$. 

    \begin{theorem}
        \label{BousfieldType}
        Let $\chi: T^0 \to \C^\times$ be a rmooth character with associated type $\rho_\chi$.
        Then the functor $\alpha : \MV(X, \rho_\chi) \to D(G, \rho_\chi)$ is a left Bousfield localization.
    \end{theorem}
    \begin{proof}
        We already know that $\alpha$ has a right adjoint given by $\pr{\chi} \circ \res{}{}$, hence we only need to show that the counit is an equivalence.
        Since both sides are compatible with colimits, it suffices to show this for the compact generator $\ind{J}{G} \rho_\chi$.
        By transitivity of left Kan extensions, $\ind{J}{G}\rho_\chi$ can be obtained as $\alpha (\Lan_{BL}^{\int_G \simp{X}}V)$ with $V= \ind{J}{L}\rho_\chi$. 
        Let $F_n^\chi V$ be the filtration from Lemma \ref{relFilt}. 
        Now we plug in the definitions of $\pr{\chi}$ and $F_n^{\chi}$ into Lemma \ref{Filtquot}.
        This shows that, to understand $\alpha(F^\chi_{n+1}V/F^\chi_nV)$, we have to compute the complex associated to  
\[\begin{tikzcd}
	{\ind{L}{G}\pr{\chi}\ind{L \cap L^{g_0}}{L}V^{g_0} \oplus \ind{L}{G} \pr{\chi}\ind{L \cap L^{g_1}}{L}V^{g_1}} \\
	{\ind{K_1}{G}\pr{\chi} \ind{K_0 \cap L^{g_0}}{K_1}V^{g_0} \oplus \ind{K_1}{G}\pr{\chi} \ind{K_1 \cap L^{g_1}}{K_1}V^{g_1} }
	\arrow[from=1-1, to=2-1]
\end{tikzcd}\]
        with $g_0,g_1$ as in Lemma \ref{Filtquot}.
        But this complex is contractible by Proposition \ref{projind}.
        Therefore, all maps $F_n^\chi V \to F_{n+1}^\chi V$ are equivalences, hence also the canonical map $F_0^\chi V \to \colim_n F_n ^\chi V = \pr{\chi } \circ \res{}{} \circ \alpha (F_0^\chi V)$.
        Again, we have the triangle identity associated to the adjunction $\alpha \dashv \pr{\chi} \circ \res{}{}$ 
\[\begin{tikzcd}
	{\alpha (F_0^\chi V)} && {\alpha \circ \pr{\chi} \circ \res{}{} \circ  \alpha (F_0^\chi V)} && {\alpha (F_0^\chi V).}
	\arrow[from=1-1, to=1-3]
	\arrow[from=1-3, to=1-5]
\end{tikzcd}\]
    The composition is always an equivalence.
    Furthermore, by the previous argument, the first map is also an equivalence.
    Therefore, the right-hand map is an equivalence as well.  
    But this is precisely the counit evaluated at $\alpha(F_0^\chi V) \simeq \ind{J}{G}\rho_\chi$, hence the claim follows.    
    \end{proof}
    Now we are in the position to prove the main result of this section.
    For any principal series type $\rho_\chi$, we can compute the $K$-theory spectrum of the corresponding Bernstein block as follows. 
    \begin{theorem}
        \label{KBernstein} 
        Let $G$ be a connected split reductive algebraic group of rank one and let $\rho_\chi$ be a principal series type associated to a smooth character $\chi: T^0 \to \C^\times$.
        \begin{enumerate}
            \item If $X_0$ has one $G$-orbit, then 
            \[\begin{tikzcd}
	            {K(G_e, \rho_\chi)} && {K(G_{\partial_0e}, \rho_\chi)} && {K(G, \rho_\chi)}
	            \arrow["{\ind{}{}}", from=1-3, to=1-5]
	            \arrow["{\ind{}{}(\partial_0) - \ind{}{}(\partial_1)^g}", from=1-1, to=1-3]
            \end{tikzcd}\]
            with maps induced by induction and $g \in G$ such that $g\partial_1e = \partial_0e$, is a cofibre sequence.
            \item If $X_0$ has two $G$-orbits, then 
            \[\begin{tikzcd}
            	{K(G_e, \rho_\chi)} && {K(G_{\partial_0e}, \rho_\chi)} \\
            	\\
	            {K(G_{\partial_1e}, \rho_\chi)} && {K(G, \rho_\chi)}
            	\arrow["\ind{}{}",from=1-1, to=1-3]
            	\arrow["\ind{}{}",from=1-1, to=3-1]
            	\arrow["\ind{}{}",from=3-1, to=3-3]
            	\arrow["\ind{}{}",from=1-3, to=3-3]
            \end{tikzcd}\]
            with maps induced by induction, is a pushout square.
        \end{enumerate}
        In particular, $K(G, \rho_\chi)$ is connective.
    \end{theorem}
    \begin{proof}
       By \ref{BousfieldType} together with Corollary \ref{D0MV0} and Proposition \ref{Kercompg} we obtain a Verdier sequence
\[\begin{tikzcd}
	{\MV_0(X, \rho_\chi)} && {\MV(X, \rho_\chi)} && {D(G, \rho_\chi)}
	\arrow[from=1-1, to=1-3]
	\arrow[from=1-3, to=1-5]
\end{tikzcd}\]
    where $\MV_0(X, \rho_\chi)$ is by definition the kernel of the localization. 
    This induces a fibre sequence in $K$-theory.
    By Proposition \ref{KMV} we have 
    $$ 
    K(\MV (X, \rho_\chi)^\omega) \simeq \bigoplus_{[x] \in X/G} K( G_x, \rho_\chi). 
    $$
    Furthermore, by \ref{RelKer} we can also compute the $K$-theory of the kernel by
    $$
    K(\MV _0(X, \rho_\chi)^\omega) \simeq K(G_e, \rho_\chi) \bigoplus K(G_e, \rho_\chi)
    $$
    where $e$ is a fixed edge of $X$.
    By the same argument as in Theorem \ref{Kseq} this yields the desired colimit formula. 
    Additionally, similar to the absolute case, Lemma \ref{tcomprel} applied to $*$ with trivial $G_x$ action, the category $D(G_x, \rho_\chi)^\omega$ admits a bounded $t$-structure with Noetherian heart. 
    Then \cite[Theorem 1.2]{AGH19} implies that $K(G_x, \rho_\chi)$ is connective and therefore $K(G, \rho_\chi)$ is also connective since it is a pushout of connective spectra. 
    \end{proof}

\begin{remark}
In the proof of Theorem \ref{KBernstein} we used that the Bruhat--Tits building is a tree to compute the kernel and to show that $\alpha$ restricts to a left Bousfield localization on the full subcategory $\MV(X, \rho_\chi)$ that we associated to every principal series type $\rho_\chi$.
However, the definition of the categories $D(G_x, \rho_\chi)$ is independent of the dimension and can always be used to construct a relative assembly map for every principal series Bernstein block, and it seems to be an interesting question to ask whether this is always an equivalence.
\end{remark}
\addcontentsline{toc}{section}{References}
\bibliography{Ref.bib}
\bibliographystyle{alpha} 
\end{document}